\documentclass[11pt]{amsart}

\usepackage[foot]{amsaddr}
\usepackage[a4paper,margin=2.65cm]{geometry}
\usepackage[english]{babel}
\usepackage[utf8]{inputenc}
\usepackage{csquotes}
\usepackage[dvipsnames,svgnames,table,x11names]{xcolor}
\usepackage{amsmath,amssymb,amsthm,mathtools,mathrsfs}
\mathtoolsset{showonlyrefs}
\NewEnvironmentCopy{standardbmatrix}{bmatrix}
\usepackage[renew-dots,renew-matrix]{nicematrix}
\usepackage[T1]{fontenc}
\usepackage{textcomp}
\usepackage{newtxtext}
\usepackage{newtxmath}
\usepackage[bb=boondox,cal=boondoxo,scr=boondoxo]{mathalfa}
\usepackage{microtype}
\usepackage{float}
\usepackage{tikz}

\allowdisplaybreaks
\usepackage[
backend=biber,
style=numeric,
giveninits=false,
maxbibnames=99,
minbibnames=99,
maxcitenames=99,
mincitenames=99,
doi=true,
url=false,
isbn=false
]{biblatex}

\DeclareFieldFormat[article]{title}{#1}
\DeclareFieldFormat{journaltitle}{\mkbibemph{#1}}
\DeclareFieldFormat[article]{volume}{\mkbibbold{#1}}
\DeclareFieldFormat{pages}{#1}
\DeclareFieldFormat{eid}{Article number\space #1}
\DeclareFieldFormat{doi}{\textsc{doi}:\space
	\href{https://doi.org/#1}{\nolinkurl{#1}}}

\DeclareBibliographyDriver{article}{%
	\usebibmacro{bibindex}%
	\usebibmacro{begentry}%
	\printnames{author}%
	\setunit{\addperiod\space}%
	\printfield{title}%
	\setunit{\addperiod\space}%
	\printfield{journaltitle}%
	\setunit{\addspace}%
	\printfield{volume}%
	\setunit{\addspace}%
	\printtext[parens]{\printdate}%
	\setunit{\addspace}%
	\iffieldundef{pages}
		{\printfield{eid}}
		{\printfield{pages}}%
	\newunit\newblock
	\iffieldundef{addendum}
		{}
		{\printfield{addendum}%
		 \newunit\newblock
		 \usebibmacro{eprint}}%
	\newunit\newblock
	\printfield{doi}%
	\addperiod%
	\usebibmacro{finentry}%
}
\renewbibmacro*{in:}{}

\usepackage[hyperfootnotes]{hyperref}
\usepackage{orcidlink}
\hypersetup{
	colorlinks=true,
	linkcolor=NavyBlue, 
	urlcolor=RoyalPurple,
	citecolor=OliveGreen,
	pdftitle={Positive Bidiagonal Factorizations for Banded Markov Processes},
	pdfauthor={Manuel Ma\~nas},
	pdfsubject={Positive bidiagonal factorizations and spectral theory of banded Markov processes},
	pdfkeywords={positive bidiagonal factorization; total positivity; banded Markov processes; mixed-type multiple orthogonality; Karlin-McGregor representation; matrix continued fractions; Pineiro polynomials}
}

\theoremstyle{plain}
\newtheorem{theorem}{Theorem}[section]
\newtheorem{proposition}[theorem]{Proposition}
\newtheorem{lemma}[theorem]{Lemma}
\newtheorem{corollary}[theorem]{Corollary}

\newtheoremstyle{definitionstyle}
{6pt}
{12pt}
{\normalfont}
{}
{\bfseries}
{.}
{0.5em}
{}

\theoremstyle{definitionstyle}
\newtheorem{definition}[theorem]{Definition}
\newtheorem{example}[theorem]{Example}

\newtheoremstyle{remarkstyle}
{6pt}
{12pt}
{\normalfont}
{}
{\itshape}
{.}
{0.5em}
{}

\theoremstyle{remarkstyle}
\newtheorem{remark}[theorem]{Remark}

\newcommand{\Nzero}{\mathbb{N}_0}

\newcommand{\one}{\mathbf{1}}
\newcommand{\ind}{\mathbf{1}}
\newcommand{\diag}{\mathrm{diag}}
\newcommand{\e}{\mathrm{e}}
\newcommand{\dd}{\,\mathrm{d}}

\newcommand{\pFq}[5]{{}_{#1}F_{#2}\!\left(\begin{matrix}#3\\[2pt]#4\end{matrix}\,;\,#5\right)}

\title[Banded Markov Processes and Bidiagonal Factorization]{Positive Bidiagonal Factorizations for \\Banded Markov Processes}

\author[M. Ma\~nas]{Manuel Ma\~nas\,\orcidlink{0000-0003-3764-5737}}
\address{Department of Theoretical Physics, Faculty of Physical Sciences,
	Complutense University of Madrid, 28040 Madrid, Spain}
\email{manuel.manas@ucm.es}



\date{September 1, 2026}

\begin{document}
	
\begin{abstract}
	An ordered positive bidiagonal factorization (PBF) of a semi-infinite band
	matrix simultaneously encodes total positivity, mixed-type spectral data,
	and elementary Markov transitions.  For arbitrary finite bandwidth, we
	develop Karlin--McGregor formulas for transition probabilities, Green
	kernels, resolvents, potentials, and first-passage transforms without
	reversibility or block symmetrizability.  Rational stochastic PBFs are
	characterized by ordered finite-urn experiments; cyclic reorderings yield
	Darboux intertwinings, while a factor-resolved construction and
	a matrix continued fraction determine scalar and level-return laws.
	Grouping states gives finite-phase quasi-birth-and-death processes, and a
	rank obstruction detects failure of block symmetrization.

	For bounded continuous-time generators, uniformization preserves the
	spectral data.  For unbounded exit rates, we prove that a conservative
	generator whose shifted leading truncations all admit scalar PBFs must be
	tridiagonal.  State-dependent uniformization \(Q=V(T-I)\) instead preserves
	arbitrary bandwidth and separates the embedded chain, which controls
	recurrence and hitting, from holding rates and invariant measures.

	The theory is explicit for the mixed Pi\~neiro system.  We establish strict
	total positivity of its step-line moment matrix, determine the exact ordered
	PBF region and larger regions in which every entry allowed by the bandwidth
	is positive, and characterize boundedness by the cyclic spacings of the beta
	parameters.  Positive harmonic normalization produces a bounded stochastic
	band matrix even when the original recurrence is unbounded and the diagonal
	conjugation is not a bounded similarity.  A rational \((3,2)\) example
	exhibits the factorization, spectral formulas, urns,
	quasi-birth-and-death model, and state-dependent process.
\end{abstract}
	
	\maketitle

\setcounter{tocdepth}{2}
\tableofcontents
	
\section{Introduction}

A positive bidiagonal factorization contains more information than the
entries of its product.  For a semi-infinite matrix with \(p\)
subdiagonals and \(q\) superdiagonals, the prescribed order considered here
is
\begin{equation}
\label{eq:introduction-PBF}
T=L_1\cdots L_pU_q\cdots U_1,
\end{equation}
where the lower and upper factors are bidiagonal and their nontrivial
entries are positive.  The product of the totally nonnegative bidiagonal
factors is totally nonnegative; under the strict band hypotheses used below
it belongs to the banded totally positive class of \cite{BFM4}.  The
individual factors retain, in addition, their order, their Christoffel
connections, and their one-step coefficients, all of which are concealed
after multiplication.  Classical background on total positivity is given
in \cite{Gantmacher,Karlin,Fallat,Pinkus}.

This factor-level information has a direct meaning in multiple
orthogonality.  The mixed-type Favard theorem reconstructs from the banded
recurrence and its initial conditions two left--right families of multiple
orthogonal polynomials and an entrywise positive \(q\times p\) matrix of
measures.  Its shifted-truncation form includes unbounded recurrence
matrices \cite{BFM1,BFM4,BFM_arXiv_2026}.  Thus the PBF supplies the
biorthogonal spectral data needed when scalar orthogonality and
self-adjointness are unavailable.  General background on multiple
orthogonality may be found in
\cite{afm,nikishin_sorokin,Ismail,andrei_walter}.

The same structure also has a probabilistic realization.  When a positive
harmonic vector is available, diagonal Doob normalization changes the PBF
into a product of stochastic bidiagonal matrices.  Each factor is an
elementary death-or-stay or birth-or-stay transition, while the product is a
finite-band Markov kernel.  The stochastic Favard theory of \cite{BFM6}
then gives a non-self-adjoint Karlin--McGregor representation.  In the
classical birth--death case this reduces to a tridiagonal operator, one
family of orthogonal polynomials, and one positive measure
\cite{KarlinMcGregor,KM1957}; for several subdiagonals and
superdiagonals, two mixed-type families are required.

Other established methods do not cover all the situations considered here.
Grouping scalar states into levels gives a quasi-birth-and-death process
(QBD), but classical QBD spectral methods based on positive-definite
matrix-valued orthogonality require a block symmetrization that need not
exist
\cite{DetteEtAlQBD2006,GrunbaumDelaIglesiaQBD2008,DelaIglesiaQBDInvariant2011}.
Related stochastic factorizations and Darboux transformations appear in
\cite{GrunbaumDelaIglesiaQBDDarboux2019,FernandezDelaIglesiaQBD2021}; see
also \cite{ZygmuntNonSymmetric2016}.  In continuous time, global
uniformization \cite{bremaud} applies when the exit rates are bounded.  For
unbounded band matrices, the shifted-truncation Favard theorem of
\cite{BFM_arXiv_2026} controls the spectral recurrence but does not impose
the sign and row-sum conditions of a conservative Markov generator.

The mixed Pi\~neiro system provides an explicit model beyond the tridiagonal
case.  Its step-line moment matrix is Cauchy, and its closed formulas make it
possible to classify positivity, boundedness, and stochastic normalization
for arbitrary finite bandwidth.

The central point of this paper is that an ordered PBF is more than a
certificate of total nonnegativity.  After diagonal normalization, the same
factors determine the mixed spectral representation, a sequence of
elementary Markov transitions, and the continued-fraction coefficients that
govern return laws.  This common structure connects scalar banded kernels,
level--phase QBD models, and continuous-time processes without assuming
reversibility or block symmetrizability.  The diagonal normalization is not
merely cosmetic: the Pi\~neiro PBF region contains recurrence matrices that
are unbounded in the original step-line normalization, although their
Doob-normalized band matrices are stochastic and bounded.

\subsection{Main results}

The main contributions are organized as follows.

\begin{enumerate}
	\item \emph{Factor-level probability and spectral theory
	(Sections~2 and~3).}
	Theorem~\ref{thm:finite-urn-correspondence} gives an exact equivalence
	between rational stochastic PBFs and ordered sequences of finite urn
	experiments; real factors have the corresponding Bernoulli
	interpretation.  Corollary~\ref{cor:urn-protocol-spectral} connects the
	same sequence with the mixed-type Karlin--McGregor representation.
	Cyclic reorderings become Darboux intertwinings.  Resolving the transition
	into its ordered bidiagonal stages gives a block continued fraction whose
	boundary entry, combined with the renewal identity, yields the complete
	scalar first-return law
	\eqref{eq:factor-resolved-first-return-law}.

	\item \emph{Continuous time (Sections~4 and~5).}
	Theorem~\ref{thm:CT-bounded} gives the semigroup and resolvent for bounded
	exit rates.  For unbounded rates,
	Proposition~\ref{pro:scalar-shift-obstruction} proves a sharp obstruction:
	a conservative generator whose shifted leading truncations all admit
	scalar PBFs must be tridiagonal.  By contrast, the state-dependent
	construction \(Q=V(T-I)\) preserves arbitrary fixed bandwidth.  Its
	resolvent, potential kernel, hitting probabilities, recurrence, and
	invariant measures are described in
	Theorems~\ref{thm:local-uniformization-resolvent}--\ref{thm:local-invariant-measure}.
	The endpoint results further give \(1/n\) and \(1/t\) local decay and
	logarithmic divergence of the Green kernel and resolvent.

	\item \emph{QBD models without block symmetrization (Section~6).}
	Proposition~\ref{pro:banded-kernel-QBD} reorganizes every banded kernel as
	a finite-phase QBD while retaining the left--right mixed spectral
	representation.  Corollary~\ref{coro:QBD-rank-obstruction} shows that a
	strictly \((p,q)\)-banded kernel with \(p\ne q\) cannot be made into a
	symmetric block Jacobi matrix by an invertible block-diagonal similarity.
	Proposition~\ref{pro:level-return-CF} gives the matrix continued fraction
	for the return time to a level together with the phase at return.

	\item \emph{Exact positivity and boundedness classifications
	(Sections~7 and~8).}
	For the mixed Pi\~neiro system,
	Corollary~\ref{cor:MP-moment-matrix-total-positivity} proves strict total
	positivity of the scalar step-line moment matrix, and
	Theorem~\ref{thm:MP-PBF-region} determines the exact PBF region for
	arbitrary \(p,q\), while
	Theorem~\ref{thm:MP-matrix-positive-region} gives larger regions in which
	every entry permitted by the bandwidth is positive.
	Proposition~\ref{prop:MP-boundedness-criterion} then characterizes when
	the Pi\~neiro recurrence itself is a bounded operator.  The rational
	\((3,2)\) example displays all factors and shows how an unbounded positive
	recurrence becomes a bounded stochastic band matrix.  It also gives the
	resulting urn, QBD, and continuous-time models.
\end{enumerate}

\subsection{Organization of the paper}

Sections~2--6 develop the general spectral, stochastic, continuous-time,
and QBD constructions.  Sections~7 and~8 treat the mixed Pi\~neiro model,
including its bounded and unbounded applications.  The appendix derives the
Pi\~neiro bidiagonal coefficients directly from the Cauchy moment matrix.

\section{Spectral and probabilistic preliminaries}
\label{sec:preliminaries}

This section fixes the algebraic and probabilistic notation used throughout
the paper.  Banded matrices and their positive and stochastic bidiagonal
factorizations are introduced first.  The corresponding mixed-type
spectral data and the conventions for discrete- and continuous-time Markov
processes are then recalled.

The rising Pochhammer symbol is denoted by
\[
(z)_0\coloneq1,
\qquad
(z)_n\coloneq\prod_{k=0}^{n-1}(z+k),
\qquad n\ge1.
\]
For a parameter vector
\(\boldsymbol c=(c_1,\ldots,c_\ell)\) and a multi-index
\(\boldsymbol m=(m_1,\ldots,m_\ell)\) of the same length, write
\[
(\boldsymbol c)_n\coloneq\prod_{\rho=1}^{\ell}(c_\rho)_n,
\qquad
(\boldsymbol c)_{\boldsymbol m}
\coloneq\prod_{\rho=1}^{\ell}(c_\rho)_{m_\rho},
\]
with the usual convention that an empty product equals \(1\).
For a polynomial \(f\), the notation \([x^k]f(x)\) denotes the
coefficient of \(x^k\) in \(f(x)\).
For \(\boldsymbol c=(c_1,\ldots,c_r)\) and
\(\boldsymbol d=(d_1,\ldots,d_s)\), the generalized hypergeometric
series is
\begin{equation}
\label{eq:generalized-hypergeometric-definition}
\pFq{r}{s}{\boldsymbol c}{\boldsymbol d}{z}
\coloneq
\sum_{k=0}^{\infty}
\frac{(\boldsymbol c)_k}{(\boldsymbol d)_k}\frac{z^k}{k!};
\end{equation}
see Slater~\cite{slater}. All generalized hypergeometric series used as
polynomial components below terminate.

\subsection{Banded matrices and bidiagonal factorizations}\label{subsec:sbf-pbf}

Let $T$ be a matrix indexed by $\Nzero$ with finite bandwidth: there exist
integers $p,q\ge 1$ such that \(T_{n,m}=0\) whenever \(m<n-p\) or
\(m>n+q\).
Throughout, \emph{positive} means \emph{strictly positive} (i.e.\ $>0$).

\subsubsection{Positive bidiagonal factorization (PBF)}

A basic structural hypothesis in this paper is that $T$ admits the
factorization \(T=L_1\cdots L_p\,U_q\cdots U_1\), where each $L_i$ is
lower bidiagonal and each $U_j$ is upper bidiagonal. The factorization is
called \emph{positive} if the nonzero entries of every factor are strictly
positive. Concretely, for a lower factor
$L=[\ell_{n,m}]_{n,m\in\Nzero}$ this means that $L$ is supported on the
diagonal and first subdiagonal and satisfies
\(\ell_{n,n-1}>0\) and \(\ell_{n,n}>0\) for \(n\ge1\), with the boundary
convention $\ell_{0,0}>0$. Similarly, an upper factor
$U=[u_{n,m}]_{n,m\in\Nzero}$ satisfies \(u_{n,n+1}>0\) and
\(u_{n,n}>0\) for \(n\ge0\).

When $T$ is tridiagonal (the birth--death benchmark), positivity of its two
off-diagonals gives the familiar irreducible nearest-neighbor structure,
but it does not by itself imply a PBF.  A sufficiently large positive
translation makes a bounded tridiagonal Jacobi matrix positive definite;
the bidiagonal-factorization theory then gives a PBF \cite{BFM1}, and
stochastic normalization gives a stochastic bidiagonal factorization
\cite{BFM6}.
Together with conservativity at the generator level, this recovers the
classical Karlin--McGregor setting.

Positive bidiagonal factorizations have strong variation-diminishing
consequences. The classical background on oscillatory, totally nonnegative,
and totally positive matrices can be found in
\cite{Gantmacher,Karlin,Fallat,Pinkus}. In particular, the banded matrices
arising from the factorizations used here are oscillatory and, more
strongly, belong to the class of banded totally positive matrices introduced
in \cite{BFM4}: all their nontrivial minors are strictly positive.

\subsubsection{Stochastic bidiagonal factorization (SBF)}

A matrix $T$ is called \emph{row-stochastic}, or simply stochastic in what
follows, if it has nonnegative entries and each row sums to $1$. A
\emph{stochastic bidiagonal factorization} (SBF) of a stochastic matrix $T$
is a factorization \(T=L^{(1)}\cdots L^{(p)}\,U^{(q)}\cdots U^{(1)}\),
where every $L^{(\ell)}$ is lower bidiagonal, supported on the diagonal and
the first subdiagonal, every $U^{(j)}$ is upper bidiagonal, supported on the
diagonal and the first superdiagonal, and all the factors are
row-stochastic. The positivity conditions used in this paper are stated
precisely in the following definition.

\begin{definition}[Positive stochastic bidiagonal matrices]\label{def:positive-bidiagonal}
	A lower bidiagonal matrix $L=[\ell_{n,m}]_{n,m\in\Nzero}$ is called \emph{positive stochastic} if it is stochastic, supported on the diagonal and the first subdiagonal, and satisfies
	\(\ell_{n,n-1}>0\) and
	\(\ell_{n,n}=1-\ell_{n,n-1}\in(0,1)\) for \(n\ge1\),
	together with the boundary convention $\ell_{0,0}=1$.
	Similarly, an upper bidiagonal matrix $U=[u_{n,m}]_{n,m\in\Nzero}$ is \emph{positive stochastic} if it is row-stochastic, supported on the diagonal and the first superdiagonal, and satisfies
	\(u_{n,n+1}>0\) and
	\(u_{n,n}=1-u_{n,n+1}\in(0,1)\) for \(n\ge0\).
\end{definition}

\begin{definition}[Positive stochastic bidiagonal factorization]\label{def:positive-sbf}
	A stochastic matrix $T$ admits a \emph{positive stochastic bidiagonal
	factorization} if
	\[
	T=L^{(1)}\cdots L^{(p)}\,U^{(q)}\cdots U^{(1)},
	\]
	where each
	$L^{(\ell)}$ and $U^{(j)}$ is a positive stochastic bidiagonal matrix in
	the sense of Definition~\ref{def:positive-bidiagonal}.
\end{definition}

Finally, a key fact used repeatedly, established in
\cite[Proposition~2.3]{BFM6}, is that, within the banded stochastic setting
considered here, the existence of a positive bidiagonal factorization can
be translated into the existence of a positive \emph{stochastic}
bidiagonal factorization: the factors may be chosen to be stochastic while
preserving positivity.

\subsection{Mixed-type spectral data}\label{subsec:mop-role-ctmc}

The discrete-time theory associates to a banded stochastic matrix admitting
a positive bidiagonal factorization a rectangular, entrywise positive
matrix of measures and two mixed-type polynomial families, yielding an
explicit representation of \((T^k)_{n,m}\). In the bounded-generator
regime, uniformization gives the continuous-time representation by
replacing \(x^k\) with \(\mathrm e^{-\nu t(1-x)}\), while leaving the
spectral measures and polynomial families unchanged.

For unbounded generators, global uniformization fails because no scalar
rate dominates all exits. State-dependent uniformization uses one fixed
banded transition kernel \(T\), applied at a rate
\(v_n\) depending on the current state. The mixed spectral data of \(T\)
then give the potential kernel and recurrence criterion directly. The
discounted resolvent involves the substochastic kernels
\(K_\lambda=\diag(v_n/(\lambda+v_n))T\). Each fixed \(K_\lambda\)
again admits a positive bidiagonal factorization. The kernel depends on
\(\lambda\) because a state-dependent change of time cannot be represented
by applying one scalar function to \(T\).

\subsection{Markov chains and generators}\label{subsec:ctmp-prelim}

Throughout, the state space is the countable set $\Nzero$, and
$(X_t)_{t\ge 0}$ denotes a continuous-time Markov process. The symbol
$\one$ denotes the column vector whose entries are all equal to $1$, with
dimension determined by the context, and
$\one_r$ denotes the vector of ones in $\mathbb R^r$. For an event $E$,
$\ind_E$ denotes its indicator. The symbols \(\Pr_i\) and
\(\operatorname E_i\) denote probability and expectation conditional on the
initial state \(i\); the same convention is used for discrete-time chains.
The transition probabilities are
\(p_{i,j}(t)\coloneq \Pr_i\{X_t=j\}\), for
\(i,j\in\Nzero\). They form the matrices
\(P(t)\coloneq[p_{i,j}(t)]_{i,j\in\Nzero}\). The infinitesimal transition
rates form a conservative generator
\(Q=[q_{i,j}]_{i,j\in\Nzero}\), so that \(q_{i,j}\ge0\) for \(i\ne j\)
and \(q_{i,i}=-\sum_{j\ne i}q_{i,j}\).

Writing \(\lambda_i\coloneq-q_{i,i}\), the bounded regime is
\(\sup_i\lambda_i<\infty\). It is non-explosive and permits global
uniformization. If the exit rates are unbounded, jump times may accumulate;
the unbounded part of the paper therefore imposes non-explosion and works
with the minimal transition function. The global formula and its urn
interpretation are given in
Section~\ref{subsec:bounded-continuous-time-chains}; state-dependent
uniformization is developed in Section~\ref{sec:unbounded}. The
tridiagonal birth--death chain remains the scalar Karlin--McGregor
benchmark, while higher bandwidth leads to the mixed spectral data used
here.
Here and below, for a matrix \(M\), \(M_{A,A}\) denotes the submatrix
whose rows and columns are indexed by \(A\), in the natural order.

\begin{proposition}[Green matrices and first returns to a set]
\label{pro:general-Green-first-return}
Let \(Y=(Y_n)_{n\ge0}\) be a discrete-time Markov chain with transition
matrix \(P\), and let \(A\) be a finite set of states.  For \(0\le z<1\),
write \(G_{A,A}(z)\coloneq
\sum_{n=0}^{\infty}z^n(P^n)_{A,A}\).
If
\(\tau_A^+\coloneq\inf\{n\ge1:Y_n\in A\}\), define
\[
\bigl(F_{A,A}^+(z)\bigr)_{i,j}
\coloneq
\operatorname E_i\!\left[
z^{\tau_A^+}
\ind_{\{\tau_A^+<\infty,\,Y_{\tau_A^+}=j\}}
\right],
\qquad i,j\in A.
\]
Then
\begin{equation}
\label{eq:general-Green-first-return}
F_{A,A}^+(z)=I_A-G_{A,A}(z)^{-1}.
\end{equation}
\end{proposition}

\begin{proof}
Every visit to \(A\) is either the initial visit at time zero or occurs
after a first positive return to \(A\).  Once that first return has taken
place, the chain starts afresh from the state in which it returned.
Therefore \(G_{A,A}(z)=
I_A+F_{A,A}^+(z)G_{A,A}(z)\).
Every row sum of \(F_{A,A}^+(z)\) is at most \(z<1\), because
\(\tau_A^+\ge1\).  Hence \(I_A-F_{A,A}^+(z)\) is invertible.  This
renewal identity gives
\(G_{A,A}(z)=(I_A-F_{A,A}^+(z))^{-1}\), which proves
\eqref{eq:general-Green-first-return}.
\end{proof}

Proposition~\ref{pro:general-Green-first-return} does not require a banded
matrix or a decomposition into levels.  A continued fraction appears only
when the states can be arranged in successive levels in such a way that one
step cannot jump over a neighboring level.  In that case the same
first-return argument can be repeated one level above, then two levels
above, and so on.  This repeated decomposition is carried out in
Section~\ref{sec:QBD-reblocking}.

\section{Finite urn constructions from positive bidiagonal factorizations}
\label{sec:finite-urn-protocols}

A stochastic bidiagonal factor is already an elementary Markov experiment.
A lower factor allows only death or stay, whereas an upper factor allows only
birth or stay.  Their ordered product therefore resolves one wide-band
transition into a finite sequence of one-sided nearest-neighbor decisions.
The purpose of this section is to show that, for rational probabilities,
this interpretation is equivalent to an ordered sequence of finite
two-color urns. This equivalence depends only on the factorization, not on
a particular family of multiple orthogonal polynomials.

Urn descriptions of Markov dynamics have classical antecedents in
Bernoulli's mixing problem \cite{Bernoulli1769} and in the Ehrenfest model
\cite{Ehrenfest1907}.  In the latter, one of a fixed number of balls is
selected and transferred between two urns, so the number of balls in one urn
makes a two-sided nearest-neighbor move $i\to i\pm1$. The construction here
is different.  The Markov state is not an urn occupancy: at each stage the
current state selects one urn from a state-dependent bank, and the color
drawn decides between a one-step move and staying put. Jacobi--Pi\~neiro
random walks were studied in \cite{JP}. Related stochastic factorizations into elementary
pure-birth and pure-death kernels appear in \cite{BFM5}.

\subsection{Equivalence between finite urns and rational PBFs}
\label{subsec:finite-urn-correspondence}

The next theorem makes the correspondence precise in both directions and
also identifies the smallest finite urn realizing each rational elementary
experiment.
Its stochastic-normalization step is the rational case of
\cite[Proposition~2.3]{BFM6}.  The diagonal-conjugation argument is the
same. In addition, rational factors make every intermediate diagonal
rational, which permits a
realization by finite urns with equiprobable balls.

\begin{theorem}[Finite-urn/PBF equivalence]
\label{thm:finite-urn-correspondence}
	Fix an integer $d\ge1$ and an ordered type word
	\(\boldsymbol\varepsilon=
	\begin{bNiceMatrix}\varepsilon_1&\Cdots&\varepsilon_d\end{bNiceMatrix}\), with
	$\varepsilon_r\in\{-,+\}$. For a row-stochastic matrix $P$ on
	$\Nzero$, the following statements are equivalent.
	\begin{enumerate}
		\renewcommand{\labelenumi}{\textup{(\roman{enumi})}}
		\item The matrix $P$ admits a rational positive bidiagonal
		factorization of the prescribed types,
		\begin{equation}
		\label{eq:rational-PBF}
		P=M_1M_2\cdots M_d,
		\end{equation}
		where $M_r$ is lower bidiagonal when $\varepsilon_r=-$, upper
		bidiagonal when $\varepsilon_r=+$, and all its nonzero entries are
		positive rational numbers.

		\item One step of the Markov chain with kernel $P$ is realized by a
		fixed ordered sequence of $d$ state-dependent finite two-color urn
		experiments of types $\boldsymbol\varepsilon$. At stage $r$, the
		current state $n$ selects an urn $\mathscr U_{r,n}$. A red draw
		leaves $n$ unchanged, while a blue draw sends $n$ to $n-1$ if
		$\varepsilon_r=-$, and to $n+1$ if $\varepsilon_r=+$. Both
		colors occur with positive probability at every non-boundary state;
		a lower experiment at $n=0$ stays at $0$.
	\end{enumerate}
	Moreover, in the first description the factors can be converted, without
	changing their order or their product, into a rational positive
	stochastic bidiagonal factorization
	\begin{equation}
	\label{eq:rational-SBF}
	P=\widehat M_1\widehat M_2\cdots\widehat M_d.
	\end{equation}
	If a move probability is written as $a/c$ in lowest terms, the urn with
	$a$ blue balls and $c-a$ red balls has the smallest possible total
	number of equiprobable balls. This construction is determined by the
	fixed ordered factorization, although the matrix \(P\) may admit other
	ordered factorizations. For the standard $(p,q)$-PBF, the type word consists of $p$
	lower stages followed by $q$ upper stages.
\end{theorem}

\begin{proof}
	Assume first that \eqref{eq:rational-PBF} holds.  Use the
	stochastic-normalization construction of
	\cite[Proposition~2.3]{BFM6}. The short proof is included because it also
	shows that all normalized entries are rational. Set
	\(v^{[d]}=\one\) and define positive vectors recursively by
	\(v^{[r-1]}=M_rv^{[r]}\),
	\(r\in\{d,d-1,\ldots,1\}\).
	Every component of every \(v^{[r]}\) is positive and rational. Moreover,
	\(v^{[0]}=M_1\cdots M_d\one=P\one=\one=v^{[d]}\).
	Let \(D_r=\diag(v^{[r]})\) and put
	\begin{equation}
	\label{eq:rational-stochastic-refinement}
	\widehat M_r=D_{r-1}^{-1}M_rD_r.
	\end{equation}
	Then
	\(\widehat M_r\one=D_{r-1}^{-1}M_rv^{[r]}
	=D_{r-1}^{-1}v^{[r-1]}=\one\), so each \(\widehat M_r\) is
	stochastic. Positive diagonal scaling
	preserves its lower or upper bidiagonal support and positivity, and
	rationality follows from \eqref{eq:rational-stochastic-refinement}.  The
	intermediate diagonal matrices telescope, while \(D_0=D_d=I\), so
	\(\widehat M_1\cdots\widehat M_d=D_0^{-1}PD_d=P\).

	It remains to realize a factor.  If \(\widehat M_r\) is lower bidiagonal,
	write its off-diagonal probability in row \(n\ge1\) in lowest terms as
	\((\widehat M_r)_{n,n-1}=a_{r,n}/c_{r,n}\), where
	\(0<a_{r,n}<c_{r,n}\).
	Fill an urn with \(a_{r,n}\) blue balls and
	\(c_{r,n}-a_{r,n}\) red balls.  A blue draw sends \(n\) to \(n-1\),
	whereas a red draw leaves the state at \(n\); at the boundary \(n=0\)
	the state stays at \(0\).  For an
	upper factor, use the same construction with
	\((\widehat M_r)_{n,n+1}=a_{r,n}/c_{r,n}\), interpreting blue as
	\(n\to n+1\).  The urn used at stage \(r\) is selected according to the
	current intermediate state.  Thus the construction uses a bank of urns
	\(\mathscr U_{r,n}\), indexed both by the stage \(r\) and by the state
	\(n\) entering that stage, rather than one fixed physical urn.  Exactly one
	ball is drawn at each stage; the urn compositions are unchanged between
	cycles.  Conditional multiplication of the \(d\) binary
	probabilities gives precisely the matrix product
	\eqref{eq:rational-SBF}.

	Conversely, suppose that the sequence of urn experiments in
	item~\textup{(ii)} is given.
	Its stage
	$r$ defines a rational row-stochastic bidiagonal matrix
	$\widehat M_r$: its two nonzero row entries are exactly the blue and
	red draw probabilities of $\mathscr U_{r,n}$. The factor is lower or
	upper according to $\varepsilon_r$, and it is positive away from the
	forced lower boundary. Conditioning successively on the intermediate
	states shows that the transition kernel of the complete sequence is
	$P=\widehat M_1\cdots\widehat M_d$. Thus the urn sequence gives the
	rational positive stochastic bidiagonal factorization in item~\textup{(i)}.
	The
	minimality assertion follows because a rational probability $a/c$ in
	lowest terms cannot be represented by an urn with fewer than $c$
	equiprobable balls.
\end{proof}

In this ordered realization, the state \(n\) is an index of the Markov
chain, not a number of balls.  After each draw, the updated state selects
the urn used at the next stage.  Rationality is what replaces weighted
trials by finite urns with equiprobable balls; the reduced fraction gives
the smallest urn for each elementary experiment.

\subsection{First-return probabilities from the continued fraction of the bidiagonal factors}
\label{subsec:factor-resolved-return-CF}

Let \(M=p+q\), let
\(T=\widehat M_1\widehat M_2\cdots\widehat M_M\)
be a positive stochastic bidiagonal factorization.  The following
construction retains the individual factors instead of multiplying them
first.  On the direct sum of \(M\) copies of the sequence space, introduce
the cyclic matrix
\begin{equation}
\label{eq:factor-cyclic-matrix}
\mathscr C_T
=
\begin{bmatrix}
0&\widehat M_1&0&\cdots&0\\
0&0&\widehat M_2&\ddots&\vdots\\
\vdots&\ddots&\ddots&\ddots&0\\
0&\cdots&0&0&\widehat M_{M-1}\\
\widehat M_M&0&\cdots&0&0
\end{bmatrix}.
\end{equation}
For \(n\ge0\), group together the \(M\) states with scalar index \(n\),
and denote their span by \(\mathcal E_n\simeq\mathbb R^M\), ordered by
factor position.  Since every factor is bidiagonal, \(\mathscr C_T\) is
block tridiagonal in the decomposition
\(\bigoplus_{n\ge0}\mathcal E_n\).  Write its three block diagonals as
\begin{equation*}
\mathsf A_n=[\mathscr C_T]_{\mathcal E_n,\mathcal E_n},\qquad
\mathsf B_n=[\mathscr C_T]_{\mathcal E_n,\mathcal E_{n+1}},\qquad
\mathsf D_n=[\mathscr C_T]_{\mathcal E_{n+1},\mathcal E_n}.
\end{equation*}
Thus every nonzero entry of these \(M\times M\) blocks is an entry of one
of the displayed stochastic bidiagonal factors.

Let \(\mathscr C_T^{\langle n\rangle}\) be the compression of
\(\mathscr C_T\) to
\(\mathcal E_n\oplus\mathcal E_{n+1}\oplus\cdots\), and define the
boundary resolvent block
\begin{equation}
\label{eq:factor-tail-definition}
\Phi_n^T(\tau)
\coloneq
\left[
\left(I-\tau\mathscr C_T^{\langle n\rangle}\right)^{-1}
\right]_{\mathcal E_n,\mathcal E_n}.
\end{equation}
Schur complementation of the first retained block gives the matrix
continued fraction
\begin{equation}
\label{eq:factor-tail-recursion}
\Phi_n^T(\tau)
=
\left[
I_M-\tau\mathsf A_n
-\tau^2\mathsf B_n\Phi_{n+1}^T(\tau)\mathsf D_n
\right]^{-1}.
\end{equation}
For \(N\ge n\), set
\begin{equation}
\label{eq:factor-tail-terminal}
\Phi_N^{T,[N]}(\tau)=(I_M-\tau\mathsf A_N)^{-1},
\end{equation}
and define backwards, for \(k=N-1,\ldots,n\),
\[
\Phi_k^{T,[N]}(\tau)
=
\left[
I_M-\tau\mathsf A_k
-\tau^2\mathsf B_k\Phi_{k+1}^{T,[N]}(\tau)\mathsf D_k
\right]^{-1}.
\]
These finite resolvent blocks sum paths confined to levels
\(n,\ldots,N\); hence, for \(0\le\tau<1\), they increase entrywise to
\(\Phi_n^T(\tau)\).  The Neumann series defining \(\Phi_n^T\) converges
because \(\|\mathscr C_T^{\langle n\rangle}\|_\infty\le1\).

Let \(Y=(Y_k)_{k\ge0}\) be the chain
with kernel \(T\), started from \(0\), and set
\(\rho_0^+\coloneq\inf\{k\ge1:Y_k=0\}\).
Write \(g_0(z)\coloneq
\sum_{k=0}^{\infty}(T^k)_{0,0}z^k\).
Also write
\[
f_0(z)\coloneq
\operatorname E_0[
z^{\rho_0^+}\ind_{\{\rho_0^+<\infty\}}].
\]
Let \(\Pi_1\) project onto the first factor position in
\eqref{eq:factor-cyclic-matrix}.  Direct block multiplication gives, for
\(k\ge0\) and \(0\le r<M\),
\[
\Pi_1\mathscr C_T^{Mk+r}\Pi_1^{\mathsf T}
=
\begin{cases}
T^k,&r=0,\\
0,&1\le r<M.
\end{cases}
\]
Therefore
\begin{equation}
\label{eq:factor-cyclic-resolvent}
\Pi_1(I-\tau\mathscr C_T)^{-1}\Pi_1^{\mathsf T}
=\sum_{k=0}^{\infty}\tau^{Mk}T^k
=(I-\tau^M T)^{-1}.
\end{equation}
The state \((1,0)\) is the first coordinate of \(\mathcal E_0\).
Consequently,
\begin{equation}
\label{eq:factor-resolved-return-Green}
e_1^{\mathsf T}\Phi_0^T(\tau)e_1
=
\left[(I-\tau^M T)^{-1}\right]_{0,0}
=g_0(\tau^M),
\qquad 0\le\tau<1,
\end{equation}
where \(e_1\) is the first coordinate vector of \(\mathbb R^M\).
Applying the scalar renewal identity of
Proposition~\ref{pro:general-Green-first-return} now yields
\begin{equation}
\label{eq:factor-resolved-first-return-law}
f_0(z)
=
1-\frac{1}{g_0(z)}
=
1-\frac{1}{
e_1^{\mathsf T}\Phi_0^T(z^{1/M})e_1},
\qquad 0\le z<1.
\end{equation}
After the renewal transform in
\eqref{eq:factor-resolved-first-return-law}, the resulting Taylor
coefficients are the probabilities
\(\Pr_0\{\rho_0^+=k\}\).

The finite convergents defined above therefore give successive
approximations to the complete first-return distribution.  For rational
factor probabilities, every convergent is a matrix of rational functions,
so its coefficients can be computed exactly.  This probabilistic
interpretation will be evaluated explicitly for the mixed Pi\~neiro chain
below.  Related matrix
continued fractions for banded path models are studied in
\cite{KimLopezGarciaProkhorov2025}.

\subsection{Spectral consequence of the urn/SBF correspondence}
\label{subsec:urn-spectral-representation}

Theorem~\ref{thm:finite-urn-correspondence} gives the equivalence, in
the rational case, between the ordered finite-urn sequence and the rational
positive stochastic bidiagonal factorization; arbitrary real probabilities
give the corresponding Bernoulli interpretation. Consequently, the
transition matrix of the ordered sequence satisfies the hypotheses of the
bounded spectral Favard theorem of \cite{BFM1}, in its stochastic form
\cite[Theorem~2.19]{BFM6}. The result is a Karlin--McGregor representation
through mixed-type multiple orthogonality.

\begin{corollary}[Karlin--McGregor representation for successive urn experiments]
\label{cor:urn-protocol-spectral}
	Consider a fixed ordered sequence on $\Nzero$ consisting of $p$
	death-or-stay binary experiments followed by $q$ birth-or-stay binary
	experiments, with strictly positive move and stay probabilities away from
	the boundary. Let $(Y_k)_{k\in\Nzero}$ be the discrete-time Markov chain
	whose transition kernel $T$ is one complete cycle of these experiments.
	Then there exist two families of mixed-type multiple orthogonal
	polynomials
	\(\bigl\{A_m^{(a)}\bigr\}_{m\in\Nzero,\,1\le a\le p}\),
	\(\bigl\{B_n^{(b)}\bigr\}_{n\in\Nzero,\,1\le b\le q}\),
	and an entrywise positive $q\times p$ matrix of finite Borel measures
	$\Psi=[\psi_{b,a}]$, supported on $[0,1]$, normalized by
	\(A_0^{(a)}(x)=\delta_{a,1}\) and
	\(B_0^{(b)}(x)=\delta_{b,1}\),
	such that
	\begin{equation}
	\label{eq:DT-KM}
		\Pr_n\{Y_k=m\}=(T^k)_{n,m}
		=
		\sum_{a=1}^{p}\sum_{b=1}^{q}
		\int_0^1
		B_n^{(b)}(x)x^kA_m^{(a)}(x)\,\dd\psi_{b,a}(x),
		\qquad k\in\Nzero.
	\end{equation}
	Consequently, the family of all transition probabilities of the urn chain
	is represented by one time-independent matrix of measures and the two
	mixed-type multiple orthogonal polynomial families. Rationality is unnecessary for this
	spectral conclusion; it is needed only to realize every binary experiment
	by an actual finite urn containing equiprobable balls.
\end{corollary}

\begin{proof}
	Each stage kernel is a positive stochastic lower or upper bidiagonal
	matrix, so their ordered product is a positive stochastic bidiagonal
	factorization of $T$. The bounded spectral Favard theorem \cite{BFM1} and
	its Markov specialization \cite[Theorem~2.19]{BFM6} give the two
	mixed-type polynomial families, the entrywise positive matrix of measures,
	and formula \eqref{eq:DT-KM}.
\end{proof}

Thus the successive urn experiments display the individual bidiagonal
factors, while the Favard data represent all transition probabilities of
the same chain simultaneously.

\subsection{Darboux transformations as cyclic reorderings of the urn experiments}
\label{subsec:urn-Darboux}

\begin{proposition}[Darboux transforms as reordered urn cycles]
\label{pro:Darboux-reordered-urns}
	Let \(P=\widehat M_1\cdots\widehat M_d\)
	be the rational stochastic bidiagonal factorization in
	Theorem~\ref{thm:finite-urn-correspondence}. Fix
	\(r\in\{1,\ldots,d-1\}\) and set
	\(K_r=\widehat M_1\cdots\widehat M_r\) and
	\(R_r=\widehat M_{r+1}\cdots\widehat M_d\).
	The cyclic Darboux transform
	\begin{equation}
	\label{eq:Darboux-reordered-kernel}
	\widetilde P_r
	\coloneq
	R_rK_r
	=
	\widehat M_{r+1}\cdots\widehat M_d
	\widehat M_1\cdots\widehat M_r
	\end{equation}
	is again a rational stochastic matrix expressed as a product of positive
	stochastic bidiagonal factors.  Its finite-urn model is obtained from that of \(P\) by
	starting the ordered urn cycle at stage \(r+1\).  Moreover,
	\begin{equation}
	\label{eq:Darboux-Markov-intertwinings}
	PK_r=K_r\widetilde P_r,
	\qquad
	R_rP=\widetilde P_rR_r.
	\end{equation}
\end{proposition}

\begin{proof}
	Products of stochastic matrices are stochastic, and positivity and
	rationality of every elementary bidiagonal factor are unchanged by the
	cyclic permutation.  The urn statement follows by executing the same
	binary experiments in the order displayed in
	\eqref{eq:Darboux-reordered-kernel}. Finally, using \(P=K_rR_r\),
	one obtains \(PK_r=K_rR_rK_r=K_r\widetilde P_r\) and
	\(R_rP=R_rK_rR_r=\widetilde P_rR_r\).
\end{proof}

\section{Bounded continuous-time Markov processes}
\label{sec:bounded-ctmc}

This section develops the continuous-time theory when the exit rates are
uniformly bounded. Global uniformization gives the semigroup and resolvent
from the discrete mixed spectral representation and the sequence of urn
experiments. It also gives formulas for first passage, recurrence,
ergodicity, and long-time asymptotics at the spectral endpoint.

\subsection{Global uniformization and spectral representation}
\label{subsec:bounded-continuous-time-chains}

The discrete-time spectral data have already been recorded in
Corollary~\ref{cor:urn-protocol-spectral}: the urn/SBF correspondence gives
the mixed-type polynomial families, the entrywise positive matrix of
measures, and formula \eqref{eq:DT-KM}. Throughout the paper, matrix entries
are written in the Markov convention: the first index is the initial state
(row) and the second index is the final state (column). Thus the row index is
paired with the $B$-family and the column index with the $A$-family. In the
classical birth--death case $p=q=1$, formula \eqref{eq:DT-KM} reduces to the
scalar Karlin--McGregor representation \cite{KarlinMcGregor,KM1957}.

\medskip

For continuous time, let
$Q=[q_{n,m}]_{n,m\in\Nzero}$ be a conservative $Q$-matrix (generator) for a
continuous-time Markov chain (CTMC) on $\Nzero$, meaning that
\(q_{n,m}\ge0\) for \(m\ne n\) and
\(\sum_{m\in\Nzero}q_{n,m}=0\) for \(n\in\Nzero\).
The exit rate from $n$ is
\(\lambda_n\coloneq-q_{n,n}=\sum_{m\ne n}q_{n,m}\).
The generator $Q$ is called bounded if
\(\Lambda\coloneq\sup_{n\in\Nzero}\lambda_n<\infty\).
This assumption implies non-explosion and ensures that the transition
semigroup \(P(t)\coloneq\e^{tQ}\), \(t\ge0\),
is well defined as a conservative stochastic semigroup. As in discrete
time, $Q$ is assumed to be banded: there exists a fixed finite range such
that $q_{n,m}=0$ whenever $|n-m|$ exceeds that range.

\medskip

Assume that $Q$ is bounded and fix $\nu>0$ with $\nu\ge \Lambda$. Define
\begin{equation}\label{eq:uniform-T}
	T\coloneq I+\frac{1}{\nu}Q.
\end{equation}
Then $T$ is stochastic and has the same bandwidth as $Q$. Indeed, for $m\ne n$ one has $T_{n,m}=\nu^{-1}q_{n,m}\ge 0$, while
\(T_{n,n}=1+\nu^{-1}q_{n,n}=1-\lambda_n/\nu\ge0\) and
\(\sum_{m\in\Nzero}T_{n,m}=1\).
Uniformization (randomization) yields \cite{bremaud}
\begin{equation}\label{eq:uniformization}
	\e^{tQ}
	=
	\sum_{k=0}^{\infty}\e^{-\nu t}\frac{(\nu t)^k}{k!}\,T^k,
	\qquad t\ge 0.
\end{equation}
Equivalently, if $(N_t)_{t\ge 0}$ is a Poisson process with rate $\nu$, independent of a discrete-time chain $(Y_k)_{k\in\Nzero}$ with transition matrix $T$, then $X_t\coloneq Y_{N_t}$ has generator $Q=\nu(T-I)$ and semigroup \eqref{eq:uniformization}.
Thus the same identity may be read in two directions. Starting from \(Q\)
and constructing \(T=I+\nu^{-1}Q\) is called \emph{uniformization};
starting from \(T\) and running the discrete chain at the times of an
independent Poisson process is called \emph{Poissonization}.

When \(T\) has a rational PBF,
Combining standard uniformization with
Theorem~\ref{thm:finite-urn-correspondence} gives a continuous-time urn
model. An
exponential clock of rate \(\nu\) means that its waiting time \(E_\nu\)
satisfies
\(\Pr\{E_\nu>t\}=\e^{-\nu t}\) for \(t\ge0\), and hence
\(\operatorname E[E_\nu]=1/\nu\). At each ring, one complete urn cycle is
performed, and the process remains in the resulting state until the next
ring. Thus
the clock determines the event times and the urn cycle determines the next
state. Cycles that return
the same state are virtual jumps; after thinning them, the holding time at
\(n\) is exponential with rate \(\nu(1-T_{n,n})=-q_{n,n}\).

\begin{remark}[One clock for the complete cycle versus separate factor clocks]
\label{rem:poissonized-urn-cycle}
	For \(T=\widehat M_1\cdots\widehat M_d\), one Poisson clock triggers the
	whole ordered urn cycle. Independent clocks for the factors generally
	produce a different generator,
	\(\sum_r\nu_r(\widehat M_r-I)\), rather than
	\(\nu(\widehat M_1\cdots\widehat M_d-I)\).
\end{remark}

The Darboux identities remain valid after Poissonization. With the notation
of Proposition~\ref{pro:Darboux-reordered-urns}, put
\(Q=\nu(P-I)\) and
\(\widetilde Q_r=\nu(\widetilde P_r-I)\).
Then
\begin{equation}
\label{eq:Darboux-generator-intertwinings}
QK_r=K_r\widetilde Q_r,
\qquad
R_rQ=\widetilde Q_rR_r,
\end{equation}
and hence
\(\e^{tQ}K_r=K_r\e^{t\widetilde Q_r}\) and
\(R_r\e^{tQ}=\e^{t\widetilde Q_r}R_r\).
Thus the continuous-time Darboux transform is represented by the same
Poisson clock and the cyclically reordered urn experiment.

\medskip

The resolvent and its probabilistic interpretation will also be used. For
$\lambda>0$, define \(R(\lambda)\coloneq(\lambda I-Q)^{-1}\).
In the bounded-generator regime (and, more generally, for $C_0$-semigroups), the resolvent is the Laplace transform of the semigroup:
\begin{equation}\label{eq:Laplace-resolvent-prelim}
	R(\lambda)=\int_{0}^{\infty}\e^{-\lambda t}P(t)\,\dd t,
	\qquad \lambda>0,
\end{equation}
where the integral is understood entrywise (equivalently, in operator norm in the bounded case).
In parallel, define the discounted potential (discounted Green) kernel
\begin{equation}\label{eq:discounted-potential-def}
	U^{(\lambda)}_{n,m}\coloneq \int_{0}^{\infty}\e^{-\lambda t}\,(P(t))_{n,m}\,\dd t,
	\qquad \lambda>0.
\end{equation}
Comparing \eqref{eq:Laplace-resolvent-prelim} with \eqref{eq:discounted-potential-def} shows that, entrywise,
\begin{equation}\label{eq:UequalsR}
	U^{(\lambda)}_{n,m}=\bigl(R(\lambda)\bigr)_{n,m},
	\qquad n,m\in\Nzero.
\end{equation}

Here is a short proof of \eqref{eq:Laplace-resolvent-prelim} under boundedness.
Set
\(\widetilde{R}(\lambda)\coloneq
\int_{0}^{\infty}\e^{-\lambda t}P(t)\,\dd t\).
Since $Q$ is bounded, differentiation under the integral is legitimate and
\[
\frac{\mathrm{d}}{\mathrm{d}t}\bigl(\e^{-\lambda t}P(t)\bigr)
=
-\lambda\e^{-\lambda t}P(t)+\e^{-\lambda t}P(t)Q
=
-\e^{-\lambda t}P(t)(\lambda I-Q).
\]
Integrating from $0$ to $\infty$ and using $\e^{-\lambda t}P(t)\to 0$ as $t\to\infty$ (because $\lambda>0$) gives
\[
\widetilde{R}(\lambda)(\lambda I-Q)
=
-\int_{0}^{\infty}\frac{\mathrm{d}}{\mathrm{d}t}\bigl(\e^{-\lambda t}P(t)\bigr)\dd t
=
P(0)
=
I.
\]
Similarly, one checks $(\lambda I-Q)\widetilde{R}(\lambda)=I$, hence $\widetilde{R}(\lambda)=R(\lambda)$.

\begin{remark}[Probabilistic meaning of $U^{(\lambda)}$]\label{rem:prob-meaning-discount}
	Let $\tau_\lambda$ be an exponential time with parameter $\lambda$, independent of the chain.
	Then, for all $n,m\in\Nzero$,
	\(U^{(\lambda)}_{n,m}=
	\operatorname{E}_n\!\left[\int_{0}^{\tau_\lambda}\ind_{\{X_t=m\}}\,\dd t\right]\).
	Equivalently, $U^{(\lambda)}$ is the Green kernel of the chain killed at an independent exponential time with rate $\lambda$.
\end{remark}

\medskip

Assume now that the uniformized kernel $T$ admits a positive stochastic bidiagonal factorization, so that Corollary~\ref{cor:urn-protocol-spectral} applies. Inserting the discrete-time spectral representation into the Poisson mixture \eqref{eq:uniformization} yields a closed integral formula in continuous time. Laplace transforming in time then produces the resolvent/discounted-potential representation.

\begin{theorem}[Continuous-time spectral representation, bounded case]\label{thm:CT-bounded}
	Let $Q$ be a bounded banded generator, fix $\nu\ge \sup_n(-q_{n,n})$, and let $T$ be defined by \eqref{eq:uniform-T}.
	Assume that $T$ admits a positive stochastic bidiagonal factorization, and let $\Psi$, $A^{(a)}_m$, $B^{(b)}_n$ be the associated spectral data from Corollary~\ref{cor:urn-protocol-spectral}.
	Then, for all $t\ge 0$ and all $n,m\in\Nzero$,
	\begin{equation}\label{eq:CT-KM}
		(\e^{tQ})_{n,m}
		=
		\sum_{a=1}^{p}\sum_{b=1}^{q}
		\int_{0}^{1}
		B^{(b)}_n(x)\,\e^{-\nu t(1-x)}\,A^{(a)}_m(x)\,\dd\psi_{b,a}(x).
	\end{equation}
	Moreover, for every $\lambda>0$,
	\begin{equation}\label{eq:resolvent-spectral}
		U^{(\lambda)}_{n,m}=\bigl(R(\lambda)\bigr)_{n,m}
		=
		\sum_{a=1}^{p}\sum_{b=1}^{q}
		\int_{0}^{1}
		B_n^{(b)}(x)\,\frac{1}{\lambda+\nu(1-x)}\,A_m^{(a)}(x)\,\dd\psi_{b,a}(x).
	\end{equation}
\end{theorem}

\begin{proof}
	Fix $t\ge 0$ and $n,m\in\Nzero$. By \eqref{eq:uniformization},
\(	(\e^{tQ})_{n,m}
	=
	\sum_{k=0}^{\infty}\e^{-\nu t}\frac{(\nu t)^k}{k!}\,(T^k)_{n,m}\).
	By Corollary~\ref{cor:urn-protocol-spectral}, for each $k\in\Nzero$,
	\[
	(T^k)_{n,m}
	=
	\sum_{a=1}^{p}\sum_{b=1}^{q}
	\int_{0}^{1}
	B_n^{(b)}(x)\,x^{k}\,A_m^{(a)}(x)\,\dd\psi_{b,a}(x).
	\]
	Insert this identity. Since $\psi_{b,a}$ is a finite positive measure, the exchange of the $k$-sum with the $x$-integral is justified by Fubini's theorem with domination. Indeed, for fixed $n,m,a,b$ the polynomial product $B_n^{(b)}A_m^{(a)}$ is bounded on $[0,1]$, and
	\[
	\sum_{k=0}^{\infty}\e^{-\nu t}\frac{(\nu t)^k}{k!}|x|^{k}
	\le
	\sum_{k=0}^{\infty}\e^{-\nu t}\frac{(\nu t)^k}{k!}=1,
	\qquad x\in[0,1].
	\]
	Therefore
	\[
	(\e^{tQ})_{n,m}
	=
	\sum_{a=1}^{p}\sum_{b=1}^{q}
	\int_{0}^{1}
	B_n^{(b)}(x)
	\Biggl(
	\sum_{k=0}^{\infty}\e^{-\nu t}\frac{(\nu t)^k}{k!}x^{k}
	\Biggr)
	A_m^{(a)}(x)\,\dd\psi_{b,a}(x).
	\]
	For $x\in[0,1]$,
	\[
	\sum_{k=0}^{\infty}\e^{-\nu t}\frac{(\nu t)^k}{k!}x^{k}
	=
	\e^{-\nu t}\sum_{k=0}^{\infty}\frac{(\nu tx)^k}{k!}
	=
	\e^{-\nu t}\e^{\nu tx}
	=
	\e^{-\nu t(1-x)}.
	\]
	This gives \eqref{eq:CT-KM}.
	For $\lambda>0$, combine \eqref{eq:UequalsR} with \eqref{eq:Laplace-resolvent-prelim} and write
	\(\bigl(R(\lambda)\bigr)_{n,m}
	=\int_0^\infty \e^{-\lambda t}(\e^{tQ})_{n,m}\,\dd t\).
	Insert \eqref{eq:CT-KM} and apply Fubini with domination. The kernel is bounded by $\e^{-\lambda t}$ on $[0,1]$, while $B_n^{(b)}A_m^{(a)}$ is bounded and $\psi_{b,a}$ has finite total variation. Hence
	\(\int_{0}^{\infty}\e^{-\lambda t}\e^{-\nu t(1-x)}\,\dd t
	=1/(\lambda+\nu(1-x))\).
	This yields \eqref{eq:resolvent-spectral}.
\end{proof}

When $T$ is the transition kernel of one complete ordered sequence of urn
experiments in Corollary~\ref{cor:urn-protocol-spectral}, formula
\eqref{eq:CT-KM} is
the Karlin--McGregor representation of its continuous-time Poissonization.
Thus continuous time changes the scalar spectral kernel from $x^k$ to
$\e^{-\nu t(1-x)}$, but it introduces neither new mixed-type polynomials nor
new spectral measures.

\subsection{Spectral and probabilistic consequences: recurrence and long-time behavior}
\label{subsec:bounded-potential-long-time}

The continuous-time spectral representation of
Theorem~\ref{thm:CT-bounded} also yields qualitative information about the
Markov process.  First, the non-discounted potential kernel is obtained as
a boundary value of the resolvent, and resolvent entries are related to
Laplace transforms of first-passage times.  The positive stochastic
bidiagonal factorization and uniformization are then used to establish
irreducibility and to characterize recurrence, transience, and positive
recurrence.  In
spectral terms, recurrence is governed by the behavior of
\(\psi_{1,1}\) near \(x=1\), whereas ergodicity is determined by the
presence of a mass point at \(x=1\), corresponding to the generator
eigenvalue \(0\).  This also provides explicit formulas for the stationary
distribution. Finally, when the spectral matrix has a power-law density at
the endpoint, the same representation gives precise discrete- and
continuous-time local limits and, in the critical case, first-return
asymptotics.

\begin{corollary}[Non-discounted potentials]\label{cor:non-discounted-potential}
	Whenever the non-discounted potential kernel
	\(U_{n,m}\coloneq\int_0^\infty(P(t))_{n,m}\,\dd t\)
	is finite, one has $U_{n,m}=\lim_{\lambda\downarrow 0}U^{(\lambda)}_{n,m}$. Hence its spectral representation is the following improper limit of \eqref{eq:resolvent-spectral}:
	\[
	U_{n,m}
	=
	\sum_{a=1}^{p}\sum_{b=1}^{q}
	\int_{0}^{1}
	B_n^{(b)}(x)\,\frac{1}{\nu(1-x)}\,A_m^{(a)}(x)\,\dd\psi_{b,a}(x),
	\]
	whenever these integrals are absolutely convergent. For the scalar positive measure $\psi_{1,1}$ used in the recurrence criterion, the corresponding diagonal formula follows by monotone convergence.
\end{corollary}

\medskip

It is also useful to record the resolvent interpretation of first-passage
Laplace transforms; this part is independent of the spectral representation.
For $B\subset\Nzero$, let
\(\tau_B\coloneq\inf\{t\ge0:X_t\in B\}\) and
\(\tau_j\coloneq\tau_{\{j\}}\).

\begin{proposition}[Resolvent and first-passage Laplace transforms]\label{pro:bounded-resolvent-hitting}
	Fix $\lambda>0$ and $i,j\in\Nzero$. Then
	\begin{equation}\label{eq:bounded-resolvent-hitting-factorization}
		\bigl(R(\lambda)\bigr)_{i,j}
		=
		\operatorname{E}_i\!\left[\mathrm e^{-\lambda \tau_j}\,\ind_{\{\tau_j<\infty\}}\right]\,
		\bigl(R(\lambda)\bigr)_{j,j},
	\end{equation}
	hence
	\begin{equation}\label{eq:bounded-Laplace-hitting-time}
		\operatorname{E}_i\!\left[\mathrm e^{-\lambda \tau_j}\,\ind_{\{\tau_j<\infty\}}\right]
		=
		\frac{\bigl(R(\lambda)\bigr)_{i,j}}{\bigl(R(\lambda)\bigr)_{j,j}}.
	\end{equation}
\end{proposition}

\begin{proof}
	By definition,
	\(\bigl(R(\lambda)\bigr)_{i,j}
	=
	\operatorname{E}_i\!\left[\int_{0}^{\infty}\mathrm e^{-\lambda t}\,\ind_{\{X_t=j\}}\,\dd t\right]\).
	Since $\ind_{\{X_t=j\}}=0$ for $t<\tau_j$, splitting at $\tau_j$ yields
	\[
	\bigl(R(\lambda)\bigr)_{i,j}
	=
	\operatorname{E}_i\!\left[\ind_{\{\tau_j<\infty\}}
	\int_{\tau_j}^{\infty}\mathrm e^{-\lambda t}\,\ind_{\{X_t=j\}}\,\dd t\right].
	\]
	By the strong Markov property at $\tau_j$ and the change of variables $t=\tau_j+u$,
	\[
	\int_{\tau_j}^{\infty}\mathrm e^{-\lambda t}\,\ind_{\{X_t=j\}}\,\dd t
	=
	\mathrm e^{-\lambda \tau_j}
	\int_{0}^{\infty}\mathrm e^{-\lambda u}\,\ind_{\{X_{\tau_j+u}=j\}}\,\dd u.
	\]
	Conditioning on $\{\tau_j<\infty\}$, the process $(X_{\tau_j+u})_{u\ge 0}$ has the law of the chain started from $j$.
	Therefore the conditional expectation of the last integral equals $\bigl(R(\lambda)\bigr)_{j,j}$, which gives
	\eqref{eq:bounded-resolvent-hitting-factorization}. Dividing by $\bigl(R(\lambda)\bigr)_{j,j}>0$ yields
	\eqref{eq:bounded-Laplace-hitting-time}.
\end{proof}

\medskip

Irreducibility and the recurrence classification in the bounded case are
recorded next.

\begin{proposition}[Irreducibility]\label{pro:CT-irreducible}
	Assume the hypotheses of Theorem~\ref{thm:CT-bounded}. Then $T$ is irreducible. Consequently, the CTMC semigroup $P(t)=\e^{tQ}$ is irreducible.
\end{proposition}

\begin{proof}
	Since $T$ admits a positive stochastic bidiagonal factorization, there exist lower and upper bidiagonal stochastic matrices
	\(L^{(1)},\dots,L^{(p)}\) and \(U^{(1)},\dots,U^{(q)}\) such that
	\(T=L^{(1)}\cdots L^{(p)}U^{(q)}\cdots U^{(1)}\).
	Each lower factor has positive nontrivial subdiagonal and positive nontrivial diagonal entries, and similarly each upper factor has positive nontrivial superdiagonal and positive nontrivial diagonal entries. Hence in each lower factor every state $n\ge 1$ has positive probability to move to $n-1$ and a positive probability to stay at $n$, while in each upper factor every state $n\ge 0$ has positive probability to move to $n+1$ and a positive probability to stay at $n$.
	
	Fix $i,j\in\Nzero$. If $j\ge i$, then composing the upper factors yields a path from $i$ to $j$ with positive probability in a finite number of steps. If $j\le i$, the analogous statement holds using the lower factors. Therefore there exists $k\ge 1$ such that $(T^k)_{i,j}>0$, so $j$ is accessible from $i$; by symmetry, also $i$ is accessible from $j$, hence $T$ is irreducible.
	
	By uniformization \eqref{eq:uniformization}, for every $t>0$,
	\((P(t))_{i,j}
	=
	\sum_{\ell=0}^{\infty}\e^{-\nu t}\frac{(\nu t)^\ell}{\ell!}\,(T^\ell)_{i,j}\).
	Choosing $k$ such that $(T^k)_{i,j}>0$ gives
	\[
	(P(t))_{i,j}\ge \e^{-\nu t}\frac{(\nu t)^k}{k!}\,(T^k)_{i,j}>0,
	\]
	so the CTMC is irreducible.
\end{proof}

\medskip

Uniformization shows that recurrence/transience and positive recurrence are
shared by the CTMC and its uniformized chain. The criteria are recalled in
terms of the spectral data of $T$, which, by
Theorem~\ref{thm:CT-bounded}, also control the large-time behavior of the
continuous-time kernel \eqref{eq:CT-KM}.

\begin{theorem}[Recurrence criterion]\label{thm:CT-recurrence-psi}
	Assume the hypotheses of Theorem~\ref{thm:CT-bounded}. Then the CTMC generated by $Q$ is recurrent if and only if
	\[
	\int_0^{1}\frac{\dd\psi_{1,1}(x)}{1-x}=+\infty.
	\]
	Otherwise, it is transient.
\end{theorem}

\begin{proof}
	By Proposition~\ref{pro:CT-irreducible}, recurrence/transience is a class property, hence it suffices to decide it for a fixed state, say $0$.
	Let $(Y_k)_{k\ge 0}$ be the discrete-time chain with transition matrix $T$. Uniformization gives, for every $t\ge0$,
	\[
	(P(t))_{0,0}
	=
	\sum_{k=0}^{\infty}\e^{-\nu t}\frac{(\nu t)^k}{k!}(T^k)_{0,0}.
	\]
	All terms are nonnegative, so Tonelli may be applied to this probabilistic series. Hence the continuous-time Green function satisfies
	\[
	\int_0^\infty (P(t))_{0,0}\,\dd t
	=
	\frac1\nu\sum_{k=0}^{\infty}(T^k)_{0,0}.
	\]
	Thus the CTMC is recurrent if and only if the uniformized discrete-time chain is recurrent. The integral condition is precisely the discrete-time recurrence criterion for $T$ in the PBF setting (see \cite{BFM6}), using the normalization $A_0^{(a)}=\delta_{a,1}$ and $B_0^{(b)}=\delta_{b,1}$.
\end{proof}

\begin{theorem}[Ergodicity and stationary distribution]\label{thm:bounded-stationary}
	Assume the hypotheses of Theorem~\ref{thm:CT-bounded}. Then the CTMC generated by $Q$ is ergodic (equivalently, positive recurrent) if and only if $1$ is a mass point of $\psi_{1,1}$, that is,
	\(m_{b,a}\coloneq\psi_{b,a}(\{1\})\)
	satisfies $m_{1,1}>0$. Since every $\psi_{b,a}$ is positive, all the other boundary masses $m_{b,a}$ are automatically nonnegative.
	In that case, for any fixed reference state $n_0\in\Nzero$, the stationary distribution $\pi$ is given by
	\[
	\pi_m
	=
	\sum_{a=1}^p\sum_{b=1}^q B_{n_0}^{(b)}(1)\,m_{b,a}\,A_m^{(a)}(1),
	\qquad m\in\Nzero.
	\]
	Equivalently, the same expression with $n_0$ replaced by any $n\in\Nzero$ is independent of $n$. Moreover, $P(t)$ converges to equilibrium with limit $\lim_{t\to\infty}(P(t))_{n,m}=\pi_m$.
\end{theorem}

\begin{proof}
	Since $Q=\nu(T-I)$, a probability row vector $\pi$ satisfies $\pi Q=0$ if and only if $\pi T=\pi$. Thus the CTMC and the uniformized chain have exactly the same stationary distributions, and the CTMC is ergodic if and only if $T$ is ergodic.
	
	In the PBF setting, ergodicity of $T$ is equivalent to the presence of a mass point at $x=1$ in $\psi_{1,1}$, and the corresponding boundary masses $m_{b,a}$ determine the stationary distribution (see \cite{BFM6}). This yields the stated criterion.
	
	To identify the limit from \eqref{eq:CT-KM}, decompose $\psi_{b,a}=\psi_{b,a}\!\restriction_{[0,1)}+m_{b,a}\delta_1$ and insert into \eqref{eq:CT-KM}. The contribution of the mass point at $x=1$ is time-independent, since $\e^{-\nu t(1-1)}=1$, and equals
	\[
	\sum_{a=1}^p\sum_{b=1}^q B_n^{(b)}(1)\,m_{b,a}\,A_m^{(a)}(1).
	\]
	For $x\in[0,1)$ one has $\e^{-\nu t(1-x)}\to 0$ as $t\to\infty$, and dominated convergence shows that the $[0,1)$ contribution vanishes. Therefore
	\[
	\lim_{t\to\infty}(P(t))_{n,m}
	=
	\sum_{a=1}^p\sum_{b=1}^q B_n^{(b)}(1)\,m_{b,a}\,A_m^{(a)}(1).
	\]
	Since the left-hand side converges to the stationary distribution, the right-hand side is independent of the initial state $n$ and equals $\pi_m$. Taking $n=n_0$ gives the announced formula.
\end{proof}

The preceding criteria distinguish recurrence, transience, and positive
recurrence.  When the spectral measures are absolutely continuous near
\(x=1\), their endpoint behavior also determines the precise rate at
which the transition probabilities decay.  The number \(\delta\) in the
next theorem is only the exponent describing this endpoint behavior; it
is not one of the parameters of a particular polynomial family.

\begin{theorem}[Spectral endpoint asymptotics]
\label{thm:bounded-endpoint-asymptotics}
Assume the hypotheses of Theorem~\ref{thm:CT-bounded}, and suppose that
none of the measures \(\psi_{b,a}\) has an atom at \(x=1\). Assume,
moreover, that there exist \(\varepsilon>0\), an endpoint exponent
\(\delta>-1\), and constants \(\omega_{b,a}\ge0\) such that, on
\((1-\varepsilon,1)\),
\[
\dd\psi_{b,a}(x)=w_{b,a}(x)\,\dd x,
\qquad
\lim_{x\uparrow1}
\frac{w_{b,a}(x)}{(1-x)^\delta}
=\omega_{b,a}.
\]
For \(n,m\in\Nzero\), put
\begin{equation}
\label{eq:bounded-endpoint-coefficient}
c_{n,m}
\coloneq
\sum_{a=1}^{p}\sum_{b=1}^{q}
B_n^{(b)}(1)\,\omega_{b,a}\,A_m^{(a)}(1).
\end{equation}
Then
\begin{align}
\label{eq:bounded-discrete-endpoint-limit}
\lim_{k\to\infty}
k^{\delta+1}(T^k)_{n,m}
&=
\Gamma(\delta+1)c_{n,m},
\\
\label{eq:bounded-continuous-endpoint-limit}
\lim_{t\to\infty}
t^{\delta+1}(\e^{tQ})_{n,m}
&=
\frac{\Gamma(\delta+1)}{\nu^{\delta+1}}c_{n,m}.
\end{align}
In particular, the continuous-time and discrete-time kernels have the
same endpoint exponent.
\end{theorem}

\begin{proof}
	Insert the spectral formula \eqref{eq:DT-KM}.  On
	\([0,1-\varepsilon]\), the contribution is exponentially small in \(k\).
	On \((1-\varepsilon,1)\), the polynomial product
	\(B_n^{(b)}(x)A_m^{(a)}(x)\) is continuous at \(1\). After decreasing
	\(\varepsilon\), if necessary, the endpoint assumption also gives
	\(w_{b,a}(x)\le C_{b,a}(1-x)^\delta\) on this interval. Under the change
	of variables \(u=k(1-x)\), the rescaled integrand is therefore dominated
	by a constant multiple of \(\e^{-u}u^\delta\), which is integrable because
	\(\delta>-1\). Dominated convergence now gives
\[
\lim_{k\to\infty}
k^{\delta+1}
\int_{1-\varepsilon}^{1}
B_n^{(b)}(x)A_m^{(a)}(x)x^k
w_{b,a}(x)\,\dd x
=
B_n^{(b)}(1)A_m^{(a)}(1)\omega_{b,a}
\int_0^\infty \e^{-u}u^\delta\,\dd u.
\]
The last integral is \(\Gamma(\delta+1)\).  Summing over \(a,b\)
proves \eqref{eq:bounded-discrete-endpoint-limit}.

For continuous time, insert \eqref{eq:CT-KM}. The contribution of
\([0,1-\varepsilon]\) is exponentially small in \(t\), and the change of
variables \(u=\nu t(1-x)\) gives
\[
\lim_{t\to\infty}
t^{\delta+1}
\int_{1-\varepsilon}^{1}
B_n^{(b)}(x)A_m^{(a)}(x)
\e^{-\nu t(1-x)}w_{b,a}(x)\,\dd x
=
\frac{\Gamma(\delta+1)}{\nu^{\delta+1}}
B_n^{(b)}(1)A_m^{(a)}(1)\omega_{b,a}.
\]
This proves \eqref{eq:bounded-continuous-endpoint-limit}.
\end{proof}

\begin{corollary}[Classification and first returns at a regular endpoint]
\label{coro:bounded-endpoint-classification}
Under the assumptions of
Theorem~\ref{thm:bounded-endpoint-asymptotics}, suppose that
\(\omega_{1,1}>0\).  The discrete-time chain with kernel \(T\), and its
continuous-time Poissonization, are null recurrent when
\(-1<\delta\le0\), and transient when \(\delta>0\).

In the critical case \(\delta=0\), one has, whenever \(c_{n,m}>0\),
\begin{align}
\label{eq:bounded-critical-Green}
\sum_{k=0}^{\infty}z^k(T^k)_{n,m}
&\sim
c_{n,m}\log\frac{1}{1-z},
&&z\uparrow1,
\\
\label{eq:bounded-critical-resolvent}
\bigl(R(\lambda)\bigr)_{n,m}
&\sim
\frac{c_{n,m}}{\nu}\log\frac{\nu}{\lambda},
&&\lambda\downarrow0.
\end{align}
If
\(\tau_m^+\coloneq\inf\{k\ge1:Y_k=m\}\) is the first positive return
time of the discrete-time chain to \(m\), and \(c_{m,m}>0\), then
\begin{equation}
\label{eq:bounded-critical-return-tail}
\Pr_m\{\tau_m^+>k\}
\sim
\frac{1}{c_{m,m}\log k}.
\end{equation}
In particular, \(\operatorname E_m[\tau_m^+]=+\infty\).
\end{corollary}

\begin{proof}
The normalization
\(A_0^{(a)}=\delta_{a,1}\) and
\(B_0^{(b)}=\delta_{b,1}\) gives
\(c_{0,0}=\omega_{1,1}>0\).
Near \(x=1\), the recurrence integral in
Theorem~\ref{thm:CT-recurrence-psi} behaves as
\(\omega_{1,1}\int_0^\varepsilon y^{\delta-1}\,\dd y\).
It diverges exactly when \(\delta\le0\).  Since there is no atom at
\(x=1\), Theorem~\ref{thm:bounded-stationary} excludes positive
recurrence, proving the classification.

When \(\delta=0\),
\eqref{eq:bounded-discrete-endpoint-limit} gives
\((T^k)_{n,m}\sim c_{n,m}/k\).  Its Abelian generating-function
form is \eqref{eq:bounded-critical-Green}; the continuous-time limit in
\eqref{eq:bounded-continuous-endpoint-limit}, followed by Laplace
transformation, gives \eqref{eq:bounded-critical-resolvent}.
Finally, if
\(F_m(z)=\operatorname E_m[z^{\tau_m^+}]\), the renewal identity is
\[
\sum_{k=0}^{\infty}z^k(T^k)_{m,m}
=\frac{1}{1-F_m(z)}.
\]
Consequently,
\(1-F_m(z)\sim
1/(c_{m,m}\log(1/(1-z)))\).
Moreover,
\((1-F_m(z))/(1-z)=
\sum_{k=0}^{\infty}\Pr_m\{\tau_m^+>k\}z^k\).
Since the coefficients form a nonincreasing sequence, the monotone
Hardy--Littlewood--Karamata Tauberian theorem gives
\eqref{eq:bounded-critical-return-tail}; see
\cite[Chapter~XIII, Section~5]{feller2}.
\end{proof}

\section{Unbounded generators and state-dependent uniformization}
\label{sec:unbounded}

Throughout this section, \(Q=[q_{n,m}]_{n,m\in\Nzero}\) is a conservative
banded \(Q\)-matrix on \(\Nzero\). Thus \(q_{n,m}\ge0\) for \(m\ne n\),
\(q_{n,n}=-\sum_{m\ne n}q_{n,m}\), and the exit rate from \(n\) is
\(\lambda_n\coloneq-q_{n,n}\). The generator is called unbounded when
\(\sup_n\lambda_n=+\infty\).

For a bounded generator, one scalar rate dominates all the exit rates and
global uniformization reduces the continuous-time process to one fixed
stochastic kernel. This reduction is unavailable when the exit rates are
unbounded. A rate depending on the current state is used instead. Before
that construction is developed, the next proposition explains
why scalar shifts of all leading truncations cannot provide a wide-band
unbounded Markov model.  In \cite{BFM_arXiv_2026}, every
truncation of a general unbounded banded matrix may instead use its own
\(N\)-dependent shift \(s_N\), leading to an unbounded Favard theorem.
The proposition below does not contradict that result: it identifies the
additional restriction produced when the matrix is required to be a conservative
Markov generator.

\subsection{The obstruction to scalar shifted factorizations}

For \(N\in\Nzero\), let \(Q_N\) denote the leading principal submatrix of
\(Q\), indexed by \(\{0,\ldots,N\}\). It is the generator of the process
killed upon its first exit from that set and satisfies
\begin{equation}
\label{eq:QN-subgenerator}
\sum_{j=0}^{N}(Q_N)_{i,j}
=-\sum_{j>N}q_{i,j}\le0,
\qquad 0\le i\le N.
\end{equation}

\begin{proposition}[Scalar-shift obstruction]
\label{pro:scalar-shift-obstruction}
Suppose that the exit rates of \(Q\) are unbounded. Assume that, for every
\(N\), there exists \(c_N>0\) such that
\(Q_N+c_NI_{N+1}\) admits a positive bidiagonal factorization. Then
\[
q_{n,m}=0\qquad\text{whenever}\qquad |n-m|\ge2.
\]
Consequently, an unbounded conservative generator satisfying the
scalar-shift hypothesis is necessarily a birth--death generator.
\end{proposition}

\begin{proof}
Put \(A_N\coloneq Q_N+c_NI_{N+1}\). Every nonnegative bidiagonal matrix is
totally nonnegative, and products of totally nonnegative matrices are
totally nonnegative by the Cauchy--Binet formula. Hence every minor of
\(A_N\) is nonnegative.

The diagonal entries of a positive bidiagonal product are positive.
Therefore
\[
c_N-\lambda_i=(A_N)_{i,i}>0,
\qquad 0\le i\le N.
\]
It follows that \(c_N>\max_{0\le i\le N}\lambda_i\). Since the exit rates
are unbounded, \(c_N\to+\infty\).

Fix \(n\ge0\) and \(k\ge2\), and take \(N\ge n+k\). Consider the minor of
\(A_N\) formed by rows \(n,n+1\) and columns \(n+1,n+k\). Total
nonnegativity gives
\[
0\le
\begin{vmatrix}
q_{n,n+1}&q_{n,n+k}\\
c_N-\lambda_{n+1}&q_{n+1,n+k}
\end{vmatrix}
=q_{n,n+1}q_{n+1,n+k}
-q_{n,n+k}(c_N-\lambda_{n+1}).
\]
All entries of \(Q\) appearing here are fixed when \(n\) and \(k\) are
fixed, whereas \(c_N\to+\infty\). Hence \(q_{n,n+k}=0\).

The lower-band argument is the reflected one. Total nonnegativity of the
minor formed by rows \(n+1,n+k\) and columns \(n,n+1\) gives
\[
0\le
\begin{vmatrix}
q_{n+1,n}&c_N-\lambda_{n+1}\\
q_{n+k,n}&q_{n+k,n+1}
\end{vmatrix}
=q_{n+1,n}q_{n+k,n+1}
-(c_N-\lambda_{n+1})q_{n+k,n}.
\]
Letting \(N\to\infty\) yields \(q_{n+k,n}=0\). Thus all diagonals beyond
the first sub- and superdiagonal vanish.
\end{proof}

\begin{remark}[Scope of the obstruction]
\label{rem:scalar-shift-scope}
Proposition~\ref{pro:scalar-shift-obstruction} is specific to conservative
generators with unbounded exit rates. The shifted-truncation Favard theorem
for general unbounded banded matrices \cite{BFM_arXiv_2026} remains
applicable outside the
Markov-generator sign and row-sum constraints. Within the present
probabilistic setting, however, the proposition shows that a different
construction is required to allow jumps of sizes
\(-p,\ldots,q\) with \(p>1\) or \(q>1\).
\end{remark}

\subsection{Local uniformization and state-dependent urn clocks}

\begin{definition}[State-dependent SBF uniformization]
\label{def:state-dependent-SBF-uniformization}
A conservative generator \(Q\) is called \emph{state-dependently
SBF-uniformizable} if there exist a positive diagonal matrix
\(V=\diag(v_0,v_1,\ldots)\), with \(\inf_n v_n>0\), and a banded
stochastic matrix \(T\) admitting a positive stochastic bidiagonal
factorization such that
\begin{equation}
\label{eq:local-uniformization-Q}
Q=V(T-I).
\end{equation}
The number \(v_n\) is the local clock rate at state \(n\). Equivalently,
\(q_{n,m}=v_nt_{n,m}\) for \(m\ne n\) and
\(\lambda_n=v_n(1-t_{n,n})\le v_n\).
\end{definition}

Starting instead from a conservative generator \(Q\), any choice
\(v_n\ge\lambda_n\) gives a stochastic matrix
\(T=I+V^{-1}Q\). Definition~\ref{def:state-dependent-SBF-uniformization}
selects those choices for which this single embedded kernel has a positive
stochastic bidiagonal factorization. The strict inequality
\(v_n>\lambda_n\) gives a positive probability of a virtual stay at a
local-clock event.

\begin{proposition}[Local-clock and urn realization]
\label{pro:local-clock-urn-realization}
Suppose that \(Q\) satisfies
\eqref{eq:local-uniformization-Q}, where
\[
T=\widehat L_1\cdots\widehat L_p
\widehat U_q\cdots\widehat U_1
\]
is a positive stochastic bidiagonal factorization. Then the minimal
\(Q\)-process has the following realization. At state \(n\), a clock of
rate \(v_n\) is run. When it rings, the ordered sequence of \(p\)
death-or-stay experiments and \(q\) birth-or-stay experiments associated
with the displayed factorization is executed once. The state at the end of
the complete sequence is the next state of the embedded event chain.

If the factor entries are rational, every experiment can be implemented by
a finite two-color urn. If the local-clock construction is non-explosive,
the resulting process is conservative. Moreover, positivity of the
nontrivial bidiagonal entries of the factors makes both the discrete event chain and the
continuous-time process irreducible.
\end{proposition}

\begin{proof}
During a time interval of length \(\dd t\), the local clock rings with
probability \(v_n\dd t+o(\dd t)\). Conditional on a ring, the complete urn
sequence has transition kernel \(T\). Therefore, for \(m\ne n\), the
transition rate from \(n\) to \(m\) is \(v_nt_{n,m}=q_{n,m}\), while the
diagonal rate is \(v_n(t_{n,n}-1)=q_{n,n}\). This proves that the
construction has generator \(Q\).

The ordered binary and finite-urn realizations follow from
Theorem~\ref{thm:finite-urn-correspondence}. A positive stochastic
bidiagonal factorization gives positive nearest-neighbor paths in both
directions, as in Proposition~\ref{pro:CT-irreducible}; multiplication by
the positive rates \(v_n\) does not change the communication graph.
\end{proof}

The following construction gives non-explosive unbounded examples of every
bandwidth with exit rate \(n+1\).

\begin{proposition}[Non-explosive local-clock construction with linear exit rates]
\label{pro:canonical-unbounded-local-clock}
Let \(T=[t_{n,m}]_{n,m\in\Nzero}\) be an irreducible stochastic
\((p,q)\)-band matrix with \(t_{n,n}<1\) for every \(n\). Define
\begin{equation}
\label{eq:canonical-local-rates}
v_n\coloneq\frac{n+1}{1-t_{n,n}},
\qquad
V\coloneq\diag(v_0,v_1,\ldots),
\qquad
Q^{\langle V\rangle}\coloneq V(T-I).
\end{equation}
Then \(Q^{\langle V\rangle}\) is an irreducible non-explosive
\((p,q)\)-band generator whose exit rate from \(n\) is exactly \(n+1\).
In particular, it is unbounded. If \(T\) admits a positive stochastic
bidiagonal factorization, this generator has the local urn realization of
Proposition~\ref{pro:local-clock-urn-realization}.
\end{proposition}

\begin{proof}
The off-diagonal entries of \(Q^{\langle V\rangle}\) are
\(v_nt_{n,m}\ge0\), and its rows sum to zero. Its exit rate is
\[
v_n(1-t_{n,n})=n+1.
\]
Thus the generator is unbounded and has the same communication graph and
the same bandwidth as \(T\).

For non-explosion, take the Lyapunov function \(F(n)=n+1\). Since an
upward transition has size at most \(q\), while downward transitions only
decrease \(F\),
\[
\begin{aligned}
(Q^{\langle V\rangle}F)(n)
&=v_n\sum_m t_{n,m}(m-n)\\
&\le qv_n\sum_{m>n}t_{n,m}
\le qv_n(1-t_{n,n})
=q(n+1)=qF(n).
\end{aligned}
\]
The standard Lyapunov non-explosion criterion for countable-state
continuous-time chains now proves conservativity of the minimal process;
see, for example, \cite{bremaud}.
At every visit to \(n\), the number of virtual stays before a genuine
departure is geometric with success probability \(1-t_{n,n}>0\), and is
therefore finite almost surely.  Since genuine departures cannot
accumulate in finite time and every \(v_n\) is finite, the local urn-clock
realization cannot accumulate virtual events either.
\end{proof}

\subsection{Resolvents and discounted urn cycles}

State-dependent uniformization does not express the semigroup as a scalar
Poisson mixture of the powers of \(T\), because the clock rate changes
with the state. It does, however, give an exact renewal expansion for every
resolvent.

For \(\lambda>0\), put
\begin{equation}
\label{eq:local-resolvent-kernel}
D_\lambda
\coloneq\diag\left(\frac{v_n}{\lambda+v_n}\right)_{n\ge0},
\qquad
W_\lambda
\coloneq\diag\left(\frac{1}{\lambda+v_n}\right)_{n\ge0},
\qquad
K_\lambda\coloneq D_\lambda T.
\end{equation}
The matrix \(K_\lambda\) is banded, nonnegative, and substochastic; its
\(n\)-th row sums to \(v_n/(\lambda+v_n)<1\).

\begin{theorem}[Resolvent under state-dependent uniformization]
\label{thm:local-uniformization-resolvent}
Assume that the local-clock process generated by
\(Q=V(T-I)\) is non-explosive. For every \(\lambda>0\), its minimal
resolvent satisfies, entrywise,
\begin{equation}
\label{eq:local-resolvent-factorization}
R_Q(\lambda)
\coloneq(\lambda I-Q)^{-1}
=(I-K_\lambda)^{-1}W_\lambda
=\sum_{k=0}^{\infty}K_\lambda^kW_\lambda.
\end{equation}
In particular,
\begin{equation}
\label{eq:local-resolvent-entry}
\bigl(R_Q(\lambda)\bigr)_{n,m}
=\frac{1}{\lambda+v_m}
\sum_{k=0}^{\infty}(K_\lambda^k)_{n,m}.
\end{equation}
\end{theorem}

\begin{proof}
The algebraic identity behind the formula is
\[
\lambda I-Q
=\lambda I+V-VT
=(\lambda I+V)(I-D_\lambda T)
=(\lambda I+V)(I-K_\lambda).
\]
The entrywise inverse follows from first-event decomposition. At state
\(n\), the local event clock of rate \(v_n\)
competes with the exponential discount clock of rate \(\lambda\). The
local event occurs first with probability
\(v_n/(\lambda+v_n)\), which gives the factor \(D_\lambda\); conditional
on survival, the next state is selected according to \(T\). The expected
discounted time spent before the first of the two clocks rings is
\((\lambda+v_n)^{-1}\), which gives \(W_\lambda\).

After \(k\) local events occurring before killing, the contribution to the discounted
occupation kernel is \(K_\lambda^kW_\lambda\). All its entries are
nonnegative, so monotone convergence allows summation over \(k\). The
result is the minimal resolvent. Non-explosion makes the minimal process
conservative and identifies this resolvent with the resolvent of its
transition semigroup.
\end{proof}

For each fixed value of \(\lambda\), the same kernel also has a positive
bidiagonal factorization.

\begin{proposition}[Positive bidiagonal factorization of the killed kernel]
\label{pro:local-kernel-PBF}
Write the positive stochastic bidiagonal factorization of \(T\) as
\[
T=M_1\cdots M_d,
\qquad d=p+q,
\]
where \(M_1,\ldots,M_p\) are the lower factors and
\(M_{p+1},\ldots,M_d\) are the upper factors in the prescribed order.
For every fixed \(\lambda>0\), the substochastic kernel \(K_\lambda\)
admits the positive bidiagonal factorization
\begin{equation}
\label{eq:local-kernel-PBF}
K_\lambda
=(D_\lambda M_1D_\lambda^{-1})\cdots
(D_\lambda M_{d-1}D_\lambda^{-1})(D_\lambda M_d).
\end{equation}
It has the same ordered lower/upper type as \(T\).
\end{proposition}

\begin{proof}
Since \(\inf_n v_n>0\), both \(D_\lambda\) and
\(D_\lambda^{-1}\) are bounded positive diagonal matrices. Diagonal
conjugation preserves bidiagonality and positivity of all nontrivial
bidiagonal entries.
The last factor \(D_\lambda M_d\) is again positive upper bidiagonal.
The diagonal matrices between consecutive factors cancel, leaving
\(D_\lambda M_1\cdots M_d=D_\lambda T=K_\lambda\).
\end{proof}

Consequently, the bounded Favard theorem \cite{BFM1} applies separately to
\(K_\lambda\). Denote its mixed polynomial families and positive matrix of
measures by
\[
A_m^{(a,\lambda)},\qquad
B_n^{(b,\lambda)},\qquad
\Psi^{(\lambda)}
=\bigl(\psi_{b,a}^{(\lambda)}\bigr)_{
1\le b\le q,\,1\le a\le p}.
\]
Combining its spectral formula with
\eqref{eq:local-resolvent-entry} gives the following positive
representation of each resolvent entry:
\begin{equation}
\label{eq:local-resolvent-spectral-Abel}
\bigl(R_Q(\lambda)\bigr)_{n,m}
=\frac{1}{\lambda+v_m}
\lim_{r\uparrow1}
\sum_{a=1}^{p}\sum_{b=1}^{q}
\int_{0}^{1}
\frac{
B_n^{(b,\lambda)}(x)A_m^{(a,\lambda)}(x)}
{1-rx}\,
\dd\psi_{b,a}^{(\lambda)}(x).
\end{equation}
The Abel parameter \(r\) permits the geometric summation; the
limit exists entrywise by
\eqref{eq:local-resolvent-entry} and monotone convergence at the matrix
level.

Formula~\eqref{eq:local-resolvent-spectral-Abel} is an exact positive
representation for each \(\lambda>0\), but its kernel, polynomials, and
measures depend on \(\lambda\); it is not one time-independent
diagonalization of \(Q\).

\subsection{The fixed spectral data: potentials, hitting, and recurrence}

Although the discounted kernels depend on \(\lambda\), the
non-discounted potential is controlled by the single stochastic matrix
\(T\). Let
\[
G_T(n,m)\coloneq\sum_{k=0}^{\infty}(T^k)_{n,m},
\qquad
U_Q(n,m)\coloneq
\int_0^\infty\Pr_n\{X_t=m\}\,\dd t,
\]
allowing the value \(+\infty\).

\begin{theorem}[Potential kernel and the fixed spectral matrix]
\label{thm:local-potential-spectral}
Assume that \(Q=V(T-I)\) is non-explosive. Then
\begin{equation}
\label{eq:local-potential-Green}
U_Q(n,m)=\frac{G_T(n,m)}{v_m}.
\end{equation}
Suppose in addition that \(T\) admits a positive stochastic bidiagonal
factorization, and let \(A_m^{(a)},B_n^{(b)}\), and
\(\Psi=[\psi_{b,a}]\) be the fixed spectral data of \(T\) from
Corollary~\ref{cor:urn-protocol-spectral}. Then
\begin{equation}
\label{eq:local-potential-spectral-Abel}
U_Q(n,m)
=\frac1{v_m}
\lim_{r\uparrow1}
\sum_{a=1}^{p}\sum_{b=1}^{q}
\int_0^1
\frac{B_n^{(b)}(x)A_m^{(a)}(x)}
{1-rx}\,
\dd\psi_{b,a}(x).
\end{equation}
If the limiting integrals are absolutely convergent, the Abel limit may be
replaced by \(r=1\).
\end{theorem}

\begin{proof}
Set \(\lambda=0\) in the first-event decomposition used in
Theorem~\ref{thm:local-uniformization-resolvent}. Each visit of the event
chain to \(m\) is followed by a mean local-clock interval \(1/v_m\).
Summing over all event-chain visits gives
\eqref{eq:local-potential-Green}. Equivalently, the same identity follows
by monotone convergence from \eqref{eq:local-resolvent-entry} as
\(\lambda\downarrow0\).

For \(0<r<1\), sum the discrete spectral representation
\eqref{eq:DT-KM} with weights \(r^k\). The geometric series gives
\[
\sum_{k=0}^{\infty}r^k(T^k)_{n,m}
=
\sum_{a=1}^{p}\sum_{b=1}^{q}
\int_0^1
\frac{B_n^{(b)}(x)A_m^{(a)}(x)}
{1-rx}\,
\dd\psi_{b,a}(x).
\]
Letting \(r\uparrow1\) and using monotone convergence on the nonnegative
matrix entries yields \(G_T(n,m)\). Multiplication by \(1/v_m\) proves
\eqref{eq:local-potential-spectral-Abel}. Absolute convergence permits
passage to \(r=1\) inside each integral.
\end{proof}

For the reference state \(0\), the normalization
\(A_0^{(a)}=\delta_{a,1}\) and \(B_0^{(b)}=\delta_{b,1}\) removes all
cross terms. Hence
\begin{equation}
\label{eq:local-origin-potential}
U_Q(0,0)
=\frac1{v_0}
\int_0^1\frac{\dd\psi_{1,1}(x)}{1-x},
\end{equation}
where the integral is allowed to be infinite. Thus the recurrence
criterion depends on the same positive measure even though the clock rate
depends on the state.

\begin{theorem}[Recurrence, hitting times, and local rates]
\label{thm:local-recurrence-hitting}
Assume that \(T\) is irreducible and that the local-clock process
\(Q=V(T-I)\) is non-explosive.

\begin{enumerate}
\item The continuous-time process is recurrent if and only if the event
chain with kernel \(T\) is recurrent. If \(T\) has a positive stochastic
bidiagonal factorization, these equivalent conditions are characterized by
\[
\int_0^1\frac{\dd\psi_{1,1}(x)}{1-x}=+\infty.
\]

\item For \(i,j\in\Nzero\), the eventual hitting probability is independent
of the local rates:
\[
\Pr_i\{\tau_j<\infty\}
=\Pr_i\{\text{the event chain ever reaches }j\}.
\]
In the transient case,
\[
\Pr_i\{\tau_j<\infty\}
=\frac{U_Q(i,j)}{U_Q(j,j)}
=\frac{G_T(i,j)}{G_T(j,j)}.
\]

\item For every \(\lambda>0\),
\begin{equation}
\label{eq:local-Laplace-hitting-time}
\operatorname E_i\!\left[
\e^{-\lambda\tau_j}\ind_{\{\tau_j<\infty\}}
\right]
=
\frac{\bigl(R_Q(\lambda)\bigr)_{i,j}}
{\bigl(R_Q(\lambda)\bigr)_{j,j}}.
\end{equation}
Consequently, whenever the corresponding moments are finite, they are
obtained by differentiating this quotient at \(\lambda=0\).
\end{enumerate}
\end{theorem}

\begin{proof}
The continuous path and the event chain visit exactly the same sequence of
states, apart from virtual stays. Non-explosion ensures that this
correspondence persists for all physical times. This proves the equivalence
of recurrence and the equality of eventual hitting events. The spectral
criterion follows from \eqref{eq:local-origin-potential}.

In the transient case, split the total occupation time of \(j\) at the
first hitting time \(\tau_j\) and apply the strong Markov property. This
gives \(U_Q(i,j)=\Pr_i\{\tau_j<\infty\}U_Q(j,j)\). Substitution of
\eqref{eq:local-potential-Green} cancels the common factor \(1/v_j\) and
gives the ratio of the two discrete Green entries.

The same argument with the discount factor
\(\e^{-\lambda t}\) gives
\[
\bigl(R_Q(\lambda)\bigr)_{i,j}
=
\operatorname E_i\!\left[
\e^{-\lambda\tau_j}\ind_{\{\tau_j<\infty\}}
\right]
\bigl(R_Q(\lambda)\bigr)_{j,j},
\]
which proves \eqref{eq:local-Laplace-hitting-time}.
\end{proof}

\begin{theorem}[Invariant measures and positive recurrence]
\label{thm:local-invariant-measure}
Let \(Q=V(T-I)\), where \(T\) is irreducible. A positive row vector
\(\mu\) is invariant for \(T\) if and only if the row vector
\(\pi=\mu V^{-1}\), given by
\begin{equation}
\label{eq:local-invariant-measure}
\pi_n=\frac{\mu_n}{v_n},
\end{equation}
is invariant for \(Q\).

Assume in addition that the local-clock process is non-explosive and
recurrent. Then it is positive recurrent if and only if
\begin{equation}
\label{eq:local-positive-recurrence}
Z_V\coloneq\sum_{n=0}^{\infty}\frac{\mu_n}{v_n}<\infty,
\end{equation}
where \(\mu\) is the invariant measure of the recurrent event chain,
unique up to a positive scalar multiple. In that case the stationary
distribution is
\[
\overline\pi_n
=\frac{\mu_n/v_n}{Z_V}.
\]
If \(Z_V=+\infty\), the continuous-time process is null recurrent.
\end{theorem}

\begin{proof}
The identities
\[
(\mu V^{-1})Q
=\mu(T-I),
\qquad
(\pi V)T-\pi V
=\pi Q
\]
prove the two implications. For an irreducible recurrent
countable-state process, a positive invariant measure is unique up to
scale, and positive recurrence is equivalent to its having finite total
mass. Formula~\eqref{eq:local-positive-recurrence} is precisely that total
mass.
\end{proof}

\section{Grouping banded chains and generators as quasi-birth-and-death processes}
\label{sec:QBD-reblocking}

The construction in this section groups consecutive scalar states into
levels, with the position inside a level used as the phase. It does not
depend on boundedness or on the choice of clock. This grouping writes a
scalar banded kernel as a level-dependent block-tridiagonal process. The
mixed spectral representation remains valid even when the block matrix
cannot be symmetrized. The same grouping applies to a discrete-time
kernel, its bounded Poissonization, and the generators and resolvents of
the state-dependent processes in the preceding section.

A quasi-birth-and-death (QBD) chain is a two-index Markov chain whose states
are described by a level and one of finitely many phases, and whose level is
skip-free: in one step it can decrease by one, remain unchanged, or
increase by one. Equivalently, its transition matrix is block
tridiagonal. Standard matrix-analytic methods for QBD processes are
developed in \cite{LatoucheRamaswami1999}. Their relation with
matrix-valued orthogonal polynomials and explicit spectral formulas was
studied, among other places, in
\cite{GrunbaumDelaIglesiaQBD2008}; stochastic block factorizations,
Darboux transformations, and urn models were considered in
\cite{GrunbaumDelaIglesiaQBDDarboux2019}. QBD processes constructed from
multivariate orthogonal polynomials, including both discrete- and
continuous-time examples, appear in
\cite{FernandezDelaIglesiaQBD2021}.

Block tridiagonality alone does not, however, place a QBD within the usual
positive-definite matrix-valued orthogonality theory. Write its diagonal,
up-level, and down-level blocks as \(D_K,U_K,L_K\), respectively. That
theory requires positive definite matrices \(\Pi_K\) satisfying
\begin{equation}
\label{eq:QBD-block-symmetrizer}
	\Pi_KD_K=D_K^{\mathsf T}\Pi_K,
	\qquad
	\Pi_KU_K=L_{K+1}^{\mathsf T}\Pi_{K+1},
	\qquad K\ge0.
\end{equation}
Equivalently, if \(\Pi_K=R_K^{\mathsf T}R_K\), the block diagonal
similarity with diagonal blocks \(R_K\) transforms the QBD matrix into a
symmetric block Jacobi matrix. Conditions
\eqref{eq:QBD-block-symmetrizer} are explicit in the matrix-valued
orthogonal-polynomial approach to QBDs
\cite{DetteEtAlQBD2006,DelaIglesiaQBDInvariant2011}; they are not a
consequence of stochasticity or of equal lower and upper bandwidths.
Non-symmetric matrices of measures provide a different extension of
matrix orthogonality \cite{ZygmuntNonSymmetric2016}.

The mixed-type construction used here instead provides two
polynomial families and a \(q\times p\) matrix whose entries are positive
measures. Even when \(p=q\), entrywise positivity does not assert that this
matrix of measures is symmetric or positive definite in the Hermitian
matrix sense. Consequently, the mixed spectral formulas do not require
\eqref{eq:QBD-block-symmetrizer}.

The following proposition connects that theory with the scalar banded
kernels considered here. The map is a bijective relabelling, not a
lumping: the phase specifies the position of the original scalar state
inside its block.

\begin{proposition}[Grouping consecutive states and the mixed spectral representation]
\label{pro:banded-kernel-QBD}
	Let \(T=[T_{n,m}]_{n,m\in\Nzero}\) be a stochastic matrix of bandwidth
	\((p,q)\), so that \(T_{n,m}=0\) whenever
	\(m-n\notin\{-p,\ldots,q\}\). Set \(r\coloneq\max\{p,q\}\) and
	write \(n=rK+i\), where \(K\in\Nzero\) and
	\(i\in\{0,\ldots,r-1\}\).
	Let \((X_\ell)_{\ell\in\Nzero}\) be the discrete-time Markov chain on
	\(\Nzero\) with transition matrix \(T\); thus
	\((X_\ell)_{\ell\in\Nzero}\) is a stochastic process satisfying
	\(\Pr\{X_{\ell+1}=m\mid X_\ell=n\}=T_{n,m}\) for
	\(n,m,\ell\in\Nzero\).
Define the two-component process
	\begin{equation}
	Z_\ell\coloneq
	\begin{bNiceMatrix}K_\ell&I_\ell\end{bNiceMatrix}
	\coloneq
	\begin{bNiceMatrix}
	\left\lfloor\frac{X_\ell}{r}\right\rfloor&
	X_\ell-r\left\lfloor\frac{X_\ell}{r}\right\rfloor
	\end{bNiceMatrix}.
	\end{equation}
	Thus \(Z_\ell\) takes values in the semi-infinite strip
	\[\mathcal E_r\coloneq\Nzero\times\{0,\ldots,r-1\}.\]
	Under the bijection \(n\leftrightarrow(K,i)\), \(Z_\ell\) is a
	level-dependent QBD with \(r\) phases and transition probabilities
	\[\Pr\{Z_{\ell+1}=(L,j)\mid Z_\ell=(K,i)\}=T_{rK+i,rL+j}.\]
	More precisely, put \(c_s(n):=T_{n,n+s}\), with \(c_s(n):=0\) if
	\(s\notin\{-p,\ldots,q\}\) or \(n+s<0\).
	For \(\varepsilon\in\{-1,0,1\}\), define the \(r\times r\) block
	\begin{equation}
	\label{eq:QBD-blocks-from-band}
	\left(\mathsf T_K^{\varepsilon}\right)_{i,j}
	\coloneq
	c_{\varepsilon r+j-i}(rK+i),
	\qquad 0\le i,j\le r-1.
	\end{equation}
	Then \(\mathsf T_0^-=0\) and the blocked matrix is
	\begin{equation}
	\label{eq:QBD-block-tridiagonal}
	T=
	\begin{bNiceMatrix}[cell-space-limits=4pt,margin=5pt]
	\mathsf T_0^0 & \mathsf T_0^+ & 0_{r\times r} &
	\Cdots[shorten-start=2pt,shorten-end=-20pt] & \phantom{0}  \\[-5pt]
	\mathsf T_1^- & \mathsf T_1^0 & \mathsf T_1^+ & \Ddots[shorten-start=-0pt,shorten-end=5pt] &  \phantom{0} \\
	0_{r\times r} & \mathsf T_2^- & \mathsf T_2^0 & \mathsf T_2^+ &  \\
	\Vdots[shorten-start=0pt,shorten-end=-15pt] &
	\Ddots[shorten-start=0pt,shorten-end=-35pt] &
	\Ddots[shorten-start=0pt,shorten-end=-32pt] &
	\Ddots[shorten-start=0pt,shorten-end=-30pt] &	\Ddots[shorten-start=0pt,shorten-end=-0pt] &
\\
	\phantom{0} & \phantom{\mathsf T_2^-} &
	\phantom{\mathsf T_2^0} & \phantom{\mathsf T_2^+} & \phantom{0}
	\end{bNiceMatrix}.
	\end{equation}
	Each block row is stochastic in the sense that
	\((\mathsf T_K^-+\mathsf T_K^0+\mathsf T_K^+)\one_r=\one_r\).
	The QBD is level-homogeneous precisely when the three blocks in
	\eqref{eq:QBD-blocks-from-band} are independent of \(K\); away from the
	boundary, this amounts to \(T_{n+r,m+r}=T_{n,m}\).

	Suppose, in addition, that \(T\) satisfies the hypotheses of
	Corollary~\ref{cor:urn-protocol-spectral}. Define the \(r\)-component columns
	\begin{equation}
	\label{eq:QBD-blocked-polynomials}
	\boldsymbol B_K^{(b)}(x)
	\coloneq
	\begin{bNiceMatrix}
	B_{rK}^{(b)}(x)\\
	\Vdots\\
	B_{rK+r-1}^{(b)}(x)
	\end{bNiceMatrix},
	\qquad
	\boldsymbol A_L^{(a)}(x)
	\coloneq
	\begin{bNiceMatrix}
	A_{rL}^{(a)}(x)\\
	\Vdots\\
	A_{rL+r-1}^{(a)}(x)
	\end{bNiceMatrix}.
	\end{equation}
	For every \(s\in\Nzero\), the \(r\times r\) block of \(s\)-step
	transition probabilities from level \(K\) to level \(L\) is
	\begin{equation}
	\label{eq:QBD-block-spectral-discrete}
	\left(T^s\right)_{K,L}
	=
	\sum_{a=1}^{p}\sum_{b=1}^{q}
	\int_0^1
	\boldsymbol B_K^{(b)}(x)
	\left(\boldsymbol A_L^{(a)}(x)\right)^{\mathsf T}
	x^s\,\dd\psi_{b,a}(x).
	\end{equation}
	If \(Q=\nu(T-I)\), the corresponding continuous-time QBD has the block
	semigroup representation
	\begin{equation}
	\label{eq:QBD-block-spectral-continuous}
	\left(\e^{tQ}\right)_{K,L}
	=
	\sum_{a=1}^{p}\sum_{b=1}^{q}
	\int_0^1
	\boldsymbol B_K^{(b)}(x)
	\left(\boldsymbol A_L^{(a)}(x)\right)^{\mathsf T}
	\e^{-\nu t(1-x)}\,\dd\psi_{b,a}(x),
	\qquad
	Q=\nu(T-I).
	\end{equation}
\end{proposition}

\begin{proof}
	Every state has a unique representation \(n=rK+i\). If \(L\ge K+2\),
	the smallest possible scalar displacement from level \(K\) to level \(L\)
	is \(r(K+2)-(rK+r-1)=r+1>q\).
	If \(L\le K-2\), the largest possible displacement is
	\((r(K-2)+r-1)-rK=-r-1<-p\).
	Thus transitions between non-neighboring levels vanish. For
	\(L=K+\varepsilon\), the displacement from \(rK+i\) to \(rL+j\) is
	\(\varepsilon r+j-i\), which gives
	\eqref{eq:QBD-blocks-from-band}. Stochasticity is unchanged because the
	map \(n\leftrightarrow(K,i)\) is a bijection. Finally,
	\eqref{eq:QBD-block-spectral-discrete} is obtained by collecting the
	\(r^2\) scalar identities \eqref{eq:DT-KM} with
	\(n=rK+i\) and \(m=rL+j\). The continuous-time formula follows in the
	same way from \eqref{eq:CT-KM}.
\end{proof}

We now apply the same Schur-complement construction directly to the
block-tridiagonal kernel obtained above.  It gives the matrix generating
function for the first return to a level, including the return phase.

Let \(T^{[K]}\) be the transition matrix obtained by retaining the levels
\(K,K+1,\ldots\); a transition below \(K\) stops the chain.  Define the
Green matrix at the first retained level by
\begin{equation}
\label{eq:level-tail-Green-definition}
\mathcal G_K^\uparrow(z)
\coloneq
\left[
\left(I-zT^{[K]}\right)^{-1}
\right]_{K,K},
\qquad 0\le z<1.
\end{equation}
Thus its \((i,j)\)-entry counts, with the factor \(z^n\), all \(n\)-step
paths which start in phase \(i\) of level \(K\), never go below \(K\), and
end in phase \(j\) of that level.

\begin{proposition}[QBD return law from the continued fraction]
\label{pro:level-return-CF}
Applied to the blocks
\(\mathsf T_K^-\), \(\mathsf T_K^0\), and \(\mathsf T_K^+\), the Schur
recursion gives, for
every \(K\ge0\),
\begin{equation}
\label{eq:level-return-CF-recursion}
\mathcal G_K^\uparrow(z)
=
\left[
I_r-z\mathsf T_K^0
-z^2\mathsf T_K^+
\mathcal G_{K+1}^\uparrow(z)
\mathsf T_{K+1}^-
\right]^{-1}.
\end{equation}
Let \(\mathcal F_K^{\uparrow,+}(z)\) be the matrix generating function of
the first positive return to level \(K\), with paths stopped if they go
below \(K\).  Then
\begin{equation}
\label{eq:level-first-return-CF}
\mathcal F_K^{\uparrow,+}(z)
=
z\mathsf T_K^0
+z^2\mathsf T_K^+
\mathcal G_{K+1}^\uparrow(z)
\mathsf T_{K+1}^-
=
I_r-\left(\mathcal G_K^\uparrow(z)\right)^{-1}.
\end{equation}
In particular, the continued fraction determines the return probabilities
and the return phases.
\end{proposition}

\begin{proof}
The specialization can be verified directly.  Decompose the state space of
\(T^{[K]}\) into level \(K\) and the tail \(K+1,K+2,\ldots\).
In \(I-zT^{[K]}\), the diagonal block belonging to level \(K\) is
\(I_r-z\mathsf T_K^0\).  The only block from level \(K\) to the tail is
\(-z\mathsf T_K^+\), and the only block returning from the tail to level
\(K\) is \(-z\mathsf T_{K+1}^-\).  The relevant corner of the inverse of
the tail operator is, by definition,
\[
\left[\left(I-zT^{[K+1]}\right)^{-1}\right]_{K+1,K+1}
=\mathcal G_{K+1}^\uparrow(z).
\]
Taking the Schur complement of the tail therefore gives
\[
\left(\mathcal G_K^\uparrow(z)\right)^{-1}
=
I_r-z\mathsf T_K^0
-z^2\mathsf T_K^+
\mathcal G_{K+1}^\uparrow(z)
\mathsf T_{K+1}^-,
\]
which proves \eqref{eq:level-return-CF-recursion}.

For the first positive return to level \(K\), there are exactly two
possibilities.  The chain either remains in level \(K\) for one step,
contributing \(z\mathsf T_K^0\), or it moves to level \(K+1\), follows an
arbitrary path in the retained tail, and returns from level \(K+1\) to
level \(K\).  The latter paths contribute
\[
z^2\mathsf T_K^+
\mathcal G_{K+1}^\uparrow(z)
\mathsf T_{K+1}^-.
\]
This proves the first equality in
\eqref{eq:level-first-return-CF}.  Finally, the matrix renewal identity
\(
\mathcal G_K^\uparrow
=I_r+\mathcal F_K^{\uparrow,+}\mathcal G_K^\uparrow
\)
gives the second equality.
\end{proof}

The resulting nested expression is the matrix continued fraction
\begin{equation}
\label{eq:level-return-CF-expanded}
\mathcal G_K^\uparrow(z)
=
\left[
I_r-z\mathsf T_K^0
-z^2\mathsf T_K^+
\left[
I_r-z\mathsf T_{K+1}^0
-z^2\mathsf T_{K+1}^+
\left[\cdots\right]^{-1}
\mathsf T_{K+2}^-
\right]^{-1}
\mathsf T_{K+1}^-
\right]^{-1}.
\end{equation}
The order of the matrix products in this expression is essential.

The finite fractions also have a direct probabilistic meaning in the
present QBD.  If only the levels
\(K,\ldots,N\) are retained, start with
\(\mathcal G_N^{[K,N]}(z)=
\left(I_r-z\mathsf T_N^0\right)^{-1}\)
and evaluate \eqref{eq:level-return-CF-recursion} backwards.  The result is
the Green matrix of the chain stopped when it leaves those levels.  These
finite Green matrices sum the weights of nested classes of paths; as the
upper level increases, the classes exhaust all paths which never go below
\(K\).  Hence they converge entry by entry to
\(\mathcal G_K^\uparrow(z)\) for \(0\le z<1\).  This is the direct QBD
interpretation of the convergence.

Let \(Q\) be a stable conservative, non-explosive block-tridiagonal
\(Q\)-matrix, and let \(Q^{[K]}\) be the minimal generator killed upon
leaving the levels \(K,K+1,\ldots\). For \(\lambda>0\), define
\[
\mathcal R_K^\uparrow(\lambda)
\coloneq
\int_0^\infty e^{-\lambda t}
\bigl[P_t^{[K]}\bigr]_{K,K}\,\dd t.
\]
Then
\begin{equation}
\label{eq:level-resolvent-CF-recursion}
\mathcal R_K^\uparrow(\lambda)
=
\left[
\lambda I_r-\mathsf Q_K^0
-\mathsf Q_K^+
\mathcal R_{K+1}^\uparrow(\lambda)
\mathsf Q_{K+1}^-
\right]^{-1},
\qquad \lambda>0.
\end{equation}
Indeed, the finite process killed on leaving the levels
\(K,\ldots,M\) satisfies this identity by an ordinary finite-dimensional
Schur complement. Its killed resolvents increase entrywise to the minimal
tail resolvents as \(M\to\infty\); passing to the limit proves
\eqref{eq:level-resolvent-CF-recursion}, without invoking an unspecified
infinite-dimensional realization. This is the continuous-time resolvent
counterpart of the discrete-time continued fraction.

\begin{remark}[The two continued fractions use different data]
\label{rem:coarse-versus-factor-resolved-CF}
The fraction in this subsection starts from the transition kernel \(T\) and
therefore uses the three combined level blocks
\(\mathsf T_K^-\), \(\mathsf T_K^0\), and \(\mathsf T_K^+\).  The
factor-resolved fraction of
Subsection~\ref{subsec:factor-resolved-return-CF} uses more information:
its coefficients are the \(p+q\) individual death-or-stay and
birth-or-stay factors.  The QBD application
computes returns and arrival phases by levels; the general application in
Subsection~\ref{subsec:factor-resolved-return-CF} computes the first-return
law of the original scalar chain directly.
\end{remark}

To display the distinction from ordinary matrix orthogonality, collect the
columns in \eqref{eq:QBD-blocked-polynomials} as
\begin{align*}
\mathbb B_K(x)
&\coloneq
\begin{bNiceMatrix}
\boldsymbol B_K^{(1)}(x)&\cdots&\boldsymbol B_K^{(q)}(x)
\end{bNiceMatrix}
\in\mathbb R^{r\times q},
\\
\mathbb A_L(x)
&\coloneq
\begin{bNiceMatrix}
\boldsymbol A_L^{(1)}(x)&\cdots&\boldsymbol A_L^{(p)}(x)
\end{bNiceMatrix}
\in\mathbb R^{r\times p}.
\end{align*}
Then \eqref{eq:QBD-block-spectral-discrete} is the rectangular mixed
formula
\begin{equation}
\label{eq:QBD-compact-mixed-spectral}
	\left(T^s\right)_{K,L}
	=
	\int_0^1
	\mathbb B_K(x)\,x^s\,\dd\Psi(x)\,
	\mathbb A_L(x)^{\mathsf T},
\end{equation}
where the integral is understood entrywise. If \(p=q=r\), all three
matrices in the integrand are square, but in general
\(\mathbb A_K\ne\mathbb B_K\) and
\(\dd\Psi\ne\dd\Psi^{\mathsf T}\). Thus squareness does not turn
\eqref{eq:QBD-compact-mixed-spectral} into the usual self-adjoint
matrix-valued Karlin--McGregor formula.

\begin{corollary}[Rank obstruction to block symmetrization]
\label{coro:QBD-rank-obstruction}
	Assume that \(T\) is \((p,q)\)-banded and that its two outer diagonals
	are nowhere zero,
	\[
	T_{n,n+q}\ne0,\qquad T_{n+p,n}\ne0,\qquad n\ge0,
	\]
	and use blocks of size
	\(r=\max\{p,q\}\). Then
	\[
	\operatorname{rank}\mathsf T_K^+=q,
	\qquad
	\operatorname{rank}\mathsf T_{K+1}^-=p,
	\qquad K\ge0.
	\]
	Consequently, if \(p\ne q\), no invertible block diagonal similarity can
	transform \(T\) into a symmetric block Jacobi matrix.
\end{corollary}

\begin{proof}
	In \(\mathsf T_K^+\), the rows \(r-q,\ldots,r-1\) and columns
	\(0,\ldots,q-1\) form a lower triangular \(q\times q\) submatrix whose
	diagonal consists of entries from the \(q\)-th superdiagonal of \(T\);
	all remaining independent rows or columns vanish. Hence its rank is
	\(q\). Similarly, the rows \(0,\ldots,p-1\) and columns
	\(r-p,\ldots,r-1\) of \(\mathsf T_{K+1}^-\) form an upper triangular
	\(p\times p\) submatrix with nonzero diagonal from the \(p\)-th
	subdiagonal, and its rank is \(p\).

	If a block diagonal similarity made \(T\) symmetric, its
	\((K,K+1)\) and \((K+1,K)\) blocks would be transposes and therefore
	would have equal rank. Invertible left and right multiplications preserve
	rank, so this would force \(p=q\).
\end{proof}

When \(p=q\), the rank obstruction disappears but block symmetrizability
is still not automatic: one must solve the full matrix identities
\eqref{eq:QBD-block-symmetrizer}. The PBF hypothesis supplies the mixed
representation \eqref{eq:QBD-compact-mixed-spectral}, not those additional
self-adjointness identities. In the exceptional case in which a positive
block symmetrizer exists, a block-diagonal similarity reduces the representation to the
standard positive-definite matrix-valued setting. Otherwise the two-sided
mixed formula is the relevant spectral object.

The state of the QBD is specified by a level \(K\) and a phase \(i\), the
position within that level. Ordering the pairs level by level as
\((0,0),(0,1),\ldots,(0,r-1),(1,0),(1,1),\ldots,(1,r-1),\ldots\)
gives the block-tridiagonal representation
\eqref{eq:QBD-block-tridiagonal} of the transition matrix of
\(Z_\ell=\begin{bNiceMatrix}K_\ell&I_\ell\end{bNiceMatrix}\). The second
coordinate is finite, so the state space is a semi-infinite strip rather
than the full quadrant \(\Nzero^2\).

The word ``birth'' in this QBD description refers to an increase of the
\emph{level}, not necessarily to a scalar jump by \(+1\); analogously, a
QBD death decreases the level by one. The phase is the residue of the
scalar state modulo \(r\). Formula \eqref{eq:QBD-blocks-from-band} states
which scalar diagonals form each of the three QBD blocks; no
additional symmetry is introduced by this relabelling.

This grouping is applied explicitly to the rational mixed Pi\~neiro example
with \((p,q)=(3,2)\) in
Subsection~\ref{subsec:MP-QBD}: grouping three consecutive scalar states
produces a three-phase level-dependent QBD. Corollary
\ref{coro:QBD-rank-obstruction} proves in advance that this QBD is not
block symmetrizable. Its five ordered binary urn experiments realize every
block transition, while
\eqref{eq:QBD-block-spectral-discrete} and
\eqref{eq:QBD-block-spectral-continuous} provide its discrete- and
continuous-time spectral descriptions.

\section{The mixed Pi\~neiro system: spectral construction and positivity}
\label{sec:mixed-Pineiro-model}

The mixed Pi\~neiro system studied in \cite{BFDM_Pineiro_2026} is the main
explicit example of the general \((p,q)\)-band theory. The general
Christoffel--Gauss--Borel bidiagonal
construction is developed in \cite{BFM_Bidiagonal_2026}.  In the present
case the scalar step-line moment matrix is a Cauchy matrix, so its
Gauss--Borel factorization gives a direct construction of the mixed forms,
their biorthogonality, the recurrence matrix, and its bidiagonal factors.
The factor entries used here are derived directly from the Cauchy formulas in
Appendix~\ref{app:MP-bidiagonal-derivation}, so the construction required
in this paper is self-contained. This
section determines the exact region of positivity of the prescribed
bidiagonal factorization, separates it from the larger structural-band
nonnegativity region, and ends with the stochastic normalization and its
mixed spectral representation. The probabilistic models built from this
kernel are treated in Section~\ref{sec:mixed-Pineiro-Markov}.

\subsection{Explicit mixed forms and bidiagonal factorization}
\label{subsec:MP-factorization}

\subsubsection{Cauchy moment matrix and step-line forms}

Fix integers $p,q\ge1$. Let
\(\boldsymbol\alpha=(\alpha_1,\dots,\alpha_p)\) and
\(\boldsymbol\beta=(\beta_1,\dots,\beta_q)\),
and consider the mixed
Pi\~neiro matrix of measures on $(0,1)$
\begin{equation}\label{eq:MP-matrix-measure}
\dd \boldsymbol\mu^{\mathrm P}(x)
=
\left[x^{\alpha_i+\beta_j}\right]_{
\substack{j\in\{1,\ldots,q\}\\ i\in\{1,\ldots,p\}}}
\dd x.
\end{equation}
The matrix measure is locally integrable precisely when
\(\alpha_i+\beta_j>-1\) for every pair \(i,j\). In the ordered region
introduced below, this reduces to \(\alpha_1+\beta_1>-1\).

\begin{definition}[Admissible Pi\~neiro parameter domain]
\label{def:MP-admissible-domain}
Let \(\mathcal A^{\mathrm P}_{p,q}\) be the set of pairs
\((\boldsymbol\alpha,\boldsymbol\beta)\in\mathbb R^{p+q}\) satisfying
\begin{equation}
\label{eq:MP-admissible-domain}
\begin{aligned}
\alpha_i&>-1 &&(1\le i\le p),
&\qquad \beta_j&>-1 &&(1\le j\le q),\\
\alpha_i+\beta_j&>-1
&& (1\le i\le p,\ 1\le j\le q),\\
\alpha_i-\alpha_k&\notin\mathbb Z &&(i\ne k),
&\qquad \beta_j-\beta_h&\notin\mathbb Z &&(j\ne h).
\end{aligned}
\end{equation}
\end{definition}

The sum inequalities ensure the integrability of all matrix entries. The
nonintegrality conditions will ensure that the two cyclic exponent
sequences used below have no repeated terms. The
corresponding step-line multiplication matrix is denoted by $T^{\mathrm P}$.
It is the recurrence matrix for multiplication by $x$ in the mixed-type
Pi\~neiro system and has bandwidth $(p,q)$, that is,
\((T^{\mathrm P})_{N,M}=0\) if
\(M-N\notin\{-p,\ldots,q\}\).
Thus, for $(p,q)=(3,2)$, one obtains a six-diagonal matrix, with possible
jumps of sizes $-3,-2,-1,0,+1,+2$ after stochastic normalization.

To specify both families on the step line, write
\begin{equation}\label{eq:MP-step-quotients}
\begin{aligned}
N&=pe_N+i_N-1=qd_N+s_N-1,\\
e_N,d_N&\in\mathbb N_0,
&i_N&\in\{1,\ldots,p\},
&s_N&\in\{1,\ldots,q\}.
\end{aligned}
\end{equation}
To avoid any collision with the parameter vector $\boldsymbol\beta$, denote by $\boldsymbol\nu_d(R)$ the standard step-line multi-index of length $d$ and norm $R$: if \(R=d w_R+r_R\), with
\(r_R\in\{0,\ldots,d-1\}\), then
\[
(\boldsymbol\nu_d(R))_i=
\begin{cases}
w_R+1,& i\le r_R,\\
w_R,& i> r_R,
\end{cases}
\qquad i\in\{1,\ldots,d\}.
\]
Thus, in the present notation, the type-I step-line multi-indices are
\begin{equation}\label{eq:MP-type-I-step-indices}
 \boldsymbol n_N=\boldsymbol\nu_p(N+1),
 \qquad
 \boldsymbol m_N=\boldsymbol\nu_q(N),
 \qquad
 n_{N,i}=(\boldsymbol n_N)_i,
 \qquad
 m_{N,j}=(\boldsymbol m_N)_j.
\end{equation}
In particular, \(m_{N,s_N}=d_N\). The dual type-II step-line indices are
\begin{equation}\label{eq:MP-type-II-step-indices}
 \widetilde{\boldsymbol n}_N=\boldsymbol\nu_p(N),
 \qquad
 \widetilde{\boldsymbol m}_N=\boldsymbol\nu_q(N+1),
 \qquad
 \widetilde n_{N,i}=(\widetilde{\boldsymbol n}_N)_i,
 \qquad
 \widetilde m_{N,j}=(\widetilde{\boldsymbol m}_N)_j.
\end{equation}
Thus \(|\boldsymbol n_N|=|\boldsymbol m_N|+1\),
\(|\widetilde{\boldsymbol m}_N|=|\widetilde{\boldsymbol n}_N|+1\), and
\(\widetilde m_{N,s_N}=d_N+1\).

The two polynomial arrays and the two scalar mixed forms can now be written
without leaving any orthogonal object implicit. Put
\begin{align}
\boldsymbol A_N^{\mathrm P}(x)
&\coloneq
\begin{bNiceMatrix}
A_N^{(1),\mathrm P}(x)\\
\Vdots\\
A_N^{(p),\mathrm P}(x)
\end{bNiceMatrix},
&
A_N^{\mathrm P}(x)
&\coloneq
\sum_{i=1}^{p}A_N^{(i),\mathrm P}(x)x^{\alpha_i},
\label{eq:MP-explicit-type-I-form}
\\
\boldsymbol B_N^{\mathrm P}(x)
&\coloneq
\begin{bNiceMatrix}
B_N^{(1),\mathrm P}(x)&\Cdots&B_N^{(q),\mathrm P}(x)
\end{bNiceMatrix},
&
B_N^{\mathrm P}(x)
&\coloneq
\sum_{j=1}^{q}B_N^{(j),\mathrm P}(x)x^{\beta_j}.
\label{eq:MP-explicit-type-II-form}
\end{align}
A polynomial component occurs in the corresponding mixed form only when
its degree bound is nonnegative. Such a component is called active; thus
\(A_N^{(i),\mathrm P}\) is active when \(n_{N,i}\ge1\), and
\(B_N^{(j),\mathrm P}\) is active when
\(\widetilde m_{N,j}\ge1\).

\begin{proposition}[Direct Cauchy construction of the step-line forms]
\label{pro:MP-Cauchy-step-line-construction}
Let \((\boldsymbol\alpha,\boldsymbol\beta)\in\mathcal A^{\mathrm P}_{p,q}\), and use
the step-line quotients in \eqref{eq:MP-step-quotients}. Define
\begin{equation}\label{eq:MP-Cauchy-exponent-sequences}
\xi_N\coloneq\beta_{s_N}+d_N,
\qquad
\eta_M\coloneq\alpha_{i_M}+e_M+1,
\end{equation}
and let \(\mathscr M\) be the scalar moment matrix
\begin{equation}\label{eq:MP-step-line-Cauchy-moment-matrix}
\begin{aligned}
\mathscr M_{N,M}
&\coloneq
\int_0^1 x^{d_N}\,\dd\mu_{s_N,i_M}^{\mathrm P}(x)\,x^{e_M}
\\
&=
\int_0^1x^{d_N+\beta_{s_N}}x^{e_M+\alpha_{i_M}}\,\dd x
=\frac{1}{\xi_N+\eta_M}.
\end{aligned}
\end{equation}
Then every leading principal minor of \(\mathscr M\) is nonzero, and
there is a unique factorization
\(\mathscr M=L_{\mathscr M}U_{\mathscr M}\), where
\(L_{\mathscr M}\) is unit lower triangular and
\(U_{\mathscr M}\) is upper triangular. For \(0\le K\le N\),
\begin{equation}\label{eq:MP-Cauchy-lower-factor}
(L_{\mathscr M})_{N,K}
=
\frac{\prod_{r=0}^{K-1}(\xi_N-\xi_r)}
     {\prod_{r=0}^{K-1}(\xi_K-\xi_r)}
\frac{\prod_{r=0}^{K}(\xi_K+\eta_r)}
     {\prod_{r=0}^{K}(\xi_N+\eta_r)},
\end{equation}
and the inverse factors have the product formulas
\begin{equation}\label{eq:MP-Cauchy-lower-inverse}
(L_{\mathscr M}^{-1})_{N,K}
=
\frac{\prod_{r=0}^{N-1}(\xi_K+\eta_r)}
     {\prod_{r=0}^{N-1}(\xi_N+\eta_r)}
\frac{\prod_{r=0}^{N-1}(\xi_N-\xi_r)}
     {\prod_{\substack{0\le r\le N\\r\ne K}}
      (\xi_K-\xi_r)}
\end{equation}
and
\begin{equation}\label{eq:MP-Cauchy-upper-inverse-transpose}
(U_{\mathscr M}^{-\mathsf T})_{N,K}
=
\frac{\prod_{r=0}^{N}(\xi_N+\eta_r)}
     {\prod_{r=0}^{N-1}(\xi_N-\xi_r)}
\frac{\prod_{r=0}^{N-1}(\xi_r+\eta_K)}
     {\prod_{\substack{0\le r\le N\\r\ne K}}
      (\eta_K-\eta_r)}.
\end{equation}
Empty products are understood as \(1\). The two step-line mixed forms are
obtained directly from these inverse factors:
\begin{equation}\label{eq:MP-Cauchy-direct-step-line-forms}
\begin{aligned}
B_N^{\mathrm P}(x)
&=
\sum_{K=0}^{N}(L_{\mathscr M}^{-1})_{N,K}x^{\xi_K},
&
A_N^{\mathrm P}(x)
&=
\sum_{K=0}^{N}(U_{\mathscr M}^{-\mathsf T})_{N,K}
x^{\eta_K-1}.
\end{aligned}
\end{equation}
In particular, the first formula gives the monic type-II normalization,
the second gives the normalized type-I form, and
\begin{equation}\label{eq:MP-Cauchy-direct-biorthogonality}
\int_0^1B_N^{\mathrm P}(x)A_M^{\mathrm P}(x)\,\dd x
=\delta_{N,M}.
\end{equation}
Finally, if \((\Lambda_r)_{N,M}=\delta_{M,N+r}\), then
\begin{equation}\label{eq:MP-Cauchy-direct-recurrence-matrix}
T^{\mathrm P}
=L_{\mathscr M}^{-1}\Lambda_qL_{\mathscr M}
=U_{\mathscr M}\Lambda_p^{\mathsf T}U_{\mathscr M}^{-1}.
\end{equation}
\end{proposition}

\begin{proof}
The matrix measure in \eqref{eq:MP-matrix-measure} has \(q\) rows and
\(p\) columns. The scalar index \(N\) lists successively each of its
\(q\) rows multiplied by \(1,x,x^2,\ldots\), while \(M\) lists each of
its \(p\) columns with the same successive monomial multipliers. This is
exactly the integral in
\eqref{eq:MP-step-line-Cauchy-moment-matrix}.

For \(n\ge0\), let \(\mathscr M^{[n]}\) be the leading square block with
indices \(0,\ldots,n\). The Cauchy determinant identity gives
\begin{equation}\label{eq:MP-Cauchy-determinant}
\det\mathscr M^{[n]}
=
\frac{
\prod_{0\le r<s\le n}(\xi_s-\xi_r)
\prod_{0\le r<s\le n}(\eta_s-\eta_r)
}{
\prod_{r,s=0}^{n}(\xi_r+\eta_s)
}.
\end{equation}
This form of the identity is symmetric in the two exponent sequences:
transposing the Cauchy matrix interchanges \(\xi\) and \(\eta\), the two
Vandermonde products exchange places, and the denominator is unchanged.
This symmetry of the determinant does not assert that \(\mathscr M\)
itself is symmetric. The latter would require
\(\xi_N-\eta_N\) to be independent of \(N\), which does not hold for
general mixed Pi\~neiro parameters. The extra \(1\) in \(\eta_M\) is
likewise only a convention: it comes from
\(\int_0^1x^\gamma\,\dd x=(\gamma+1)^{-1}\) and could instead be
assigned to the \(\xi\)-sequence.

Admissibility makes the \(\xi\)'s pairwise distinct and the \(\eta\)'s
pairwise distinct, because equality between two terms would give an
integral difference between two beta parameters or between two alpha
parameters. Moreover, every \(\xi_N+\eta_M\) is positive. Hence all the
leading principal minors are nonzero, which proves the existence and
uniqueness of this factorization. Formula
\eqref{eq:MP-Cauchy-lower-factor} follows by applying
\eqref{eq:MP-Cauchy-determinant} to the elimination minors.

The inverse formulas and their polynomial meaning are verified next. Put
\(c_{N,K}=(L_{\mathscr M}^{-1})_{N,K}\), with the right-hand side given
by \eqref{eq:MP-Cauchy-lower-inverse}. Partial fractions give
\begin{equation*}
\sum_{K=0}^{N}\frac{c_{N,K}}{z+\xi_K}
=
\frac{\prod_{r=0}^{N-1}(\xi_N-\xi_r)}
     {\prod_{r=0}^{N-1}(\xi_N+\eta_r)}
\frac{\prod_{r=0}^{N-1}(z-\eta_r)}
     {\prod_{r=0}^{N}(z+\xi_r)}.
\end{equation*}
The coefficient \(c_{N,N}\) is \(1\), and the right-hand side vanishes
at \(z=\eta_0,\ldots,\eta_{N-1}\). Consequently the first expression
in \eqref{eq:MP-Cauchy-direct-step-line-forms} is monic in the last basis
element \(x^{\xi_N}\) and is orthogonal to
\(x^{\eta_0-1},\ldots,x^{\eta_{N-1}-1}\). This identifies its
coefficient row with row \(N\) of \(L_{\mathscr M}^{-1}\).

Similarly, let \(d_{N,K}\) denote the right-hand side of
\eqref{eq:MP-Cauchy-upper-inverse-transpose}. A second partial-fraction
identity is
\begin{equation*}
\sum_{K=0}^{N}\frac{d_{N,K}}{z+\eta_K}
=
\frac{\prod_{r=0}^{N}(\xi_N+\eta_r)}
     {\prod_{r=0}^{N-1}(\xi_N-\xi_r)}
\frac{\prod_{r=0}^{N-1}(z-\xi_r)}
     {\prod_{r=0}^{N}(z+\eta_r)}.
\end{equation*}
Its values at \(z=\xi_0,\ldots,\xi_N\) are respectively
\(0,\ldots,0,1\). Therefore the second expression in
\eqref{eq:MP-Cauchy-direct-step-line-forms} satisfies the type-I
orthogonality conditions and its normalization at the last row moment is
\(1\). Hence \(d_{N,K}=(U_{\mathscr M}^{-\mathsf T})_{N,K}\).

Regrouping the exponents
\(\xi_K=\beta_{s_K}+d_K\) and
\(\eta_K-1=\alpha_{i_K}+e_K\) in
\eqref{eq:MP-Cauchy-direct-step-line-forms} gives, respectively, the
type-II components in \eqref{eq:MP-explicit-type-II-form} and the type-I
components in \eqref{eq:MP-explicit-type-I-form}. If
\(\mathbf B^{\mathrm P}=(B_N^{\mathrm P})_{N\ge0}\) and
\(\mathbf A^{\mathrm P}=(A_N^{\mathrm P})_{N\ge0}\), then
\begin{equation*}
\int_0^1
\mathbf B^{\mathrm P}(x)\mathbf A^{\mathrm P}(x)^{\mathsf T}\,\dd x
=
L_{\mathscr M}^{-1}\mathscr M U_{\mathscr M}^{-1}
=I.
\end{equation*}
This proves \eqref{eq:MP-Cauchy-direct-biorthogonality}.

Finally,
\(x\boldsymbol\chi_\beta=\Lambda_q\boldsymbol\chi_\beta\) and
\(x\boldsymbol\chi_\alpha=\Lambda_p\boldsymbol\chi_\alpha\), where
\(\boldsymbol\chi_\beta=(x^{\xi_N})_{N\ge0}\) and
\(\boldsymbol\chi_\alpha=(x^{\eta_N-1})_{N\ge0}\). Conjugating the
first shift relation gives
\(x\mathbf B^{\mathrm P}
=L_{\mathscr M}^{-1}\Lambda_qL_{\mathscr M}\mathbf B^{\mathrm P}\),
whereas the second gives
\(x\mathbf A^{\mathrm P}
=U_{\mathscr M}^{-\mathsf T}\Lambda_pU_{\mathscr M}^{\mathsf T}
\mathbf A^{\mathrm P}\).
The coefficient matrices in the right and left recurrences are
\(T^{\mathrm P}\) and \((T^{\mathrm P})^{\mathsf T}\), respectively.
The first equality therefore gives
\(T^{\mathrm P}=L_{\mathscr M}^{-1}\Lambda_qL_{\mathscr M}\), while
transposing the second gives
\(T^{\mathrm P}=U_{\mathscr M}\Lambda_p^{\mathsf T}
U_{\mathscr M}^{-1}\). This proves
\eqref{eq:MP-Cauchy-direct-recurrence-matrix}.
\end{proof}

\begin{corollary}[Total positivity of the Pi\~neiro moment matrix]
\label{cor:MP-moment-matrix-total-positivity}
On the ordered cyclic region
\(\mathcal R^{\mathrm P,+}_{p,q}\), the scalar step-line moment matrix
\(\mathscr M\) in \eqref{eq:MP-step-line-Cauchy-moment-matrix} is strictly
totally positive.
\end{corollary}

\begin{proof}
In this region the affine cyclic inequalities make both exponent sequences
\((\xi_N)_{N\ge0}\) and \((\eta_N)_{N\ge0}\) strictly increasing.  Every
minor with row indices \(N_1<\cdots<N_k\) and column indices
\(M_1<\cdots<M_k\) is a Cauchy determinant.  Its numerator is the product
of the two positive Vandermonde determinants in
\((\xi_{N_i})_{i=1}^k\) and \((\eta_{M_j})_{j=1}^k\), and its denominator
is \(\prod_{i,j}(\xi_{N_i}+\eta_{M_j})>0\).  Hence every finite minor is
strictly positive.
\end{proof}

The preceding proposition is the primary construction of the step-line
forms in this paper. For later parameter calculations it is useful to
regroup its scalar coefficients according to their alpha or beta component.
The following corollary records that regrouping; it is not a second
construction of the mixed forms.

\subsubsection{Componentwise finite formulas}

\begin{corollary}[Componentwise Cauchy formulas]
\label{coro:MP-componentwise-Cauchy-formulas}
Let \((\boldsymbol\alpha,\boldsymbol\beta)\in\mathcal A^{\mathrm P}_{p,q}\), and use
the step-line indices introduced above.

For the type-I normalization define
\[
C_N^{A}
\coloneq
\frac{
\prod_{i=1}^{p}
(\alpha_i+\beta_{s_N}+d_N+1)_{n_{N,i}}
}{
\prod_{j=1}^{q}
(\beta_j-\beta_{s_N}-d_N)_{m_{N,j}}
}.
\]
For every active component \(i\) and
\(0\le k\le n_{N,i}-1\), put
\begin{equation}
\label{eq:MP-explicit-A-coefficients}
a_{N,i,k}
\coloneq
C_N^{A}
\frac{
\prod_{j=1}^{q}
(\alpha_i+\beta_j+k+1)_{m_{N,j}}
}{
\prod_{\substack{h=1\\h\ne i}}^{p}
(\alpha_h-\alpha_i-k)_{n_{N,h}}\,
k!\,(n_{N,i}-k-1)!\,(-1)^k
}.
\end{equation}
Then
\begin{equation}
\label{eq:MP-explicit-A-components}
A_N^{(i),\mathrm P}(x)
=
\sum_{k=0}^{n_{N,i}-1}a_{N,i,k}x^k,
\qquad
A_N^{(i),\mathrm P}\equiv0
\quad\hbox{if }n_{N,i}=0,
\end{equation}
and it is normalized by
\(\int_0^1x^{d_N+\beta_{s_N}}A_N^{\mathrm P}(x)\,\dd x=1\).

The type-II coefficients are equally explicit. For every active component
\(j\) and \(0\le k\le\widetilde m_{N,j}-1\), put
\begin{equation}
\label{eq:MP-raw-B-coefficients}
c_{N,j,k}
\coloneq
\frac{
\prod_{i=1}^{p}
(\alpha_i+\beta_j+k+1)_{\widetilde n_{N,i}}
}{
\prod_{\substack{h=1\\h\ne j}}^{q}
(\beta_h-\beta_j-k)_{\widetilde m_{N,h}}\,
k!\,(\widetilde m_{N,j}-k-1)!\,(-1)^k
}.
\end{equation}
The contour-normalized components are
\(\widehat B_N^{(j),\mathrm P}(x)
=\sum_{k=0}^{\widetilde m_{N,j}-1}c_{N,j,k}x^k\), and the
leading coefficient of the \(s_N\)-th component is
\[
\widehat\kappa_N
\coloneq
c_{N,s_N,d_N}
=
\frac{
\prod_{i=1}^{p}
(\alpha_i+\beta_{s_N}+d_N+1)_{\widetilde n_{N,i}}
}{
\prod_{\substack{h=1\\h\ne s_N}}^{q}
(\beta_h-\beta_{s_N}-d_N)_{\widetilde m_{N,h}}\,
d_N!\,(-1)^{d_N}
}.
\]
The normalization that makes this leading coefficient equal to one is therefore
\begin{equation}
\label{eq:MP-explicit-B-components}
B_N^{(j),\mathrm P}(x)
=
\frac{1}{\widehat\kappa_N}
\sum_{k=0}^{\widetilde m_{N,j}-1}c_{N,j,k}x^k,
\qquad
B_N^{(j),\mathrm P}\equiv0
\quad\hbox{if }\widetilde m_{N,j}=0.
\end{equation}
In particular,
\begin{equation}
\label{eq:MP-h-explicit-alpha-beta}
h_N^{\mathrm P}
\coloneq B_N^{\mathrm P}(1)
=
\frac{1}{\widehat\kappa_N}
\sum_{\substack{1\le j\le q\\\widetilde m_{N,j}\ge1}}
\sum_{k=0}^{\widetilde m_{N,j}-1}c_{N,j,k}.
\end{equation}
\end{corollary}

\begin{proof}
Every scalar index has unique representations
\(K=qk+j-1\), with \(1\le j\le q\), and
\(K=pk+i-1\), with \(1\le i\le p\). Accordingly,
\(\xi_K=\beta_j+k\) in the first representation and
\(\eta_K-1=\alpha_i+k\) in the second. Substitution in
\eqref{eq:MP-Cauchy-lower-inverse} and
\eqref{eq:MP-Cauchy-upper-inverse-transpose}, followed by collecting the
terms with the same base exponent, gives
\eqref{eq:MP-explicit-A-coefficients}--
\eqref{eq:MP-explicit-B-components}. In the type-II formula, the
coefficient corresponding to \(K=N\) is
\(c_{N,s_N,d_N}/\widehat\kappa_N=1\), which gives the monic
normalization. The type-I normalization follows from
\eqref{eq:MP-Cauchy-direct-step-line-forms}, and evaluation of the type-II
form at \(x=1\) gives \eqref{eq:MP-h-explicit-alpha-beta}.
\end{proof}

The component formulas satisfy the two mixed orthogonality conditions and
the biorthogonality relation
\eqref{eq:MP-Cauchy-direct-biorthogonality} because they are obtained from
Proposition~\ref{pro:MP-Cauchy-step-line-construction}. Expanding the two
scalar forms also shows that the spectral matrix
of measures is exactly the original Pi\~neiro matrix
\begin{equation}
\label{eq:MP-spectral-measure-explicit}
\dd\psi_{j,i}^{\mathrm P}(x)
=\dd\mu_{j,i}^{\mathrm P}(x)
=x^{\alpha_i+\beta_j}\,\dd x,
\qquad
1\le i\le p,
\quad
1\le j\le q.
\end{equation}
Every component in
\eqref{eq:MP-explicit-A-components} and
\eqref{eq:MP-explicit-B-components} is a finite polynomial. Rewriting the
ratios of consecutive coefficients gives terminating
\({}_{p+q}F_{p+q-1}\) representations. Those hypergeometric forms, and
their extension beyond the step line, are not needed for the Cauchy
construction above. Thus the polynomials, the endpoint values, and the spectral measures
used here are finite expressions in \(\boldsymbol\alpha\) and \(\boldsymbol\beta\).

\subsubsection{Recurrence matrix and bidiagonal factors}

The Cauchy factorization already gives \(T^{\mathrm P}\) through
\eqref{eq:MP-Cauchy-direct-recurrence-matrix}. For the urn interpretation
and the positivity of the individual factors, this recurrence matrix is
also factored into one-step Christoffel factors. Applying the
elementary Christoffel connections to the Cauchy forms gives
\begin{equation}\label{eq:MP-PBF}
T^{\mathrm P}
=
L_1^{\mathrm P}\cdots L_p^{\mathrm P}
U_q^{\mathrm P}\cdots U_1^{\mathrm P},
\end{equation}
where the \(L_k^{\mathrm P}\) are lower bidiagonal and the
\(U_k^{\mathrm P}\) are upper bidiagonal. Positivity of their nontrivial
entries is a separate property, established below on an explicit parameter
region.

The factor entries in simplified form are as follows.
Extend the parameters affinely by
\(\alpha_{mp+i}=\alpha_i+m\) and
\(\beta_{mq+j}=\beta_j+m\), where \(m\in\mathbb N_0\),
\(i\in\{1,\ldots,p\}\), and \(j\in\{1,\ldots,q\}\).
\begin{definition}[Affine cyclic shifts]
\label{def:MP-affine-cyclic-shifts}
For \(r,s\in\Nzero\), define
\begin{equation}
\label{eq:MP-affine-cyclic-shifts}
\boldsymbol\alpha^{[r]}
\coloneq(\alpha_{r+1},\ldots,\alpha_{r+p}),
\qquad
\boldsymbol\beta^{[s]}
\coloneq(\beta_{s+1},\ldots,\beta_{s+q}),
\end{equation}
where the entries on the right-hand side are understood through the affine
extensions above. These vectors are called the affine cyclic shifts of
\(\boldsymbol\alpha\) and \(\boldsymbol\beta\), respectively.
\end{definition}

Thus the shifts occurring in the factor entries are absorbed in the indices
of \(\alpha\) and \(\beta\). For \(a\in\{1,\ldots,p\}\), the
subdiagonal entry of the lower factor \(L_a^{\mathrm P}\) is
\begin{multline}\label{eq:MP-lower-Pochhammer-simplified}
\ell_{a,N}^{\mathrm P}
=
-
\frac{
\prod_{j=1}^{q}
(\beta_j-\beta_{s_{N+1}}-d_{N+1})_{d_{N+1}}
\prod_{h=1}^{s_{N+1}-1}(\beta_h-\beta_{s_{N+1}})
}{
\prod_{j=1}^{q}
(\beta_j-\beta_{s_N}-d_N)_{d_N}
\prod_{h=1}^{s_N-1}(\beta_h-\beta_{s_N})
}
\\[2mm]
{}
\times
\frac{
\prod_{i=1}^{p}
(\alpha_{a+i}+\beta_{s_N}+d_N+1)_{n_{N,i}}
}{
\prod_{i=1}^{p}
(\alpha_{a-1+i}+\beta_{s_{N+1}}+d_{N+1}+1)_{n_{N+1,i}}
}
\\[2mm]
{}
\times
\frac{
\alpha_{a+i_N}-\alpha_a+e_N
}{
\alpha_{a+i_N}+e_N+\beta_{s_N}+d_N+1
}.
\end{multline}
For \(b\in\{1,\ldots,q\}\), the diagonal entry of the upper factor
\(U_b^{\mathrm P}\) is
\begin{multline}\label{eq:MP-upper-Pochhammer-simplified}
u_{b,N}^{\mathrm P}
=
\prod_{i=1}^{p}
\frac{
(\alpha_i+\beta_{b-1+s_N}+d_N+1)_{n_{N,i}}
}{
(\alpha_i+\beta_{b+s_N}+d_N+1)_{n_{N,i}}
}
\prod_{j=1}^{q}
\frac{
(\beta_{b+j}-\beta_{b+s_N}-d_N)_{m_{N,j}}
}{
(\beta_{b-1+j}-\beta_{b-1+s_N}-d_N)_{m_{N,j}}
}
\\[2mm]
{}
\times
\prod_{j=1}^{q}
\frac{
(\alpha_{i_N}+n_{N,i_N}+\beta_{b-1+j})_{m_{N,j}}
}{
(\alpha_{i_N}+n_{N,i_N}+\beta_{b+j})_{m_{N,j}}
}.
\end{multline}
Equations~\eqref{eq:MP-lower-Pochhammer-simplified} and
\eqref{eq:MP-upper-Pochhammer-simplified} give the successive lower and
upper bidiagonal Christoffel factors of the Cauchy recurrence matrix.
Appendix~\ref{app:MP-bidiagonal-derivation} derives both formulas from
\eqref{eq:MP-explicit-A-coefficients}, including the cancellations of the
Pochhammer, factorial, and sign factors. No auxiliary leading coefficient
is left unevaluated.

The bidiagonal factors in \eqref{eq:MP-PBF} are therefore
\begin{align}
\label{eq:MP-lower-bidiagonal-factors}
(L_a^{\mathrm P})_{N,N}&=1,
&
(L_a^{\mathrm P})_{N+1,N}&=\ell_{a,N}^{\mathrm P},
&&
 a\in\{1,\ldots,p\},
\\
\label{eq:MP-upper-bidiagonal-factors}
(U_b^{\mathrm P})_{N,N}&=u_{b,N}^{\mathrm P},
&
(U_b^{\mathrm P})_{N,N+1}&=1,
&&
 b\in\{1,\ldots,q\},
\end{align}
with all remaining entries zero. Thus every nontrivial bidiagonal coefficient
is a simplified finite Pochhammer product.

\subsubsection{Grouped Gauss--Borel factors}

The individual Christoffel factors can also be grouped into the genuine
Gaussian \(LU\) factors of the recurrence matrix.  The following formulas
compute those two factors entry by entry and will be used below for their
stochastic normalization.

\begin{proposition}[Explicit grouped Gauss--Borel factorization]
\label{pro:MP-explicit-grouped-LU}
Put
\begin{equation}
\mathcal L^{\mathrm P}
\coloneq L_1^{\mathrm P}\cdots L_p^{\mathrm P},
\qquad
\mathcal U^{\mathrm P}
\coloneq U_q^{\mathrm P}\cdots U_1^{\mathrm P}.
\end{equation}
For \(0\le d\le\min\{p,N\}\) and \(0\le e\le q\), define
\begin{align}
\mathsf L^{\mathrm P}_{N,d}
&\coloneq
(\mathcal L^{\mathrm P})_{N,N-d}
=
\sum_{1\le a_1<\cdots<a_d\le p}
\prod_{r=1}^{d}\ell^{\mathrm P}_{a_r,N-r},
\label{eq:MP-grouped-lower-explicit}
\\[2mm]
\mathsf U^{\mathrm P}_{N,e}
&\coloneq
(\mathcal U^{\mathrm P})_{N,N+e}
=
\sum_{\substack{\varepsilon\in\{0,1\}^{q}\\
|\varepsilon|=e}}
\prod_{\substack{1\le b\le q\\\varepsilon_b=0}}
u^{\mathrm P}_{b,
N+\sum_{c=b+1}^{q}\varepsilon_c}.
\label{eq:MP-grouped-upper-explicit}
\end{align}
For \(d=0\), the sum in
\eqref{eq:MP-grouped-lower-explicit} consists of one empty product, and
empty products in \eqref{eq:MP-grouped-upper-explicit} are also \(1\).
The two extreme upper diagonals simplify to
\begin{equation}
\label{eq:MP-grouped-upper-extreme}
\begin{aligned}
\mathsf U^{\mathrm P}_{N,q}
&=1,
&
\mathsf U^{\mathrm P}_{N,0}
&=
\prod_{i=1}^{p}
\frac{\alpha_i+\beta_{s_N}+d_N+1}
{\alpha_i+\beta_{s_N}+d_N+n_{N,i}+1}
\prod_{j=1}^{q}
\frac{\alpha_{i_N}+n_{N,i_N}+\beta_j}
{\alpha_{i_N}+n_{N,i_N}+\beta_j+m_{N,j}}.
\end{aligned}
\end{equation}
All other entries of \(\mathcal L^{\mathrm P}\) and
\(\mathcal U^{\mathrm P}\) vanish.  Then
\begin{equation}
\label{eq:MP-explicit-grouped-LU}
T^{\mathrm P}=\mathcal L^{\mathrm P}\mathcal U^{\mathrm P}.
\end{equation}
Moreover, this is the unique Gauss--Borel factorization of
\(T^{\mathrm P}\) with unit lower factor.  After substituting
\eqref{eq:MP-lower-Pochhammer-simplified} and
\eqref{eq:MP-upper-Pochhammer-simplified},
\eqref{eq:MP-grouped-lower-explicit}--
\eqref{eq:MP-grouped-upper-explicit} become finite path sums whose
summands depend only on \(\boldsymbol\alpha\) and \(\boldsymbol\beta\), with no
unevaluated matrix product. Within an individual summand, consecutive
selected lower factors, and consecutive diagonal choices in the upper
product, telescope in the underlying leading-coefficient quotients.
\end{proposition}

\begin{proof}
In the ordered product of lower factors, a contribution from row \(N\) to
column \(N-d\) is obtained by choosing exactly \(d\) factors in which the
subdiagonal entry is used.  If these factors are
\(a_1<\cdots<a_d\), the state before the \(r\)-th downward move is
\(N-r+1\), and the corresponding coefficient is
\(\ell^{\mathrm P}_{a_r,N-r}\).  Summing the path weights gives
\eqref{eq:MP-grouped-lower-explicit}.

The upper factors act in the order \(q,q-1,\ldots,1\).  Let
\(\varepsilon_b=1\) when the superdiagonal of \(U_b^{\mathrm P}\) is
used and let \(\varepsilon_b=0\) when its diagonal is used.  On entering
\(U_b^{\mathrm P}\), the state has already increased by
\(\sum_{c=b+1}^{q}\varepsilon_c\).  Hence every stay contributes the
factor displayed in \eqref{eq:MP-grouped-upper-explicit}, whereas every
upward move contributes \(1\).  Summing over the binary words with \(e\)
moves proves that formula.  If \(e=q\), every factor contributes its
superdiagonal entry, which gives \(\mathsf U^{\mathrm P}_{N,q}=1\).
If \(e=0\), the state remains \(N\) throughout and
\(\mathsf U^{\mathrm P}_{N,0}=\prod_{b=1}^{q}
u^{\mathrm P}_{b,N}\).  Substitution of
\eqref{eq:MP-upper-Pochhammer-simplified} makes the products over \(b\)
telescope; the affine rule \(\beta_{q+j}=\beta_j+1\) then gives the
second identity in \eqref{eq:MP-grouped-upper-extreme}.
Equation \eqref{eq:MP-explicit-grouped-LU}
now follows from \eqref{eq:MP-PBF}.

The matrix \(\mathcal L^{\mathrm P}\) is unit lower triangular and
\(\mathcal U^{\mathrm P}\) is upper triangular, with
\(\mathsf U^{\mathrm P}_{N,0}\ne0\) on the admissible nonsingular
parameter domain by \eqref{eq:MP-grouped-upper-extreme}. Therefore every
leading principal minor of \(T^{\mathrm P}\) is nonzero, and the normalized
Gaussian \(LU\) factorization is unique.  In particular, the grouped
factors are intrinsic Gauss--Borel factors of the recurrence matrix, not
an arbitrary grouping chosen for the probabilistic interpretation.
\end{proof}

\subsection{Positive bidiagonal factorization}
\label{subsec:MP-PBF-positivity}

The parameter region where the prescribed factorization
\eqref{eq:MP-PBF} is positive is determined first. Here positivity means
that every nontrivial
entry of every lower and upper bidiagonal factor is strictly positive.

\begin{definition}[PBF-positivity set]
\label{def:MP-PBF-positivity-set}
Define
\begin{equation}
\mathcal P^{\mathrm P}_{p,q}
\coloneq\left\{(\boldsymbol\alpha,\boldsymbol\beta)\in\mathcal A^{\mathrm P}_{p,q}:
\ell_{a,N}^{\mathrm P}>0,\ u_{b,N}^{\mathrm P}>0
\text{ for all }a,b,N\right\}.
\end{equation}
The ordered cyclic PBF region is
\begin{equation}\label{eq:MP-PBF-region}
\mathcal R^{\mathrm P,+}_{p,q}:=
\left\{(\boldsymbol\alpha,\boldsymbol\beta):
\begin{aligned}
-1&<\alpha_1<\alpha_2<\cdots<\alpha_p<\alpha_1+1,\\
-1&<\beta_1<\beta_2<\cdots<\beta_q<\beta_1+1,\\
&\alpha_1+\beta_1>-1
\end{aligned}\right\}.
\end{equation}
\end{definition}

The factor entries involve affine cyclic shifts of the parameter vectors,
not only the original vectors. Thus it is not enough to impose the ordered
cyclic inequalities on \((\boldsymbol\alpha,\boldsymbol\beta)\): one must prove that they
persist under every shift occurring in the formulas. This stability is
relevant because, as the step-line index \(N\) varies, the different lower
bidiagonal factors \(L_a^{\mathrm P}\) and upper bidiagonal factors
\(U_b^{\mathrm P}\) involve different shifted parameter vectors. The following
lemma proves precisely the uniform property needed to apply the same sign
argument to every factor and every \(N\).

\begin{lemma}[Stability of the ordered cyclic inequalities]
\label{lem:MP-affine-shift-stability}
Let
\((\boldsymbol\alpha,\boldsymbol\beta)\in\mathcal R^{\mathrm P,+}_{p,q}\). For
\(r,s\in\Nzero\), let
\(\boldsymbol a=\boldsymbol\alpha^{[r]}=(a_1,\ldots,a_p)\) and
\(\boldsymbol b=\boldsymbol\beta^{[s]}=(b_1,\ldots,b_q)\) be the affine cyclic
shifts defined in Definition~\ref{def:MP-affine-cyclic-shifts}, equivalently
by \eqref{eq:MP-affine-cyclic-shifts}. Then
\begin{equation*}
a_1<\cdots<a_p<a_1+1,
\qquad
b_1<\cdots<b_q<b_1+1,
\qquad
a_1+b_1>-1.
\end{equation*}
Moreover,
\(\mathcal R^{\mathrm P,+}_{p,q}\subseteq
\mathcal A^{\mathrm P}_{p,q}\), and none of the denominators in
\eqref{eq:MP-lower-Pochhammer-simplified} and
\eqref{eq:MP-upper-Pochhammer-simplified} vanishes on
\(\mathcal R^{\mathrm P,+}_{p,q}\).
\end{lemma}

\begin{proof}
Write
\(r=k_\alpha p+\rho\) and \(s=k_\beta q+\sigma\), where
\(k_\alpha,k_\beta\in\Nzero\), \(0\le\rho<p\), and
\(0\le\sigma<q\). By Definition~\ref{def:MP-affine-cyclic-shifts},
\begin{align*}
\boldsymbol a
&=(\alpha_{\rho+1}+k_\alpha,\ldots,\alpha_p+k_\alpha,
\alpha_1+k_\alpha+1,\ldots,\alpha_\rho+k_\alpha+1),\\
\boldsymbol b
&=(\beta_{\sigma+1}+k_\beta,\ldots,\beta_q+k_\beta,
\beta_1+k_\beta+1,\ldots,\beta_\sigma+k_\beta+1),
\end{align*}
where the final sequence in \(\boldsymbol a\) is omitted when \(\rho=0\), and
the final sequence in \(\boldsymbol b\) is omitted when \(\sigma=0\). Within
each displayed sequence the strict inequalities follow from the original
ordering. If \(\rho>0\), the comparison between the two consecutive
sequences in \(\boldsymbol a\) is
\(\alpha_p+k_\alpha<\alpha_1+k_\alpha+1\), which follows from
\(\alpha_p<\alpha_1+1\). If \(\sigma>0\), the corresponding comparison
in \(\boldsymbol b\) is
\(\beta_q+k_\beta<\beta_1+k_\beta+1\), which follows from
\(\beta_q<\beta_1+1\). Hence both shifted vectors are strictly ordered.

It remains to compare the last and first entries. If \(\rho=0\), then
\(a_p=\alpha_p+k_\alpha<\alpha_1+k_\alpha+1=a_1+1\). If
\(1\le\rho<p\), then
\(a_p=\alpha_\rho+k_\alpha+1<
\alpha_{\rho+1}+k_\alpha+1=a_1+1\). Similarly, if \(\sigma=0\), then
\(b_q=\beta_q+k_\beta<\beta_1+k_\beta+1=b_1+1\), whereas if
\(1\le\sigma<q\), then
\(b_q=\beta_\sigma+k_\beta+1<
\beta_{\sigma+1}+k_\beta+1=b_1+1\). Finally,
\(a_1=\alpha_{\rho+1}+k_\alpha\ge\alpha_1\) and
\(b_1=\beta_{\sigma+1}+k_\beta\ge\beta_1\), so
\(a_1+b_1\ge\alpha_1+\beta_1>-1\).

The defining inequalities of
\(\mathcal R^{\mathrm P,+}_{p,q}\) give
\(\alpha_i>-1\), \(\beta_j>-1\), and
\(\alpha_i+\beta_j\ge\alpha_1+\beta_1>-1\) for every \(i,j\).
Every nonzero difference between two entries of the same parameter vector
has absolute value smaller than one and is therefore nonintegral. A
difference between shifted entries is such a difference plus an integer,
so it remains nonintegral. A finite Pochhammer symbol \((z)_n\) can vanish
only when \(z\) is a nonpositive integer. In the denominators of the factor
entries, every initial argument containing a difference is nonintegral,
whereas every initial argument containing a sum has the form
\(a_i+b_j+1+t\), with
\(t\in\Nzero\), and is positive. Hence all displayed denominators are
nonzero, which also proves
\(\mathcal R^{\mathrm P,+}_{p,q}\subseteq
\mathcal A^{\mathrm P}_{p,q}\).
\end{proof}

The following theorem gives the PBF condition used by the stochastic
construction below.
\begin{theorem}[Exact ordered cyclic PBF region]\label{thm:MP-PBF-region}
For every \(p,q\ge1\),
\begin{equation}\label{eq:MP-exact-PBF-region}
\mathcal R^{\mathrm P,+}_{p,q}
=\mathcal P^{\mathrm P}_{p,q}.
\end{equation}
\end{theorem}

\begin{proof}
First consider
\(\mathcal R^{\mathrm P,+}_{p,q}\subseteq
\mathcal P^{\mathrm P}_{p,q}\).
Let
\((\boldsymbol\alpha,\boldsymbol\beta)\in\mathcal R^{\mathrm P,+}_{p,q}\). By
Lemma~\ref{lem:MP-affine-shift-stability}, all affine cyclic shifts
occurring in the explicit factor-entry formulas
\eqref{eq:MP-lower-Pochhammer-simplified} and
\eqref{eq:MP-upper-Pochhammer-simplified} satisfy the same strict cyclic
inequalities, and all the entries displayed there are well defined. To prove that
\((\boldsymbol\alpha,\boldsymbol\beta)\in\mathcal P^{\mathrm P}_{p,q}\), it is enough to show
that the displayed quantities \(\ell_{a,N}^{\mathrm P}\) and
\(u_{h,N}^{\mathrm P}\) are positive for every
\(1\le a\le p\), \(1\le h\le q\), and \(N\ge0\). In both formulas, all
Pochhammer symbols involving sums of alpha
and beta parameters are manifestly positive. The linear alpha-difference
factor in \eqref{eq:MP-lower-Pochhammer-simplified} will follow directly
from the strict affine cyclic ordering of the alpha parameters. Thus the
only signs that require a separate product count come from the quotients
involving differences of beta parameters. Those quotients will be rewritten
as ratios of quantities \(E_N\), whose signs can be determined directly.

Here \(L_a^{\mathrm P}\) and \(U_h^{\mathrm P}\) are the individual lower
and upper bidiagonal factors in \eqref{eq:MP-PBF}; they should not be
confused with the grouped Gaussian factors
\(\mathcal L^{\mathrm P}=L_1^{\mathrm P}\cdots L_p^{\mathrm P}\) and
\(\mathcal U^{\mathrm P}=U_q^{\mathrm P}\cdots U_1^{\mathrm P}\).

In the definition below, \(\boldsymbol b\) has the same meaning as in
Lemma~\ref{lem:MP-affine-shift-stability}: it
denotes a generic affine cyclic shift
\(\boldsymbol b=\boldsymbol\beta^{[k]}=(\beta_{k+1},\ldots,\beta_{k+q})\) of
\(\boldsymbol\beta\). A sign formula valid for every \(k\) is established
first. It will then be used with \(k=0\) for the lower bidiagonal factors
\(L_a^{\mathrm P}\), and with \(k=h\) and \(k=h-1\) for the upper
bidiagonal factors \(U_h^{\mathrm P}\). Recall from
\eqref{eq:MP-type-I-step-indices} that, for \(1\le j\le q\),
\begin{equation*}
m_{N,j}=
\begin{cases}
d_N+1,&j<s_N,\\
d_N,&j\ge s_N,
\end{cases}
\qquad N=qd_N+s_N-1.
\end{equation*}
For any affine cyclic shift
\(\boldsymbol b=(b_1,\ldots,b_q)=\boldsymbol\beta^{[k]}\), define
\begin{equation}\label{eq:MP-beta-sign-product}
E_N(\boldsymbol b)
\coloneq
\prod_{j=1}^q
(b_j-b_{s_N}-d_N)_{m_{N,j}}.
\end{equation}
Its sign is determined one Pochhammer symbol at a time. Since
\(b_1<\cdots<b_q<b_1+1\), two entries of \(\boldsymbol b\) differ by a
number of absolute value strictly smaller than one.

Suppose first that \(j<s_N\). Then
\(-1<b_j-b_{s_N}<0\) and \(m_{N,j}=d_N+1\). Expanding the corresponding
Pochhammer symbol gives
\begin{equation*}
(b_j-b_{s_N}-d_N)_{d_N+1}
=
\prod_{r=0}^{d_N}
(b_j-b_{s_N}-d_N+r).
\end{equation*}
Its largest factor, obtained for \(r=d_N\), is
\(b_j-b_{s_N}<0\). Therefore all its \(d_N+1\) factors are strictly
negative.

For \(j=s_N\), one has \(m_{N,s_N}=d_N\), and the corresponding symbol
reduces to
\begin{equation*}
(b_{s_N}-b_{s_N}-d_N)_{d_N}
=(-d_N)_{d_N}
=(-d_N)(-d_N+1)\cdots(-1).
\end{equation*}
It consequently contributes exactly \(d_N\) negative factors. When
\(d_N=0\), this is the empty product and contributes none.

Finally, suppose that \(j>s_N\). In this case
\(0<b_j-b_{s_N}<1\) and \(m_{N,j}=d_N\). Thus
\begin{equation*}
(b_j-b_{s_N}-d_N)_{d_N}
=
\prod_{r=0}^{d_N-1}
(b_j-b_{s_N}-d_N+r).
\end{equation*}
If \(d_N>0\), its largest factor is
\(b_j-b_{s_N}-1\), which lies strictly between \(-1\) and \(0\).
Hence all its \(d_N\) factors are strictly negative; for \(d_N=0\), the
product is again empty.

There are \(s_N-1\) indices of the first kind, one central index, and
\(q-s_N\) indices of the third kind. The total number of negative
factors in \(E_N(\boldsymbol b)\) is therefore
\begin{equation*}
(d_N+1)(s_N-1)+d_N+d_N(q-s_N)
=qd_N+s_N-1=N.
\end{equation*}
The strict inequalities used above also show directly that none of these
factors vanishes. Therefore
\begin{equation}\label{eq:MP-beta-sign}
\operatorname{sgn}E_N(\boldsymbol b)=(-1)^N.
\end{equation}

This sign computation is now applied to the subdiagonal entries
\(\ell_{a,N}^{\mathrm P}\) of the lower bidiagonal factors
\(L_a^{\mathrm P}\). Take \(k=0\) in
\eqref{eq:MP-beta-sign-product}. Since
\(\boldsymbol\beta^{[0]}=\boldsymbol\beta\), this means that \(b_j=\beta_j\) for every
\(j\), and hence
\(E_N(\boldsymbol\beta)
=
\prod_{j=1}^{q}
(\beta_j-\beta_{s_N}-d_N)_{m_{N,j}}\).
For \(j<s_N\), use \(m_{N,j}=d_N+1\) and the elementary identity
\((z)_{d_N+1}=(z)_{d_N}(z+d_N)\). With
\(z=\beta_j-\beta_{s_N}-d_N\), this gives
\begin{equation*}
(\beta_j-\beta_{s_N}-d_N)_{m_{N,j}}
=
(\beta_j-\beta_{s_N}-d_N)_{d_N}
(\beta_j-\beta_{s_N}).
\end{equation*}
For \(j\ge s_N\), the definition gives \(m_{N,j}=d_N\). Separating the indices
\(j<s_N\) from the remaining indices therefore yields
\begin{equation*}
E_N(\boldsymbol\beta)
=
\prod_{j=1}^{q}
(\beta_j-\beta_{s_N}-d_N)_{d_N}
\prod_{j=1}^{s_N-1}(\beta_j-\beta_{s_N}).
\end{equation*}
This is exactly the denominator of the first quotient in
\eqref{eq:MP-lower-Pochhammer-simplified}. Replacing \(N\) by \(N+1\) in
the same identity shows that the numerator of that quotient is
\(E_{N+1}(\boldsymbol\beta)\). Thus the quotient is
\(E_{N+1}(\boldsymbol\beta)/E_N(\boldsymbol\beta)\). By
\eqref{eq:MP-beta-sign}, its sign is
\((-1)^{N+1}/(-1)^N=-1\), and the explicit minus sign preceding that
quotient changes it to \(+1\).

All the remaining signs in
\eqref{eq:MP-lower-Pochhammer-simplified} are positive. Every Pochhammer
symbol involving a sum of an alpha and a beta parameter has a positive
initial argument by Lemma~\ref{lem:MP-affine-shift-stability}. Furthermore,
the affine extension of the ordered alpha parameters is strictly
increasing. Since \(i_N\ge1\), this gives
\(\alpha_{a+i_N}>\alpha_a\), and hence
\(\alpha_{a+i_N}-\alpha_a+e_N>0\). The final denominator is also of the
positive sum type. It follows that
\(\ell_{a,N}^{\mathrm P}>0\).

Next consider the diagonal entries \(u_{h,N}^{\mathrm P}\) of the upper
bidiagonal factors \(U_h^{\mathrm P}\). Formula
\eqref{eq:MP-upper-Pochhammer-simplified} contains three quotients of
products. Only the middle one, which involves differences between beta
parameters, requires a sign computation.

First take \(k=h\) in \eqref{eq:MP-beta-sign-product}. The \(j\)-th
component of \(\boldsymbol\beta^{[h]}\) is \(\beta_{h+j}\), and its
\(s_N\)-th component is \(\beta_{h+s_N}\). Therefore
\begin{equation*}
E_N(\boldsymbol\beta^{[h]})
=
\prod_{j=1}^{q}
(\beta_{h+j}-\beta_{h+s_N}-d_N)_{m_{N,j}}.
\end{equation*}
Taking instead \(k=h-1\), which is nonnegative because \(h\ge1\), gives
\begin{equation*}
E_N(\boldsymbol\beta^{[h-1]})
=
\prod_{j=1}^{q}
(\beta_{h-1+j}-\beta_{h-1+s_N}-d_N)_{m_{N,j}}.
\end{equation*}
Consequently, the beta-difference quotient in
\eqref{eq:MP-upper-Pochhammer-simplified} is exactly
\begin{equation*}
\frac{
\prod_{j=1}^{q}
(\beta_{h+j}-\beta_{h+s_N}-d_N)_{m_{N,j}}
}{
\prod_{j=1}^{q}
(\beta_{h-1+j}-\beta_{h-1+s_N}-d_N)_{m_{N,j}}
}
=
\frac{E_N(\boldsymbol\beta^{[h]})}
{E_N(\boldsymbol\beta^{[h-1]})}.
\end{equation*}
Both \(\boldsymbol\beta^{[h]}\) and \(\boldsymbol\beta^{[h-1]}\) satisfy the strict
cyclic inequalities by Lemma~\ref{lem:MP-affine-shift-stability}. Hence
\eqref{eq:MP-beta-sign} applies to each of them:
\begin{equation*}
\operatorname{sgn}E_N(\boldsymbol\beta^{[h]})
=
\operatorname{sgn}E_N(\boldsymbol\beta^{[h-1]})
=(-1)^N.
\end{equation*}
Neither product vanishes, so their quotient is strictly positive.

It remains to check the other two quotients in
\eqref{eq:MP-upper-Pochhammer-simplified}. In the first quotient, every
Pochhammer symbol starts at a number of the form
\(\alpha_i+\beta_r+d_N+1\). The shifted parameters satisfy
\(\alpha_i+\beta_r>-1\), while \(d_N+1\ge1\); hence every such initial
argument is positive. In the third quotient, the initial arguments have
the form
\(\alpha_{i_N}+\beta_r+n_{N,i_N}\). Combining
\eqref{eq:MP-step-quotients} with
\eqref{eq:MP-type-I-step-indices} gives
\(n_{N,i_N}=e_N+1\ge1\), and therefore these initial arguments are also
positive. Thus every Pochhammer symbol in the first and third quotients is
positive, in both numerator and denominator.

All three quotients in
\eqref{eq:MP-upper-Pochhammer-simplified} are consequently positive, and
so \(u_{h,N}^{\mathrm P}>0\). It has already been proved that
\(\ell_{a,N}^{\mathrm P}>0\). Since the diagonal entries of
\(L_a^{\mathrm P}\) and the superdiagonal entries of \(U_h^{\mathrm P}\)
are equal to \(1\), every nontrivial entry of every individual bidiagonal
factor in \eqref{eq:MP-PBF} is positive. Therefore
\(\mathcal R^{\mathrm P,+}_{p,q}\subseteq
\mathcal P^{\mathrm P}_{p,q}\).

	To prove the reverse inclusion
	\(\mathcal P^{\mathrm P}_{p,q}\subseteq
	\mathcal R^{\mathrm P,+}_{p,q}\), use the affine extensions of the
	parameters defined above and put
	\begin{equation}\label{eq:MP-cyclic-increments}
	\Delta_a^\alpha:=\alpha_{a+1}-\alpha_a,
	\qquad
	\Delta_b^\beta:=\beta_{b+1}-\beta_b.
	\end{equation}
	These increments are periodic in their indices and satisfy
	\begin{equation}\label{eq:MP-cyclic-increment-sums}
	\sum_{a=1}^{p}\Delta_a^\alpha=1,
	\qquad
	\sum_{b=1}^{q}\Delta_b^\beta=1.
	\end{equation}
	Let \((\boldsymbol\alpha,\boldsymbol\beta)\in
	\mathcal P^{\mathrm P}_{p,q}\). The objective is to recover the two
	strict cyclic orderings in \eqref{eq:MP-PBF-region} from the assumed
	positivity of the individual bidiagonal factors. First, the lowest
	entries of the lower factors are used to prove that every alpha increment
	is positive and, at the same
	time, that \(\Delta_1^\beta>0\). The
	entries of the upper factors will then propagate this positivity from
	\(\Delta_1^\beta\) to all the remaining beta increments.

	Substitution of \(N=0\) in
	\eqref{eq:MP-lower-Pochhammer-simplified} gives
	\begin{equation}\label{eq:MP-lowest-lower-factor}
	\ell_{a,0}^{\mathrm P}
	=
	\frac{\Delta_a^\alpha\Delta_1^\beta}
	{(\alpha_a+\beta_2+1)(\alpha_{a+1}+\beta_2+1)},
	\qquad 1\le a\le p.
	\end{equation}
	The same identity covers the cases \(p=1\) and \(q=1\) under the
	affine cyclic conventions. More precisely, if \(p=1\), then
	\(\alpha_2=\alpha_1+1\) and \(\Delta_1^\alpha=1\); if \(q=1\), then
	\(\beta_2=\beta_1+1\) and \(\Delta_1^\beta=1\). By the definition of
	\(\mathcal A^{\mathrm P}_{p,q}\),
	\(\alpha_i+\beta_j+1>0\) for every pair of original indices
	\(1\le i\le p\) and \(1\le j\le q\). If \(q\ge2\), then \(\beta_2\)
	is one of the original beta parameters; if \(q=1\), then
	\(\beta_2=\beta_1+1\), which only increases both denominator factors.
	Likewise, \(\alpha_{a+1}\) is an original alpha parameter when \(a<p\),
	whereas \(\alpha_{p+1}=\alpha_1+1\). Consequently, each denominator
	factor in \eqref{eq:MP-lowest-lower-factor} is a positive quantity of the
	form \(\alpha_i+\beta_j+1\), possibly increased by one or two. Notice
	that this conclusion uses only admissibility. Since
	\((\boldsymbol\alpha,\boldsymbol\beta)\in\mathcal P^{\mathrm P}_{p,q}\), every entry
	\(\ell_{a,0}^{\mathrm P}\) is strictly positive. Formula
	\eqref{eq:MP-lowest-lower-factor} and positivity of its denominator then
	give
	\(\Delta_a^\alpha\Delta_1^\beta>0\) for every \(1\le a\le p\).
	Thus all the increments \(\Delta_a^\alpha\) have the same sign as the
	fixed increment \(\Delta_1^\beta\). This common sign cannot be negative,
	because \(\sum_{a=1}^p\Delta_a^\alpha=1\). Hence it is positive, and
	therefore
	\begin{equation}\label{eq:MP-alpha-order-necessary}
	\Delta_a^\alpha>0\quad(1\le a\le p),
	\qquad
	\Delta_1^\beta>0.
	\end{equation}

	If \(q\ge2\), let \(i_1\) be the index determined by
	\eqref{eq:MP-step-quotients} for \(N=1\), and use the step-line degrees
	\(n_{1,i}\) from \eqref{eq:MP-type-I-step-indices}. Substitution of \(N=1\)
	in \eqref{eq:MP-upper-Pochhammer-simplified} yields
	\begin{equation}\label{eq:MP-second-upper-factor}
	u_{b,1}^{\mathrm P}
	=D_b\frac{\Delta_{b+1}^\beta}{\Delta_b^\beta},
	\qquad 1\le b\le q,
	\end{equation}
	where the indices are affine cyclic and
	\begin{equation}\label{eq:MP-second-upper-positive-part}
	D_b:=
	\prod_{i=1}^{p}
	\frac{(\alpha_i+\beta_{b+1}+1)_{n_{1,i}}}
	{(\alpha_i+\beta_{b+2}+1)_{n_{1,i}}}
	\frac{\alpha_{i_1}+n_{1,i_1}+\beta_b}
	{\alpha_{i_1}+n_{1,i_1}+\beta_{b+1}}.
	\end{equation}
	The inequality \(D_b>0\) follows from admissibility alone. Indeed,
	each affinely extended beta parameter is an original beta parameter plus
	a nonnegative integer. Hence the Pochhammer symbols in
	\eqref{eq:MP-second-upper-positive-part} have positive initial arguments
	by \(\alpha_i+\beta_j>-1\). Moreover,
	\(n_{1,i_1}=e_1+1\ge1\), so the two remaining linear factors are also
	positive. No cyclic ordering of the beta parameters is used here.
	Thus positivity of all \(u_{b,1}^{\mathrm P}\) forces consecutive
	cyclic increments \(\Delta_b^\beta\) and
	\(\Delta_{b+1}^\beta\) to have the same sign. Since
	\(\Delta_1^\beta>0\), all \(\Delta_b^\beta\) are positive. For
	\(q=1\) this conclusion is automatic because
	\(\Delta_1^\beta=1\). Consequently,
	\[
	\alpha_1<\cdots<\alpha_p<\alpha_1+1,
	\qquad
	\beta_1<\cdots<\beta_q<\beta_1+1.
	\]
	The remaining inequalities in
	\eqref{eq:MP-PBF-region} already belong to the admissible domain.
	This proves
	\(\mathcal P^{\mathrm P}_{p,q}\subseteq
	\mathcal R^{\mathrm P,+}_{p,q}\), and hence equality.
\end{proof}

\subsection{Nonnegativity on the structural band of the recurrence matrix}
\label{subsec:MP-positivity}

The weaker requirement needed for a Doob stochastic normalization is now
considered.  For a \((p,q)\)-band matrix indexed by \(\Nzero\), define
its \emph{structural band} by
\[
\mathfrak B_{p,q}
\coloneq
\left\{(N,N+s):N\in\Nzero,\ -\min\{p,N\}\le s\le q\right\}.
\]
These are precisely the positions allowed by the bandwidth and the boundary;
positions outside \(\mathfrak B_{p,q}\) are structural zeros. In this
paper, \emph{active} always refers to the polynomial components defined
before Proposition~\ref{pro:MP-Cauchy-step-line-construction}. The recurrence
matrix need only be entrywise nonnegative on \(\mathfrak B_{p,q}\) for the
Doob normalization below; whenever strict positivity is obtained, it will be
stated explicitly.

\begin{definition}[Structural-band nonnegativity set]
\label{def:MP-structural-band-nonnegativity-set}
Define
\begin{equation}
\mathcal T^{\mathrm P,\ge0}_{p,q}
\coloneq
\left\{(\boldsymbol\alpha,\boldsymbol\beta)\in\mathcal A^{\mathrm P}_{p,q}:
(T^{\mathrm P})_{N,M}\ge0
\text{ for every }(N,M)\in\mathfrak B_{p,q}\right\}.
\end{equation}
\end{definition}

Positive bidiagonal factorization implies strict positivity
throughout the structural band. Indeed, in the product \eqref{eq:MP-PBF},
every structurally allowed entry of \(T^{\mathrm P}\) contains at least one
path whose factor entries are all positive.
Theorem~\ref{thm:MP-PBF-region} therefore gives
\begin{equation}\label{eq:MP-basic-positivity-inclusions}
\mathcal R^{\mathrm P,+}_{p,q}
=\mathcal P^{\mathrm P}_{p,q}
\subseteq\mathcal T^{\mathrm P,\ge0}_{p,q}.
\end{equation}

The recurrence matrix can remain entrywise positive beyond the region where
the prescribed factorization is positive. Two enlarged regions will be
used. The first does not impose an order on the \(\alpha\)-parameters.
\begin{definition}\label{def:MP-matrix-positive-region}
Define
\begin{equation}\label{eq:MP-matrix-positive-region}
\mathcal S^{\mathrm P,+}_{p,q}:=
\left\{(\boldsymbol\alpha,\boldsymbol\beta)\in\mathcal A^{\mathrm P}_{p,q}:
\max_i\alpha_i-\min_i\alpha_i<1,
\quad
\beta_1<\cdots<\beta_q<\beta_1+1
\right\}.
\end{equation}
\end{definition}

The second region is obtained from the transposed mixed system, but a sign
condition is needed in order to return to the original monic normalization.
For \(1\le s\le q\), set
\begin{equation}\label{eq:MP-beta-inversion-number}
\operatorname{inv}_s(\boldsymbol\beta)
\coloneq
\#\{h\in\{1,\ldots,s-1\}:\ \beta_h>\beta_s\}.
\end{equation}
Thus \(\operatorname{inv}_s(\boldsymbol\beta)\) is the number of inversions
created by the index \(s\): it counts the earlier beta parameters lying to
the right of \(\beta_s\) when
\(\beta_1,\ldots,\beta_s\) are arranged in increasing order.

\begin{definition}[Dual diameter and even-insertion regions]
\label{def:MP-dual-positive-region}
Define
\begin{align}
\mathcal D^{\mathrm P,\vee}_{p,q}
&\coloneq
\left\{(\boldsymbol\alpha,\boldsymbol\beta)\in\mathcal A^{\mathrm P}_{p,q}:
\alpha_1<\cdots<\alpha_p<\alpha_1+1,
\quad
\max_j\beta_j-\min_j\beta_j<1
\right\},
\label{eq:MP-dual-diameter-region}
\\
\mathcal S^{\mathrm P,\vee,+}_{p,q}
&\coloneq
\left\{(\boldsymbol\alpha,\boldsymbol\beta)\in
\mathcal D^{\mathrm P,\vee}_{p,q}:
\operatorname{inv}_s(\boldsymbol\beta)\equiv0\pmod 2
\text{ for }1\le s\le q
\right\}.
\label{eq:MP-dual-positive-region}
\end{align}
\end{definition}

The word \emph{insertion} has a direct combinatorial meaning here. Insert
the labels \(1,2,\ldots,q\), one at a time, into the increasing order of
their beta values. At stage \(s\), condition
\eqref{eq:MP-dual-positive-region} requires an even number of the preceding
labels to remain to the right of \(s\). Consequently the number of allowed
relative orders is
\begin{equation}\label{eq:MP-even-insertion-count}
\prod_{s=1}^{q}\left\lceil\frac{s}{2}\right\rceil
=
\left(\left\lfloor\frac{q}{2}\right\rfloor\right)!
\left(\left\lceil\frac{q}{2}\right\rceil\right)!.
\end{equation}
For example, when \(q=4\), the four possible increasing orders of the
labels are
\((1,2,3,4)\), \((1,4,2,3)\), \((3,1,2,4)\), and
\((3,4,1,2)\). Thus the increasing beta order is only one of the allowed
orders.

\begin{theorem}[Two enlarged regions of structural-band positivity]
\label{thm:MP-matrix-positive-region}
For every \(p,q\ge1\),
\begin{equation}
\mathcal R^{\mathrm P,+}_{p,q}
\subseteq
\mathcal S^{\mathrm P,+}_{p,q}
\cap
\mathcal S^{\mathrm P,\vee,+}_{p,q},
\qquad
\mathcal S^{\mathrm P,+}_{p,q}
\cup
\mathcal S^{\mathrm P,\vee,+}_{p,q}
\subseteq
\mathcal T^{\mathrm P,\ge0}_{p,q}.
\label{eq:MP-two-positive-regions}
\end{equation}
On \(\mathcal S^{\mathrm P,+}_{p,q}\), more precisely, the grouped lower factor
\(\mathcal L^{\mathrm P}:=L_1^{\mathrm P}\cdots L_p^{\mathrm P}\) is
totally nonnegative and has strictly positive entries throughout its lower
band, while every upper factor \(U_b^{\mathrm P}\) is positive. Consequently
every entry of \(T^{\mathrm P}\) in its structural band is strictly positive.

The dual assertion is exact within the diameter domain
\(\mathcal D^{\mathrm P,\vee}_{p,q}\):
\begin{equation}\label{eq:MP-dual-positive-region-exact}
\mathcal D^{\mathrm P,\vee}_{p,q}
\cap\mathcal T^{\mathrm P,\ge0}_{p,q}
=
\mathcal S^{\mathrm P,\vee,+}_{p,q}.
\end{equation}
\end{theorem}

The proof is divided into three lemmas. The first treats the grouped lower
factor, the second the upper Christoffel factors, and the third the diagonal
relation with the transposed mixed system.

\begin{lemma}[Grouped lower factor on the alpha-diameter region]
\label{lem:MP-grouped-lower-positive-region}
If
\((\boldsymbol\alpha,\boldsymbol\beta)\in\mathcal S^{\mathrm P,+}_{p,q}\), then the
grouped lower factor
\(\mathcal L^{\mathrm P}:=L_1^{\mathrm P}\cdots L_p^{\mathrm P}\)
is totally nonnegative and has strictly positive entries throughout its
lower band:
\begin{equation}\label{eq:MP-grouped-lower-structural-positive}
(\mathcal L^{\mathrm P})_{N,N-i}>0,
\qquad 1\le i\le\min\{p,N\}.
\end{equation}
\end{lemma}

\begin{proof}
	Fix
	\((\boldsymbol\alpha,\boldsymbol\beta)\in\mathcal S^{\mathrm P,+}_{p,q}\).
	Using the step-line indices defined in
	\eqref{eq:MP-step-quotients}, recall from
	\eqref{eq:MP-Cauchy-exponent-sequences} that
	\[
	\xi_N=\beta_{s_N}+d_N,
	\qquad
	\eta_N=\alpha_{i_N}+e_N+1.
	\]
	By \eqref{eq:MP-step-line-Cauchy-moment-matrix}, the Cauchy moment matrix
	has entries
	\[
	\mathscr M_{N,M}=\frac{1}{\xi_N+\eta_M}.
	\]
	Let
	\(\mathscr M=L_{\mathscr M}U_{\mathscr M}\)
	be its factorization from
	Proposition~\ref{pro:MP-Cauchy-step-line-construction}.
	The increasing order of \(\boldsymbol\beta\) makes the sequence
	\((\xi_N)_{N\ge0}\) strictly increasing. Moreover, the diameter condition
	on \(\boldsymbol\alpha\) gives
	\begin{equation}\label{eq:MP-eta-block-separation}
		\max_{0\le r<p}\eta_r
		<
		\min_{0\le r<p}\eta_{r+p},
		\qquad
		\eta_{R+p}=\eta_R+1,\quad R\ge0.
	\end{equation}
	
	By \eqref{eq:MP-Cauchy-lower-factor}, write
	\((L_{\mathscr M})_{N,K}=\phi_K(\xi_N)\), where
	\[
	\phi_K(x)=
	\frac{\prod_{r=0}^{K-1}(x-\xi_r)}
	{\prod_{r=0}^{K}(x+\eta_r)}
	\frac{\prod_{r=0}^{K}(\xi_K+\eta_r)}
	{\prod_{r=0}^{K-1}(\xi_K-\xi_r)}.
	\]
	The normalization gives
	\(\phi_K(\xi_K)=1\), while
	\(\phi_K(\xi_r)=0\) for \(r<K\). Thus
	\(\phi_0,\phi_1,\ldots\) form the rational analogue of the Newton basis,
	with poles at
	\(-\eta_0,-\eta_1,\ldots\).
	
	Let \(\mathscr M^{+}\) be the Cauchy moment matrix obtained after a
	complete left shift of the alpha indices. Its pole-parameter sequence is
	\((\eta_{M+p})_{M\ge0}\). Write
	\[
	\mathscr M^{+}
	=
	L_{\mathscr M}^{+}U_{\mathscr M}^{+}
	\]
	for its normalized Gauss--Borel factorization. Since
	\[
	\xi_{N+q}=\xi_N+1,
	\qquad
	\eta_{M+p}=\eta_M+1,
	\]
	one has
	\(\mathscr M^{+}=\Lambda_q\mathscr M\). Therefore
	\[
		\mathscr M^{+}=
		\Lambda_qL_{\mathscr M}U_{\mathscr M} =
		L_{\mathscr M}T^{\mathrm P}U_{\mathscr M} =
		\bigl(L_{\mathscr M}\mathcal L^{\mathrm P}\bigr)
		\bigl(\mathcal U^{\mathrm P}U_{\mathscr M}\bigr).
	\]
	The first factor in the last product is unit lower triangular and the
	second is upper triangular. Uniqueness of the normalized Gauss--Borel
	factorization consequently gives
	\[
	L_{\mathscr M}^{+}
	=
	L_{\mathscr M}\mathcal L^{\mathrm P},
	\]
	and hence
	\begin{equation}\label{eq:MP-grouped-lower-connection}
		\mathcal L^{\mathrm P}
		=
		L_{\mathscr M}^{-1}L_{\mathscr M}^{+}.
	\end{equation}
	
	Identity \eqref{eq:MP-grouped-lower-connection} identifies
	\(\mathcal L^{\mathrm P}\) as the connection matrix between the two
	Gauss--Borel factors, but it does not by itself determine the signs of its
	entries or minors. It will now be proved that this connection matrix is totally
	nonnegative and strictly positive throughout its lower band.
	
	Consider the ordered pole-parameter sequence
	\(\eta_0,\eta_1,\ldots\)
	that defines the rational Newton basis whose evaluation matrix is
	\(L_{\mathscr M}\); the corresponding poles are
	\(-\eta_0,-\eta_1,\ldots\). Auxiliary rational Newton bases are introduced
	by applying adjacent interchanges to this ordered sequence. This is only
	an auxiliary change of basis: no parameter of the mixed system and no
	entry of the original moment matrix is changed.
	
	Fix \(M\ge0\). Move the initial block
	\(\eta_0,\ldots,\eta_{p-1}\), preserving its internal order, across the
	next \(M+1\) elements of the pole-parameter sequence. Thus
	\[
	\bigl(
	\eta_0,\ldots,\eta_{p-1},
	\eta_p,\ldots,\eta_{M+p},
	\eta_{M+p+1},\ldots
	\bigr)
	\]
	is transformed, by adjacent interchanges, into
	\[
	\bigl(
	\eta_p,\ldots,\eta_{M+p},
	\eta_0,\ldots,\eta_{p-1},
	\eta_{M+p+1},\ldots
	\bigr).
	\]
	The first \(M+1\) pole parameters in the final sequence are therefore
	\(\eta_p,\ldots,\eta_{M+p}\), precisely those defining
	\((L_{\mathscr M}^{+})^{[M]}\).
	
	For any intermediate pole-parameter sequence
	\(\rho=(\rho_0,\rho_1,\ldots)\), define
	\[
	\phi_K^{\rho}(x)
	\coloneq
	\frac{\prod_{r=0}^{K-1}(x-\xi_r)}
	{\prod_{r=0}^{K}(x+\rho_r)}
	\frac{\prod_{r=0}^{K}(\xi_K+\rho_r)}
	{\prod_{r=0}^{K-1}(\xi_K-\xi_r)},
	\qquad K\ge0,
	\]
	with the convention that an empty product equals \(1\). Then
	\[
	\phi_K^{\rho}(\xi_K)=1,
	\qquad
	\phi_K^{\rho}(\xi_r)=0
	\quad (r<K).
	\]
	Hence the evaluation matrix
	\[
	\left[
	\phi_K^{\rho}(\xi_N)
	\right]_{N,K\ge0}
	\]
	is unit lower triangular.
	
	Suppose that two adjacent pole parameters
	\[
	u=\rho_K<v=\rho_{K+1}
	\]
	are interchanged, and denote the resulting sequence by \(\rho'\).
	The two families agree for \(j\ne K\). Putting
	\(\phi_K^{\rho}\) and
	\(\phi_K^{\rho'}\) over their common denominator gives
	\[
	\phi_K^{\rho'}(x)
	=
	\phi_K^{\rho}(x)
	+
	c_K\phi_{K+1}^{\rho}(x),
	\]
	where evaluation at \(x=\xi_{K+1}\) gives
	\[
	c_K
	=
	\phi_K^{\rho}(\xi_{K+1})
	\frac{(\xi_{K+1}-\xi_K)(v-u)}
	{(\xi_K+u)(\xi_{K+1}+v)}.
	\]
	Consequently, the evaluation matrix for \(\rho'\) is obtained from the
	evaluation matrix for \(\rho\)
	by right multiplication by
	\[
	I+c_KE_{K+1,K}.
	\]
	
	It remains to check that \(c_K>0\). The product formula above gives
	\[
	\phi_K^{\rho}(\xi_{K+1})>0,
	\]
	because \(\xi_{K+1}>\xi_r\) for \(0\le r\le K\) and every pole--node
	sum is positive. Moreover,
	\[
	\xi_{K+1}-\xi_K>0,
	\qquad
	v-u>0,
	\qquad
	\xi_K+u>0,
	\qquad
	\xi_{K+1}+v>0.
	\]
	Thus \(c_K>0\). By
	\eqref{eq:MP-eta-block-separation}, every adjacent interchange used in
	the preceding block displacement has \(u<v\). Each elementary connection
	matrix is therefore unit lower bidiagonal with a strictly positive
	subdiagonal entry, and is totally nonnegative.
	
	Let \(\widehat C_M\) be the product, in the order in which the
	interchanges are performed, of all these elementary connection matrices.
	It acts on the full index set and differs from the identity only in a
	finite leading block. Put
	\[
	C^{[M]}
	\coloneq
	(\widehat C_M)^{[M]}.
	\]
	Evaluating the initial and final rational Newton bases at
	\(\xi_0,\ldots,\xi_M\) gives
	\[
	(L_{\mathscr M}^{+})^{[M]}
	=
	L_{\mathscr M}^{[M]}C^{[M]}.
	\]
	Since \(L_{\mathscr M}^{[M]}\) is unit lower triangular,
	\[
	C^{[M]}
	=
	\bigl(L_{\mathscr M}^{[M]}\bigr)^{-1}
	(L_{\mathscr M}^{+})^{[M]}.
	\]
	For lower triangular matrices, multiplication and inversion commute with
	taking leading principal blocks. Therefore
	\[
	C^{[M]}
	=
	\bigl(
	L_{\mathscr M}^{-1}L_{\mathscr M}^{+}
	\bigr)^{[M]}.
	\]
	The matrix \(\widehat C_M\), and hence its leading block \(C^{[M]}\), is
	totally nonnegative. Since \(M\) is arbitrary, every finite minor of
	\(L_{\mathscr M}^{-1}L_{\mathscr M}^{+}\) is contained in a sufficiently
	large leading block of this form. Thus
	\[
	\mathcal L^{\mathrm P}
	=
	L_{\mathscr M}^{-1}L_{\mathscr M}^{+}
	\]
	is totally nonnegative. Moreover, whenever
	\(0\le K\le N\le M\), the \((N,K)\) entry of \(C^{[M]}\) is exactly the
	corresponding entry of
	\(L_{\mathscr M}^{-1}L_{\mathscr M}^{+}\). No limiting argument is
	involved.
	
	It remains to prove strict positivity throughout the lower band. Each
	interchange at position \(k\) contributes an elementary connection matrix
	\[
	I+c_kE_{k+1,k},
	\qquad c_k>0.
	\]
	Its subdiagonal entry may be viewed as a directed edge from column \(k\)
	to row \(k+1\). Consequently, every entry of the product of the elementary
	connection matrices is a sum of nonnegative path weights. It is enough to
	exhibit one path with strictly positive weight for each position in the
	lower band.
	
	Perform the block displacement in \(p\) passes. In pass \(j\), move
	\(\eta_{p-j}\) to the right across the next \(M+1\) elements. The
	adjacent interchanges in this pass occur successively at the positions
	\[
	p-j,p-j+1,\ldots,M+p-j,
	\qquad 1\le j\le p.
	\]
	Fix \(N\le M\) and
	\(1\le i\le\min\{p,N\}\). For each \(t=1,\ldots,i\), select the
	interchange at position \(N-t\) in pass
	\[
	j_t
	\coloneq
	\max\{t,p-N+t\}
	=
	t+\max\{0,p-N\}.
	\]
	The pass numbers satisfy
	\[
	1\le j_1<\cdots<j_i\le p.
	\]
	Moreover,
	\[
	p-j_t\le N-t\le M+p-j_t,
	\]
	so every selected interchange occurs in the stated pass.
	
	In chronological order, the selected positions are
	\[
	N-1,N-2,\ldots,N-i.
	\]
	Matrix products are read from right to left when following a path from a
	column index to a row index. Hence the selected subdiagonal entries are
	traversed in the reverse order
	\[
	N-i,N-i+1,\ldots,N-1
	\]
	and produce the path
	\[
	N-i
	\longrightarrow
	N-i+1
	\longrightarrow
	\cdots
	\longrightarrow
	N.
	\]
	All the selected subdiagonal entries are strictly positive, while every
	remaining elementary factor contributes its unit diagonal entry. This
	path therefore has strictly positive weight. Since all other path weights
	are nonnegative, the \((N,N-i)\) entry of \(C^{[M]}\), and hence of
	\(\mathcal L^{\mathrm P}\), is strictly positive. This proves
	\eqref{eq:MP-grouped-lower-structural-positive}.
\end{proof}
\begin{lemma}[Positivity of the upper Christoffel factors]
\label{lem:MP-upper-positive-region}
If
\((\boldsymbol\alpha,\boldsymbol\beta)\in\mathcal S^{\mathrm P,+}_{p,q}\), then every
upper bidiagonal factor \(U_b^{\mathrm P}\), \(1\le b\le q\), is positive.
\end{lemma}

\begin{proof}
For the \(b\)-th right Christoffel shift, let
\(\xi_N^{(b)}=\xi_{N+b}\) and let \(D_N^{(b)}\) be the determinant of the
leading \((N+1)\)-square block of the shifted Cauchy matrix. The Cauchy
determinant formula reads
\[
D_N^{(b)}=
\frac{\prod_{0\le r<s\le N}
(\xi_s^{(b)}-\xi_r^{(b)})(\eta_s-\eta_r)}
{\prod_{r,s=0}^{N}(\xi_r^{(b)}+\eta_s)}.
\]
With \(D_{-1}^{(b)}:=1\), put
\(\kappa_N^{(b)}=D_N^{(b)}/D_{N-1}^{(b)}\). Then
\((U_b^{\mathrm P})_{N,N}=\kappa_N^{(b)}/\kappa_N^{(b-1)}\). The
quotient can be displayed explicitly as
\[
\frac{\kappa_N^{(b)}}{\kappa_N^{(b-1)}}
=
\prod_{r=0}^{N-1}
\frac{\xi_{N+b}-\xi_{r+b}}
     {\xi_{N+b-1}-\xi_{r+b-1}}
\prod_{s=0}^{N}
\frac{\xi_{N+b-1}+\eta_s}
     {\xi_{N+b}+\eta_s}
\prod_{r=0}^{N-1}
\frac{\xi_{r+b-1}+\eta_N}
     {\xi_{r+b}+\eta_N}.
\]
Indeed, the \(\eta\)-difference factors cancel between the two pivots.
Every difference and sum in the displayed quotient is positive. Thus every
diagonal entry of \(U_b^{\mathrm P}\) is positive.  Its superdiagonal
entries are \(1\) by the monic normalization, so every nontrivial entry of
every upper factor is positive.
\end{proof}

\begin{lemma}[Dual diagonal relation and sign rule]
\label{lem:MP-dual-diagonal-sign}
Let
\((\boldsymbol\alpha,\boldsymbol\beta)\in\mathcal D^{\mathrm P,\vee}_{p,q}\), and set
\[
\widetilde T^{\,\mathrm P}
\coloneq T^{\mathrm P}_{q,p}(\boldsymbol\beta,\boldsymbol\alpha).
\]
The leading coefficients \(a_{N,i_N,e_N}\) are nonzero. If
\(D_A:=\diag(a_{N,i_N,e_N})_{N\ge0}\), then
\[
\widetilde T^{\mathrm P}
=D_A^{-1}(T^{\mathrm P})^{\mathsf T}D_A,
\qquad
\operatorname{sgn}a_{N,i_N,e_N}
=(-1)^{\operatorname{inv}_{s_N}(\boldsymbol\beta)}.
\]
\end{lemma}

\begin{proof}
Fix
\((\boldsymbol\alpha,\boldsymbol\beta)\in\mathcal D^{\mathrm P,\vee}_{p,q}\) and
interchange the two parameter families.
To compare the swapped matrix with the original recurrence, first identify the
normalization connecting the two systems.
In the original system, the type-I multi-index is
\(\boldsymbol n_N=\boldsymbol\nu_p(N+1)\). In the swapped system, the number of type-II
components is \(p\), and its type-II multi-index is again
\(\boldsymbol\nu_p(N+1)\). Thus the polynomial components of the original
type-I form have exactly the degree bounds required of the type-II form in
the swapped problem. In particular, the active component is the component
with index \(i_N\), and
\(n_{N,i_N}=e_N+1\), so its degree is \(e_N\). By
\eqref{eq:MP-explicit-A-components}, its leading coefficient is therefore
the already defined coefficient \(a_{N,i_N,e_N}\). Substituting
\(i=i_N\) and \(k=e_N\) into
\eqref{eq:MP-explicit-A-coefficients} gives
\begin{equation}\label{eq:MP-dual-leading-coefficient-explicit}
a_{N,i_N,e_N}
=
C_N^{A}
\frac{
\prod_{j=1}^{q}
(\alpha_{i_N}+\beta_j+e_N+1)_{m_{N,j}}
}{
e_N!(-1)^{e_N}
\prod_{\substack{h=1\\h\ne i_N}}^{p}
(\alpha_h-\alpha_{i_N}-e_N)_{n_{N,h}}
}.
\end{equation}
All these leading coefficients are nonzero on the admissible domain. Indeed,
the Pochhammer symbols involving alpha--beta sums have positive factors,
whereas the alpha- and beta-difference Pochhammer symbols cannot vanish by
the nonintegrality assumptions in
\eqref{eq:MP-admissible-domain}. The
forms
\(A_N^{\mathrm P}/a_{N,i_N,e_N}\) are precisely the monic type-II forms
of the swapped system. Indeed, transposing the matrix of measures exchanges
its row and column families, so the type-I orthogonality conditions for the
original system become the type-II orthogonality conditions for the swapped
one. Proposition~\ref{pro:MP-Cauchy-step-line-construction}, through the
nonvanishing leading Cauchy minors, shows that these conditions determine
the form up to a nonzero scalar. Division
by \(a_{N,i_N,e_N}\) fixes that scalar because it makes the active component
monic.

The matrix relation can now be derived explicitly.
The multiplication recurrence in the original left family is
\begin{equation*}
xA_M^{\mathrm P}
=
\sum_{N\ge0}(T^{\mathrm P})_{N,M}A_N^{\mathrm P}.
\end{equation*}
Put
\(\widetilde B_N^{\mathrm P}:=A_N^{\mathrm P}/a_{N,i_N,e_N}\). Dividing
the preceding identity by \(a_{M,i_M,e_M}\) and replacing
\(A_N^{\mathrm P}\) by
\(a_{N,i_N,e_N}\widetilde B_N^{\mathrm P}\) gives
\begin{equation*}
x\widetilde B_M^{\mathrm P}
=
\sum_{N\ge0}
\frac{a_{N,i_N,e_N}}{a_{M,i_M,e_M}}
(T^{\mathrm P})_{N,M}\widetilde B_N^{\mathrm P}.
\end{equation*}
By definition, the coefficient of \(\widetilde B_N^{\mathrm P}\) in this
last recurrence is \((\widetilde T^{\mathrm P})_{M,N}\). Hence, if
\(D_A:=\diag(a_{N,i_N,e_N})_{N\ge0}\),
\begin{equation}\label{eq:MP-dual-diagonal-similarity}
\widetilde T^{\mathrm P}
=D_A^{-1}(T^{\mathrm P})^{\mathsf T}D_A,
\qquad
(\widetilde T^{\mathrm P})_{M,N}
=\frac{a_{N,i_N,e_N}}{a_{M,i_M,e_M}}
(T^{\mathrm P})_{N,M}.
\end{equation}

It remains to determine exactly when the diagonal entries of \(D_A\) have
one sign. For \(N=qd_N+s_N-1\), put
\begin{equation}\label{eq:MP-dual-beta-product}
E_N(\boldsymbol\beta)
\coloneq
\prod_{j=1}^{q}
(\beta_j-\beta_{s_N}-d_N)_{m_{N,j}}.
\end{equation}
Before counting signs, recall from
\eqref{eq:MP-type-I-step-indices} that \(\boldsymbol m_N=\boldsymbol\nu_q(N)\).
Since \(N=qd_N+s_N-1\), this means explicitly that
\[
m_{N,j}
=
\begin{cases}
d_N+1,&j<s_N,\\
d_N,&j\ge s_N.
\end{cases}
\]
The negative factors in
\eqref{eq:MP-dual-beta-product} are now counted. If \(j<s_N\), then the factors are
\begin{equation*}
\beta_j-\beta_{s_N}-d_N,\ldots,
\beta_j-\beta_{s_N}-1,\beta_j-\beta_{s_N}.
\end{equation*}
When \(\beta_j<\beta_{s_N}\), the last factor is negative and hence all
\(d_N+1\) factors are negative. When
\(\beta_j>\beta_{s_N}\), the last factor is positive. If \(d_N\ge1\),
its predecessor is smaller than zero because
\(0<\beta_j-\beta_{s_N}<1\), and all earlier factors are smaller still;
if \(d_N=0\), there is no preceding factor. Thus in either case exactly
the first \(d_N\) factors are negative. Among the \(s_N-1\) indices
preceding \(s_N\), precisely
\(\operatorname{inv}_{s_N}(\boldsymbol\beta)\) satisfy the second alternative. The number of
negative factors contributed by the indices \(j<s_N\) is consequently
\begin{equation*}
(s_N-1)(d_N+1)-\operatorname{inv}_{s_N}(\boldsymbol\beta).
\end{equation*}
If \(d_N=0\), every Pochhammer symbol with \(j\ge s_N\) is empty and
contributes no negative factor. If \(d_N\ge1\), then for \(j=s_N\) the
\(d_N\) factors are
\(-d_N,\ldots,-1\). For \(j>s_N\), they are
\begin{equation*}
\beta_j-\beta_{s_N}-d_N,\ldots,
\beta_j-\beta_{s_N}-1.
\end{equation*}
The last of these is negative because
\(\beta_j-\beta_{s_N}<1\), and every preceding factor is still smaller.
Thus each index \(j\ge s_N\) contributes exactly \(d_N\) negative factors,
for a total contribution of \((q-s_N+1)d_N\). Adding both contributions
and using \(qd_N+s_N-1=N\) gives
\begin{equation}\label{eq:MP-dual-beta-product-sign}
\operatorname{sgn}E_N(\boldsymbol\beta)
=(-1)^{N-\operatorname{inv}_{s_N}(\boldsymbol\beta)}.
\end{equation}

The corresponding alpha-difference product in
\eqref{eq:MP-dual-leading-coefficient-explicit} is
\begin{equation*}
F_N(\boldsymbol\alpha)
\coloneq
\prod_{\substack{h=1\\h\ne i_N}}^{p}
(\alpha_h-\alpha_{i_N}-e_N)_{n_{N,h}}.
\end{equation*}
Here the counting is similar but simpler because the alpha parameters are
in increasing cyclic order. If \(h<i_N\), then
\(-1<\alpha_h-\alpha_{i_N}<0\), and the \(e_N+1\) factors run from
\(\alpha_h-\alpha_{i_N}-e_N\) to
\(\alpha_h-\alpha_{i_N}\); all are negative. If \(h>i_N\), then
\(0<\alpha_h-\alpha_{i_N}<1\). When \(e_N=0\), the corresponding
Pochhammer symbol is empty; when \(e_N\ge1\), its factors run from
\(\alpha_h-\alpha_{i_N}-e_N\) to
\(\alpha_h-\alpha_{i_N}-1\) and are all negative. Hence their total number
is
\begin{equation*}
(i_N-1)(e_N+1)+(p-i_N)e_N=N-e_N,
\end{equation*}
and therefore
\begin{equation}\label{eq:MP-dual-alpha-product-sign}
\operatorname{sgn}F_N(\boldsymbol\alpha)=(-1)^{N-e_N}.
\end{equation}
Every Pochhammer symbol involving a sum of an alpha and a beta parameter is
positive. The numerator in the definition of \(C_N^{A}\) is therefore
positive, while its denominator is exactly \(E_N(\boldsymbol\beta)\). Since the
sign of the reciprocal of a nonzero real number is the sign of the number
itself,
\(\operatorname{sgn}C_N^{A}
=\operatorname{sgn}E_N(\boldsymbol\beta)
=(-1)^{N-\operatorname{inv}_{s_N}(\boldsymbol\beta)}\).
In \eqref{eq:MP-dual-leading-coefficient-explicit}, the remaining
factors in the denominator whose signs must be counted are
\((-1)^{e_N}\) and
\(F_N(\boldsymbol\alpha)\). Therefore equations
\eqref{eq:MP-dual-leading-coefficient-explicit},
\eqref{eq:MP-dual-beta-product-sign}, and
\eqref{eq:MP-dual-alpha-product-sign} give, step by step,
\begin{equation*}
\operatorname{sgn}a_{N,i_N,e_N}
=(-1)^{N-\operatorname{inv}_{s_N}(\boldsymbol\beta)}
(-1)^{e_N}(-1)^{N-e_N}
=(-1)^{\operatorname{inv}_{s_N}(\boldsymbol\beta)}.
\end{equation*}
This proves the sign rule
\begin{equation}\label{eq:MP-dual-leading-coefficient-sign}
\operatorname{sgn}a_{N,i_N,e_N}
=(-1)^{\operatorname{inv}_{s_N}(\boldsymbol\beta)}.
\end{equation}
\end{proof}

\begin{proof}[Proof of Theorem~\ref{thm:MP-matrix-positive-region}]
Let
\((\boldsymbol\alpha,\boldsymbol\beta)\in\mathcal S^{\mathrm P,+}_{p,q}\).
Lemma~\ref{lem:MP-grouped-lower-positive-region} gives the total
nonnegativity of \(\mathcal L^{\mathrm P}\) and strict positivity throughout
its lower band, while
Lemma~\ref{lem:MP-upper-positive-region} gives the positivity of every
upper factor. If \((N,M)\in\mathfrak B_{p,q}\) and \(M\ge N\), the
\((N,M)\) entry of
\(T^{\mathrm P}=\mathcal L^{\mathrm P}
U_q^{\mathrm P}\cdots U_1^{\mathrm P}\)
contains the product of the diagonal entry of \(\mathcal L^{\mathrm P}\)
with a positive path from \(N\) to \(M\) through the upper factors. If
\(M<N\), it contains the product of
\((\mathcal L^{\mathrm P})_{N,M}>0\) with the positive diagonal path
through the upper factors. All other contributions are nonnegative.
Thus every structural-band entry of \(T^{\mathrm P}\) is strictly positive,
which proves
\(\mathcal S^{\mathrm P,+}_{p,q}\subseteq
\mathcal T^{\mathrm P,\ge0}_{p,q}\) together with the more precise
assertions in the theorem.

For the dual region, fix
\((\boldsymbol\alpha,\boldsymbol\beta)\in\mathcal D^{\mathrm P,\vee}_{p,q}\) and set
\(\widetilde T^{\,\mathrm P}
\coloneq T^{\mathrm P}_{q,p}(\boldsymbol\beta,\boldsymbol\alpha)\).
The swapped pair
\((\boldsymbol\beta,\boldsymbol\alpha)\) belongs to
\(\mathcal S^{\mathrm P,+}_{q,p}\): its first parameter vector has diameter
less than one and its second parameter vector is in strict cyclic order.
The first part of the proof therefore shows that
\(\widetilde T^{\,\mathrm P}\) is strictly positive throughout its
structural band.
By Lemma~\ref{lem:MP-dual-diagonal-sign}, the leading coefficients are
nonzero, the matrix
\(D_A=\diag(a_{N,i_N,e_N})_{N\ge0}\) satisfies
\eqref{eq:MP-dual-diagonal-similarity}, and the coefficients obey the
sign rule \eqref{eq:MP-dual-leading-coefficient-sign}.

If \((\boldsymbol\alpha,\boldsymbol\beta)\in
\mathcal S^{\mathrm P,\vee,+}_{p,q}\), all entries of \(D_A\) are positive.
The diagonal relation \eqref{eq:MP-dual-diagonal-similarity} then transfers
the strict positivity of every structurally allowed entry of
\(\widetilde T^{\mathrm P}\) to the corresponding structurally allowed
entry of \(T^{\mathrm P}\). Explicitly, for every
\((N,M)\in\mathfrak B_{p,q}\),
\begin{equation*}
(T^{\mathrm P})_{N,M}
=
\frac{a_{M,i_M,e_M}}{a_{N,i_N,e_N}}
(\widetilde T^{\mathrm P})_{M,N}>0,
\end{equation*}
because both leading coefficients and the transposed structural-band entry are
positive. This proves the second sufficient inclusion in
\eqref{eq:MP-two-positive-regions}.

Conversely, remain in \(\mathcal D^{\mathrm P,\vee}_{p,q}\) and assume
that \(T^{\mathrm P}\) is nonnegative on its structural band, as required by
the definition of \(\mathcal T^{\mathrm P,\ge0}_{p,q}\). Every
\(a_{N,i_N,e_N}\) is nonzero, and the swapped matrix
\(\widetilde T^{\mathrm P}\) is strictly positive on its structural band.
Therefore \eqref{eq:MP-dual-diagonal-similarity} shows that no
structural-band entry of \(T^{\mathrm P}\) can vanish. In this dual diameter
domain, structural-band nonnegativity consequently implies strict positivity
throughout that band.

In particular, for every \(N\ge0\), the two adjacent entries below are
structurally allowed and strictly positive, and
\eqref{eq:MP-dual-diagonal-similarity} gives
\begin{equation*}
(\widetilde T^{\mathrm P})_{N+1,N}
=\frac{a_{N,i_N,e_N}}
{a_{N+1,i_{N+1},e_{N+1}}}
(T^{\mathrm P})_{N,N+1}>0.
\end{equation*}
Thus consecutive leading coefficients have the same sign. At \(N=0\), one
has \(d_0=e_0=0\), \(s_0=i_0=1\), and
\eqref{eq:MP-explicit-A-coefficients} reduces to
\[
a_{0,1,0}=\alpha_1+\beta_1+1>0
\]
by admissibility. Induction on \(N\) therefore shows that every
\(a_{N,i_N,e_N}\) is positive. Now fix \(s\in\{1,\ldots,q\}\) and take
\(N=s-1\). Since \(0\le N<q\), the decomposition
\(N=qd_N+s_N-1\) has \(d_N=0\) and \(s_N=s\). The sign identity
\eqref{eq:MP-dual-leading-coefficient-sign} becomes
\(1=\operatorname{sgn}a_{s-1,i_{s-1},e_{s-1}}
=(-1)^{\operatorname{inv}_s(\boldsymbol\beta)}\).
Hence \(\operatorname{inv}_s(\boldsymbol\beta)\) is even. Since \(s\) was arbitrary, the
even-insertion condition holds for all beta indices, which establishes
\eqref{eq:MP-dual-positive-region-exact}.

Finally, the ordered cyclic conditions defining
\(\mathcal R^{\mathrm P,+}_{p,q}\) imply both diameter conditions, and
increasing beta order gives \(\operatorname{inv}_s(\boldsymbol\beta)=0\) for all \(s\).
Therefore \(\mathcal R^{\mathrm P,+}_{p,q}\) is contained in both enlarged
regions, completing the proof.
\end{proof}

\begin{example}[Duality and nonmaximality]
\label{ex:MP-asymmetric-positive-region}
	The two enlarged regions in \eqref{eq:MP-two-positive-regions} are not
	the global positivity set. For example,
	\[
	(p,q)=(3,1),\qquad
	\boldsymbol\alpha=\left(\frac34,-\frac12,\frac54\right),\qquad
	\beta_1=0
	\]
	is admissible and direct substitution in the recurrence coefficients
	shows that every structural-band entry of \(T^{\mathrm P}\) is positive.
	Its alpha diameter is greater than one and its alpha parameters are not
	in cyclic order; hence it lies in
	\(\mathcal T^{\mathrm P,\ge0}_{3,1}\setminus
	(\mathcal S^{\mathrm P,+}_{3,1}\cup
	\mathcal S^{\mathrm P,\vee,+}_{3,1})\).
	The Markov normalization below is restricted to
	\(\mathcal S^{\mathrm P,+}_{p,q}\), where the required positivity of
	\(B_N^{\mathrm P}(1)\) is also proved.
\end{example}

The first enlarged region still has a direct Markovian meaning, even though the
individual lower factors are not all positive outside the ordered cyclic
region.  The following result
also supplies the positivity of the vector used in the Doob normalization.

\begin{proposition}[Markov normalization on the enlarged region]
\label{pro:MP-Markov-enlarged-region}
Let \((\boldsymbol\alpha,\boldsymbol\beta)\in\mathcal S^{\mathrm P,+}_{p,q}\), and let
\(B_N^{\mathrm P}\) be the monic type-II step-line form. Then
\(h_N^{\mathrm P}:=B_N^{\mathrm P}(1)>0\) and
\(T^{\mathrm P}h^{\mathrm P}=h^{\mathrm P}\). Consequently, with
\(H^{\mathrm P}=\diag(h_N^{\mathrm P})\), the matrix
\(P^{\mathrm P}:=(H^{\mathrm P})^{-1}T^{\mathrm P}H^{\mathrm P}\) is the
classical Doob \(h\)-transform \cite{Doob1957,Doob1984}; it is row-stochastic
and is strictly positive throughout its structural band.
Thus every
parameter point in \(\mathcal S^{\mathrm P,+}_{p,q}\) defines a banded
Markov chain.  On
\(\mathcal S^{\mathrm P,+}_{p,q}\setminus
\mathcal R^{\mathrm P,+}_{p,q}\), however, the prescribed bidiagonal
factorization is not positive.
\end{proposition}

\begin{proof}
The \(N\) functions \(x^{\alpha_i+k}\), where \(1\le i\le p\) and
\(0\le k\le\widetilde n_{N,i}-1\), have distinct real exponents and form
an extended complete Chebyshev system on \((0,\infty)\); see, for example,
the classical total-positivity treatment in \cite{Karlin}. The usual AT
sign-change argument applied to the type-II orthogonality established in
Proposition~\ref{pro:MP-Cauchy-step-line-construction} therefore forces
\(B_N^{\mathrm P}\) to have at least \(N\) sign changes in \((0,1)\).
On the other hand, \(B_N^{\mathrm P}\) is a nonzero linear combination of
the \(N+1\) distinct generalized monomials \(x^{\beta_j+k}\), with
\(1\le j\le q\) and \(0\le k\le\widetilde m_{N,j}-1\). These monomials
form another extended complete Chebyshev system on \((0,\infty)\), so a
nonzero such combination has at most \(N\) positive zeros, counted with
multiplicity. Consequently, the \(N\) zeros forced by orthogonality are
simple sign-change zeros in \((0,1)\), and there is no further zero on
\((0,\infty)\).

It remains to determine the sign to the right of these zeros. Since
\(N=qd_N+s_N-1\), one has \(\widetilde m_{N,j}=d_N+1\) for \(j\le s_N\)
and \(\widetilde m_{N,j}=d_N\) for \(j>s_N\). Thus the active monic
component contributes \(x^{\beta_{s_N}+d_N}\) with coefficient \(1\). If
\(j<s_N\), then \(\beta_j+d_N<\beta_{s_N}+d_N\); if \(j>s_N\), the cyclic
order gives \(\beta_j+d_N-1<\beta_{s_N}+d_N\). Hence this is the unique
term of largest exponent, and \(B_N^{\mathrm P}(x)>0\) for all sufficiently
large \(x\). Since there is no zero in \([1,\infty)\), it follows that
\(B_N^{\mathrm P}(1)>0\).

The multiplication recurrence evaluated at \(x=1\) gives
\(T^{\mathrm P}h^{\mathrm P}=h^{\mathrm P}\).  Theorem~\ref{thm:MP-matrix-positive-region}
gives entrywise nonnegativity of \(T^{\mathrm P}\), with strict positivity
throughout its structural band.  Positive diagonal conjugation preserves
these signs,
and \(P^{\mathrm P}\one
=(H^{\mathrm P})^{-1}T^{\mathrm P}h^{\mathrm P}=\one\).
Thus \(P^{\mathrm P}\) is row-stochastic.  The final statement follows
because Theorem~\ref{thm:MP-matrix-positive-region} concerns the grouped
lower factor, not the signs of its individual bidiagonal factors.
\end{proof}

\subsection{Stochastic normalization and stochastic bidiagonal factorization}
\label{subsec:MP-stochastic-normalization}
The matrix $T^{\mathrm P}$ is the spectral recurrence matrix. The Markov transition matrix is obtained by the normalization at the spectral endpoint $x=1$. By Proposition~\ref{pro:MP-Markov-enlarged-region}, let
\(h^{\mathrm P}=(h_N^{\mathrm P})_{N\ge0}\), with
\(h_N^{\mathrm P}>0\) and
\(T^{\mathrm P}h^{\mathrm P}=h^{\mathrm P}\), be the positive right eigenvector given by evaluating the normalized step-line forms at $x=1$, and let $H^{\mathrm P}=\diag(h_N^{\mathrm P})$. Define
\(P^{\mathrm P}:=(H^{\mathrm P})^{-1}T^{\mathrm P}H^{\mathrm P}\).
Since $T^{\mathrm P}\ge0$ throughout the enlarged region
\eqref{eq:MP-matrix-positive-region}, this conjugation gives
\(P^{\mathrm P}\one=(H^{\mathrm P})^{-1}T^{\mathrm P}h^{\mathrm P}=\one\), so $P^{\mathrm P}$ is a row-stochastic banded kernel throughout
\(\mathcal S^{\mathrm P,+}_{p,q}\).

The grouped Gaussian factors in
Proposition~\ref{pro:MP-explicit-grouped-LU} admit a direct stochastic
normalization throughout this larger region. Set
\(v^{\mathrm P}:=\mathcal U^{\mathrm P}h^{\mathrm P}\) and
\(D^{\mathrm P}:=\diag(v^{\mathrm P})\); that is,
\(v_N^{\mathrm P}=\sum_{e=0}^{q}
\mathsf U^{\mathrm P}_{N,e}h^{\mathrm P}_{N+e}\). Define
\begin{equation}
\label{eq:MP-grouped-stochastic-factors}
P_{\downarrow}^{\mathrm P}
\coloneq
(H^{\mathrm P})^{-1}\mathcal L^{\mathrm P}D^{\mathrm P},
\qquad
P_{\uparrow}^{\mathrm P}
\coloneq
(D^{\mathrm P})^{-1}\mathcal U^{\mathrm P}H^{\mathrm P}.
\end{equation}
Then
\begin{equation}
\label{eq:MP-grouped-stochastic-LU}
P^{\mathrm P}
=P_{\downarrow}^{\mathrm P}P_{\uparrow}^{\mathrm P},
\end{equation}
and the nonzero entries of the two factors are explicitly
\begin{align}
(P_{\downarrow}^{\mathrm P})_{N,N-d}
&=
\frac{\mathsf L^{\mathrm P}_{N,d}\,v^{\mathrm P}_{N-d}}
{h^{\mathrm P}_N},
&&0\le d\le\min\{p,N\},
\label{eq:MP-grouped-lower-stochastic-explicit}
\\
(P_{\uparrow}^{\mathrm P})_{M,M+e}
&=
\frac{\mathsf U^{\mathrm P}_{M,e}\,h^{\mathrm P}_{M+e}}
{v^{\mathrm P}_M},
&&0\le e\le q.
\label{eq:MP-grouped-upper-stochastic-explicit}
\end{align}
Indeed, Theorem~\ref{thm:MP-matrix-positive-region} gives positivity of
all entries in the lower and upper bands of \(\mathcal L^{\mathrm P}\) and
\(\mathcal U^{\mathrm P}\), respectively, so \(v^{\mathrm P}>0\). Moreover,
\(P_{\uparrow}^{\mathrm P}\one
=(D^{\mathrm P})^{-1}\mathcal U^{\mathrm P}h^{\mathrm P}=\one\) and
\(P_{\downarrow}^{\mathrm P}\one
=(H^{\mathrm P})^{-1}\mathcal L^{\mathrm P}v^{\mathrm P}
=(H^{\mathrm P})^{-1}T^{\mathrm P}h^{\mathrm P}=\one\).
Thus \eqref{eq:MP-grouped-stochastic-LU} is an explicit stochastic
Gauss--Borel factorization with downward jumps of sizes at most \(p\) and
upward jumps of sizes at most \(q\).  For \(p=2,q=1\), it specializes to
the \(P_LP_U\) factorization of the Jacobi--Pi\~neiro urn model in
\cite{GrunbaumDelaIglesia}; formulas
\eqref{eq:MP-grouped-lower-stochastic-explicit} and
\eqref{eq:MP-grouped-upper-stochastic-explicit} give its general mixed
\((p,q)\) counterpart.

On the smaller PBF region \(\mathcal R^{\mathrm P,+}_{p,q}\), the
stochastic normalization can moreover be performed factor by factor.  This
is the construction of \cite[Proposition~2.3]{BFM6}, now applied to the
explicit mixed Pi\~neiro factors. Put
\(M_1=L_1^{\mathrm P},\ldots,M_p=L_p^{\mathrm P}\) and
\(M_{p+1}=U_q^{\mathrm P},\ldots,M_{p+q}=U_1^{\mathrm P}\), so that
$T^{\mathrm P}=M_1\cdots M_{p+q}$. Set $h^{[p+q]}=h^{\mathrm P}$ and define recursively
\(h^{[r-1]}=M_rh^{[r]}\), for
\(r=p+q,p+q-1,\ldots,1\).
Then $h^{[0]}=h^{\mathrm P}$. If $H_r=\diag(h^{[r]})$, the matrices
\(\widehat M_r:=H_{r-1}^{-1}M_rH_r\),
\(r\in\{1,\ldots,p+q\}\), are stochastic bidiagonal matrices, because
\(\widehat M_r\one
=H_{r-1}^{-1}M_rh^{[r]}
=H_{r-1}^{-1}h^{[r-1]}=\one\).
Consequently
\begin{equation}\label{eq:MP-SBF}
P^{\mathrm P}
=
\widehat L_1^{\mathrm P}\cdots \widehat L_p^{\mathrm P}
\widehat U_q^{\mathrm P}\cdots \widehat U_1^{\mathrm P},
\end{equation}
where each factor is bidiagonal, stochastic, and has two positive entries in every non-boundary row.  More explicitly, for a lower factor, with the convention that the entry with column index \(-1\) is absent,
\begin{align}
(\widehat L_k^{\mathrm P})_{N,N}
&=
\frac{h_N^{[k]}}{h_N^{[k-1]}},
&
(\widehat L_k^{\mathrm P})_{N,N-1}
&=
\frac{\ell_{k,N-1}^{\mathrm P}h_{N-1}^{[k]}}{h_N^{[k-1]}},
\qquad k\in\{1,\ldots,p\},
\label{eq:MP-stochastic-lower-explicit}
\end{align}
and, for the upper factor \(U_b^{\mathrm P}\), which occupies the position
\(r=p+q-b+1\) in the ordered product,
\begin{align}
(\widehat U_b^{\mathrm P})_{N,N}
&=
\frac{u_{b,N}^{\mathrm P}h_N^{[r]}}{h_N^{[r-1]}},
&
(\widehat U_b^{\mathrm P})_{N,N+1}
&=
\frac{h_{N+1}^{[r]}}{h_N^{[r-1]}},
\qquad b\in\{1,\ldots,q\}.
\label{eq:MP-stochastic-upper-explicit}
\end{align}
The quantities \(\ell_{a,N}^{\mathrm P}\) and \(u_{b,N}^{\mathrm P}\) are the simplified Pochhammer products in \eqref{eq:MP-lower-Pochhammer-simplified} and \eqref{eq:MP-upper-Pochhammer-simplified}; hence the stochastic factors are the explicit Doob normalizations of those bidiagonal coefficients. This is the stochastic bidiagonal factorization used in this paper and in \cite{BFM6}.

The original recurrence matrix need not be bounded even though its
Doob-normalized matrix is stochastic.  The distinction is determined
exactly by the cyclic spacings of the beta parameters.

\begin{proposition}[Boundedness of the Pi\~neiro recurrence]
\label{prop:MP-boundedness-criterion}
Assume the ordered Pi\~neiro PBF conditions and put
\begin{equation}
\label{eq:MP-cyclic-beta-spacings}
\vartheta_j^{\mathrm P}
\coloneq
\beta_{j+1}-\beta_j
\quad (1\le j<q),
\qquad
\vartheta_q^{\mathrm P}
\coloneq
1+\beta_1-\beta_q.
\end{equation}
For \(b\in\{1,\ldots,q\}\), let
\(\rho(N,b)\in\{1,\ldots,q\}\) be determined by
\(\rho(N,b)\equiv N+b\pmod q\).  Along every fixed residue class
\(N\equiv r\pmod{\operatorname{lcm}(p,q)}\),
\begin{equation}
\label{eq:MP-upper-factor-asymptotic}
u_{b,N}^{\mathrm P}
=
C_{b,r}^{\mathrm P}
N^{q\vartheta_{\rho(N,b)}^{\mathrm P}-1}
\bigl(1+\mathrm O(N^{-1})\bigr),
\qquad C_{b,r}^{\mathrm P}>0.
\end{equation}
For every \(a\in\{1,\ldots,p\}\), along the same residue class,
\begin{equation}
\label{eq:MP-lower-factor-asymptotic}
\ell_{a,N}^{\mathrm P}
=
\widetilde C_{a,r}^{\mathrm P}
N^{q\vartheta_{s_N}^{\mathrm P}-1}
\bigl(1+\mathrm O(N^{-1})\bigr),
\qquad \widetilde C_{a,r}^{\mathrm P}>0,
\end{equation}
where \(s_N\) is determined by \(N=qd_N+s_N-1\).  Consequently,
\begin{equation}
\label{eq:MP-boundedness-equivalence}
T^{\mathrm P}\text{ is bounded on }\ell^s(\mathbb N_0)
\text{ for every }1\le s\le\infty
\quad\Longleftrightarrow\quad
\vartheta_1^{\mathrm P}=\cdots=\vartheta_q^{\mathrm P}=\frac1q.
\end{equation}
Regardless of this alternative, \(P^{\mathrm P}\) is bounded on every
\(\ell^s\).  If \(T^{\mathrm P}\) is unbounded, the identity
\(P^{\mathrm P}=(H^{\mathrm P})^{-1}T^{\mathrm P}H^{\mathrm P}\)
is therefore not a bounded similarity.
\end{proposition}

\begin{proof}
Fix a residue class modulo \(\operatorname{lcm}(p,q)\).  On that class all
step-line indices occurring in
\eqref{eq:MP-lower-Pochhammer-simplified} and
\eqref{eq:MP-upper-Pochhammer-simplified} are affine functions of \(N\).
Write each Pochhammer symbol as a quotient of Gamma functions and use
\[
\frac{\Gamma(x+c)}{\Gamma(x+d)}
=x^{c-d}\bigl(1+\mathrm O(x^{-1})\bigr).
\]
Put \(x_N=\beta_{s_N}+d_N\) and extend
\(\beta_{q+1}=\beta_1+1\).  Then
\(x_{N+1}-x_N=\vartheta_{s_N}^{\mathrm P}\).  In the first quotient of
\eqref{eq:MP-lower-Pochhammer-simplified}, the identity
\[
(\beta_j-\beta_{s_N}-d_N)_{d_N}
=(-1)^{d_N}
\frac{\Gamma(d_N+1+\beta_{s_N}-\beta_j)}
{\Gamma(1+\beta_{s_N}-\beta_j)}
\]
shows that the \(q\) Gamma quotients contribute
\(N^{q\vartheta_{s_N}^{\mathrm P}}\).  In the quotient of the alpha
Pochhammer products, the affine cyclic identity
\[
\sum_{i=1}^{p}\alpha_{a+i}
=\sum_{i=1}^{p}\alpha_{a-1+i}+1
\]
together with
\(\lvert\boldsymbol n_{N+1}\rvert-\lvert\boldsymbol n_N\rvert=1\)
leaves the net power \(N^{-1}\); the remaining rational factor has a
finite positive limit.  This proves
\eqref{eq:MP-lower-factor-asymptotic}.

For the upper factor, collecting the Gamma powers leaves the cyclic
difference between the two beta parameters selected by \(N+b\).  The net
power is
\[
q\bigl(\beta_{\rho+1}-\beta_\rho\bigr)-1
\quad(\rho<q),
\qquad
q\bigl(1+\beta_1-\beta_q\bigr)-1
\quad(\rho=q),
\]
which proves \eqref{eq:MP-upper-factor-asymptotic}.  Notice that
\(\sum_{j=1}^q\vartheta_j^{\mathrm P}=1\).

If all cyclic spacings equal \(1/q\), both displayed asymptotics have
exponent zero, so every bidiagonal factor has uniformly bounded entries.
Their finite product is a band matrix with uniformly
bounded entries and hence defines a bounded operator on every \(\ell^s\).
If the spacings are not all equal, one of them is greater than \(1/q\).
Equation~\eqref{eq:MP-upper-factor-asymptotic} then gives an unbounded
subsequence in an upper factor.  The \((q-1)\)-st superdiagonal of
\(U_q^{\mathrm P}\cdots U_1^{\mathrm P}\) is a sum of the upper-factor
diagonal coefficients at shifted indices, all with positive coefficients.
Multiplication on the left by the unit lower factors only adds nonnegative
terms.  Thus an entry sequence of \(T^{\mathrm P}\) is unbounded.

Finally, a row-stochastic band matrix has entries at most one, row sums one,
and uniformly bounded column sums because each column meets only finitely
many rows.  Hence \(P^{\mathrm P}\) is bounded on \(\ell^1\) and
\(\ell^\infty\), and interpolation gives boundedness on the intermediate
spaces.  A bounded invertible diagonal similarity could not conjugate an
unbounded operator to this bounded matrix.
\end{proof}

All quantities used in the construction have now been given explicitly.
Given \(p,q,\boldsymbol\alpha,\boldsymbol\beta\) and a state \(N\), equations
\eqref{eq:MP-type-I-step-indices} and
\eqref{eq:MP-type-II-step-indices} determine the two multi-indices;
\eqref{eq:MP-explicit-type-I-form} and
\eqref{eq:MP-explicit-type-II-form} define the corresponding mixed forms;
\eqref{eq:MP-explicit-A-coefficients},
\eqref{eq:MP-raw-B-coefficients}, and
\eqref{eq:MP-h-explicit-alpha-beta} compute the two mixed forms and the
harmonic vector; equations
\eqref{eq:MP-lower-Pochhammer-simplified} and
\eqref{eq:MP-upper-Pochhammer-simplified} compute the bidiagonal factors;
\eqref{eq:MP-lower-bidiagonal-factors} and
\eqref{eq:MP-upper-bidiagonal-factors} place those coefficients in the
factor matrices;
and equations \eqref{eq:MP-stochastic-lower-explicit} and
\eqref{eq:MP-stochastic-upper-explicit} compute every elementary
transition probability. Finite multiplication then gives any requested entry of
\(T^{\mathrm P}\) or \(P^{\mathrm P}\). If the parameters are rational,
all these quantities are rational and are obtained by exact arithmetic.

\subsection{Spectral representation after stochastic normalization}
\label{subsec:MP-discrete-spectral}
The mixed Pi\~neiro spectral representation is preserved by the stochastic
normalization: it is conjugated by the positive harmonic vector. Unlike the
general Favard representation, no measure has to be recovered here: the polynomial
vectors are given by
\eqref{eq:MP-explicit-A-components} and
\eqref{eq:MP-explicit-B-components}, while every entry of the spectral
measure is the elementary beta density
\(x^{\alpha_a+\beta_b}\,\dd x\) from
\eqref{eq:MP-matrix-measure} and
\eqref{eq:MP-spectral-measure-explicit}. Consequently, the multiplication
matrix satisfies the completely explicit formula
\begin{equation}
\label{eq:MP-Tn-spectral}
(T^{\mathrm P})^n_{N,M}
=
\sum_{a=1}^{p}\sum_{b=1}^{q}
\int_0^1
B_N^{(b),\mathrm P}(x)
A_M^{(a),\mathrm P}(x)
x^{n+\alpha_a+\beta_b}\,\dd x,
\qquad n\in\Nzero,
\end{equation}
in which both polynomial factors are finite sums with coefficients expressed
by Pochhammer products. Since
\(P^{\mathrm P}=(H^{\mathrm P})^{-1}T^{\mathrm P}H^{\mathrm P}\), with
\(H^{\mathrm P}=\diag(h_N^{\mathrm P})\), one has
\((P^{\mathrm P})^n
=(H^{\mathrm P})^{-1}(T^{\mathrm P})^nH^{\mathrm P}\).
Therefore
\begin{equation}
\label{eq:MP-Pn-spectral}
(P^{\mathrm P})^n_{N,M}
=
\frac{h_M^{\mathrm P}}{h_N^{\mathrm P}}
\sum_{a=1}^{p}\sum_{b=1}^{q}
\int_0^1
B_N^{(b),\mathrm P}(x)
A_M^{(a),\mathrm P}(x)
x^{n+\alpha_a+\beta_b}\,\dd x.
\end{equation}
Equivalently, define the Doob-normalized left and right polynomial vectors by
\begin{equation}
\widetilde B_N^{(b),\mathrm P}(x)
\coloneq
\frac{B_N^{(b),\mathrm P}(x)}{h_N^{\mathrm P}},
\qquad
\widetilde A_M^{(a),\mathrm P}(x)
\coloneq
h_M^{\mathrm P}A_M^{(a),\mathrm P}(x).
\end{equation}
Then the transition probabilities of the discrete-time chain are given by the
Karlin--McGregor type formula
\begin{equation}\label{eq:MP-Pn-spectral-normalized}
(P^{\mathrm P})^n_{N,M}
=
\sum_{a=1}^{p}\sum_{b=1}^{q}
\int_0^1
\widetilde B_N^{(b),\mathrm P}(x)
\widetilde A_M^{(a),\mathrm P}(x)
x^{n+\alpha_a+\beta_b}\,\dd x.
\end{equation}
Thus the mixed Pi\~neiro polynomials provide the spectral interpretation of the
discrete-time Markov chain itself, not only of the underlying recurrence
matrix. The stochastic normalization changes the polynomial normalization on
the two sides, but it leaves the spectral variable and the explicit matrix
of beta densities unchanged. Every transition probability is therefore a
finite sum of beta integrals and can, for fixed \(n,N,M\), be evaluated as a
finite combination of rational functions of the parameters.

\begin{corollary}[Evaluation of the discrete spectral integrals]
\label{coro:MP-discrete-spectral-explicit}
For \(N,M,n\in\Nzero\), the entries of the recurrence matrix and of the
discrete-time Markov kernel are
\begin{align}
\label{eq:MP-Tn-spectral-explicit}
(T^{\mathrm P})^n_{N,M}
&=
\frac{1}{\widehat\kappa_N}
\sum_{a=1}^{p}\sum_{b=1}^{q}
\sum_{r=0}^{\widetilde m_{N,b}-1}
\sum_{s=0}^{n_{M,a}-1}
\frac{c_{N,b,r}\,a_{M,a,s}}
{\alpha_a+\beta_b+n+r+s+1},
\\
\label{eq:MP-Pn-spectral-explicit}
(P^{\mathrm P})^n_{N,M}
&=
\frac{h_M^{\mathrm P}}
{h_N^{\mathrm P}\widehat\kappa_N}
\sum_{a=1}^{p}\sum_{b=1}^{q}
\sum_{r=0}^{\widetilde m_{N,b}-1}
\sum_{s=0}^{n_{M,a}-1}
\frac{c_{N,b,r}\,a_{M,a,s}}
{\alpha_a+\beta_b+n+r+s+1},
\end{align}
where \(a_{M,a,s}\) and \(c_{N,b,r}\) are given by
\eqref{eq:MP-explicit-A-coefficients} and
\eqref{eq:MP-raw-B-coefficients}, respectively. Moreover,
\(\widehat\kappa_N=c_{N,s_N,d_N}\), while
\(h_K^{\mathrm P}=B_K^{\mathrm P}(1)\), for \(K\ge0\), is given
explicitly by \eqref{eq:MP-h-explicit-alpha-beta}.
As usual, a sum whose upper limit is \(-1\) is empty. In particular, for
fixed \(N,M,n\), both entries are rational functions of
\(\boldsymbol\alpha\) and \(\boldsymbol\beta\); when the parameters are rational, the
entries are rational numbers.
\end{corollary}

\begin{proof}
Substitute the finite expansions
\eqref{eq:MP-explicit-A-components} and
\eqref{eq:MP-explicit-B-components} into
\eqref{eq:MP-Tn-spectral}. Each summand is evaluated by
\[
\int_0^1
x^{\alpha_a+\beta_b+n+r+s}\,\dd x
=
\frac{1}{\alpha_a+\beta_b+n+r+s+1}.
\]
This gives \eqref{eq:MP-Tn-spectral-explicit}. The Doob relation
\((P^{\mathrm P})^n_{N,M}
=(h_M^{\mathrm P}/h_N^{\mathrm P})(T^{\mathrm P})^n_{N,M}\)
then gives \eqref{eq:MP-Pn-spectral-explicit}. Rationality follows from
the explicit Pochhammer-product formulas
\eqref{eq:MP-explicit-A-coefficients},
\eqref{eq:MP-raw-B-coefficients}, and
\eqref{eq:MP-h-explicit-alpha-beta}.
\end{proof}

\section{Markov processes associated with the mixed Pi\~neiro system}
\label{sec:mixed-Pineiro-Markov}

Section~\ref{sec:mixed-Pineiro-model} produces one explicit stochastic
banded kernel. Here it is used to define a discrete-time chain, two finite-urn
constructions, a bounded continuous-time chain, a QBD obtained by grouping
consecutive states, and an unbounded chain with state-dependent rates.
All these constructions use the same Pi\~neiro transition kernel; only the
time scale or the labeling of the states changes.

\subsection{The discrete-time Markov chain and the elementary factor transitions}
\label{subsec:MP-discrete-chain}
The stochastic matrix $P^{\mathrm P}$ defines a discrete-time Markov chain
\(X^{\mathrm P}=(X^{\mathrm P}_n)_{n\ge0}\), whose transition law is
\(\Pr\{X^{\mathrm P}_{r+1}=M\mid X^{\mathrm P}_r=N\}
=(P^{\mathrm P})_{N,M}\). Since $P^{\mathrm P}$ has bandwidth $(p,q)$,
this probability vanishes whenever \(M-N\notin\{-p,\ldots,q\}\).
For admissible jumps one has the Doob-normalized formula
\begin{equation}
p^{\mathrm P}_{N,N+s}
\coloneq
(P^{\mathrm P})_{N,N+s}
=
\frac{h^{\mathrm P}_{N+s}}{h^{\mathrm P}_{N}}
(T^{\mathrm P})_{N,N+s},
\qquad
s\in\{-p,\ldots,q\},
\end{equation}
with the convention that the right-hand side is zero when $N+s<0$. Hence
\(\sum_{s=-p}^{q}p^{\mathrm P}_{N,N+s}=1\) for
\(N\in\Nzero\). In particular, the case $(p,q)=(3,2)$ gives a
six-diagonal chain with possible one-step moves
\(N\mapsto N-3,N-2,N-1,N,N+1,N+2\), subject only to the boundary
restrictions at $N=0,1,2$.

This wide-jump Markov chain exists throughout
\(\mathcal S^{\mathrm P,+}_{p,q}\).  In the part of that region lying
outside \(\mathcal R^{\mathrm P,+}_{p,q}\), the factorization into the
individual Christoffel bidiagonal factors is not positive.  On the PBF
region, the stochastic bidiagonal factorization gives the following more
refined description of the same transition. Write the factors in the order
\(F_1=\widehat L_1^{\mathrm P},\ldots,F_p=\widehat L_p^{\mathrm P}\) and
\(F_{p+1}=\widehat U_q^{\mathrm P},\ldots,F_{p+q}=\widehat U_1^{\mathrm P}\), so that
\(P^{\mathrm P}=F_1\cdots F_{p+q}\).
For a lower factor $\widehat L_k^{\mathrm P}$ put
\(\delta_k(N):=(\widehat L_k^{\mathrm P})_{N,N-1}\) and
\(1-\delta_k(N)=(\widehat L_k^{\mathrm P})_{N,N}\),
with $\delta_k(0)=0$. For an upper factor $\widehat U_j^{\mathrm P}$ put
\(\gamma_j(N):=(\widehat U_j^{\mathrm P})_{N,N+1}\) and
\(1-\gamma_j(N)=(\widehat U_j^{\mathrm P})_{N,N}\).
In the ordered cyclic PBF region \eqref{eq:MP-PBF-region},
\(0<\delta_k(N)<1\) for \(N\ge1\) and
\(0<\gamma_j(N)<1\) for \(N\ge0\),
with the boundary convention $\delta_k(0)=0$. Thus one step of
$X^{\mathrm P}$ is the result of \(p+q\) nearest-neighbor transitions:
first \(p\) possible downward moves, followed by \(q\) possible upward
moves. Explicitly, if
\(N=N_0,N_1,\ldots,N_{p+q}=M\)
is a path such that, for $r\in\{1,\ldots,p\}$, $N_r-N_{r-1}\in\{0,-1\}$, and, for $r\in\{p+1,\ldots,p+q\}$, $N_r-N_{r-1}\in\{0,+1\}$, then
\begin{equation}\label{eq:MP-path-probabilities}
(P^{\mathrm P})_{N,M}
=
\sum_{N_1,\ldots,N_{p+q-1}}
\prod_{r=1}^{p+q}(F_r)_{N_{r-1},N_r},
\end{equation}
where the sum is restricted to paths satisfying these conditions.
Formula~\eqref{eq:MP-path-probabilities} is the probabilistic form of the
stochastic bidiagonal factorization: the wide jump \(N\to M\) is the net
displacement after the successive birth/death-or-stay transitions.

\subsection{Finite urns and a rational (3,2) worked example}
\label{subsec:MP-finite-urns}

The preceding elementary transitions are now realized with finite urns.
At stage \(r\), the current intermediate state \(N\) selects the
urn \(\mathscr U_{r,N}\); the color drawn determines whether the state
moves by one in the direction of that factor or remains fixed.  The updated
state then selects the urn used at stage \(r+1\).  Thus the model is a
state-dependent collection of binary urns used in a prescribed order, not
a single urn whose number of blue balls is the state.  The intermediate
state is updated after every draw.

The only quantity in
\eqref{eq:MP-stochastic-lower-explicit}--
\eqref{eq:MP-stochastic-upper-explicit} that might appear implicit is the
harmonic vector. Formula~\eqref{eq:MP-h-explicit-alpha-beta} has already
computed it directly from the explicit type-II coefficients, so no spectral
eigenvector remains to be determined. In the ordered cyclic PBF region it is
positive.

For brevity set \(m=p+q\) and write the ordered factors as
\(M_1=L_1^{\mathrm P},\ldots,M_p=L_p^{\mathrm P}\) and
\(M_{p+1}=U_q^{\mathrm P},\ldots,M_m=U_1^{\mathrm P}\).
The intermediate vectors used above can equivalently be written without an
implicit recurrence as
\begin{equation}
\label{eq:MP-intermediate-h-explicit}
h^{[r]}=M_{r+1}\cdots M_mh^{\mathrm P},
\qquad 0\le r<m,
\qquad
h^{[m]}=h^{\mathrm P}.
\end{equation}
Since every entry of every \(M_r\) is given by
\eqref{eq:MP-lower-Pochhammer-simplified} or
\eqref{eq:MP-upper-Pochhammer-simplified}, equations
\eqref{eq:MP-raw-B-coefficients}--\eqref{eq:MP-intermediate-h-explicit}
express all the urn probabilities below solely in terms of
\(\boldsymbol\alpha,\boldsymbol\beta,N,p,q\).

\begin{proposition}[Mixed Pi\~neiro finite-urn model]
\label{pro:MP-finite-urn-model}
	Assume
	\((\boldsymbol\alpha,\boldsymbol\beta)
	\in\mathcal R^{\mathrm P,+}_{p,q}\cap\mathbb Q^{p+q}\).
	At an intermediate state \(N\), the \(a\)-th lower experiment,
	\(a\in\{1,\ldots,p\}\), uses the two positive rational weights
	\begin{align}
	\mathcal B^-_{a,N}
	&=
	\ell_{a,N-1}^{\mathrm P}h_{N-1}^{[a]},
	&
	\mathcal R^-_{a,N}
	&=
	h_N^{[a]},
	\label{eq:MP-lower-urn-weights}
	\end{align}
	with \(\mathcal B^-_{a,0}=0\). A blue draw sends \(N\) to \(N-1\)
	and a red draw leaves it at \(N\). Hence
	\begin{equation}
	\label{eq:MP-lower-urn-probability-alpha-beta}
	\delta_a(N)
	=
	\frac{\mathcal B^-_{a,N}}
	{\mathcal B^-_{a,N}+\mathcal R^-_{a,N}}
	=
	\frac{
	\ell_{a,N-1}^{\mathrm P}h_{N-1}^{[a]}
	}{
	h_N^{[a]}+
	\ell_{a,N-1}^{\mathrm P}h_{N-1}^{[a]}
	}.
	\end{equation}

	For \(b\in\{1,\ldots,q\}\), let
	\(r(b)=p+q-b+1\), the position of \(U_b^{\mathrm P}\) in the ordered
	product. The corresponding upper experiment uses
	\begin{align}
	\mathcal B^+_{b,N}
	&=
	h_{N+1}^{[r(b)]},
	&
	\mathcal R^+_{b,N}
	&=
	u_{b,N}^{\mathrm P}h_N^{[r(b)]}.
	\label{eq:MP-upper-urn-weights}
	\end{align}
	A blue draw sends \(N\) to \(N+1\) and a red draw leaves it at \(N\),
	with
	\begin{equation}
	\label{eq:MP-upper-urn-probability-alpha-beta}
	\gamma_b(N)
	=
	\frac{\mathcal B^+_{b,N}}
	{\mathcal B^+_{b,N}+\mathcal R^+_{b,N}}
	=
	\frac{
	h_{N+1}^{[r(b)]}
	}{
	h_{N+1}^{[r(b)]}
	+u_{b,N}^{\mathrm P}h_N^{[r(b)]}
	}.
	\end{equation}
	Write each probability in lowest terms,
	\(\delta_a(N)=A^-_{a,N}/D^-_{a,N}\) and
	\(\gamma_b(N)=A^+_{b,N}/D^+_{b,N}\).
	Then the lower urn contains \(A^-_{a,N}\) blue balls and
	\(D^-_{a,N}-A^-_{a,N}\) red balls, whereas the upper urn contains
	\(A^+_{b,N}\) blue balls and \(D^+_{b,N}-A^+_{b,N}\) red balls.
	Executing the \(p\) lower urns in the order
	\(1,\ldots,p\), followed by the \(q\) upper urns in the order
	\(q,\ldots,1\), gives exactly one transition of \(P^{\mathrm P}\).
\end{proposition}

\begin{proof}
	For rational parameters, every finite Pochhammer product in
	\eqref{eq:MP-lower-Pochhammer-simplified},
	\eqref{eq:MP-upper-Pochhammer-simplified}, and
	\eqref{eq:MP-raw-B-coefficients} is rational. Therefore
	\(h_N^{\mathrm P}\), every component of \(h^{[r]}\), and all four
	weights in
	\eqref{eq:MP-lower-urn-weights}--\eqref{eq:MP-upper-urn-weights}
	are rational. They are positive in
	\(\mathcal R^{\mathrm P,+}_{p,q}\), except for the forced lower
	boundary weight \(\mathcal B^-_{a,0}=0\).

	The lower denominator in
	\eqref{eq:MP-lower-urn-probability-alpha-beta} is
	\(h_N^{[a-1]}\), while the upper denominator in
	\eqref{eq:MP-upper-urn-probability-alpha-beta} is
	\(h_N^{[r(b)-1]}\). Hence these are exactly the entries of the
	stochastic factors in
	\eqref{eq:MP-stochastic-lower-explicit} and
	\eqref{eq:MP-stochastic-upper-explicit}. Reducing the rational
	probabilities gives finite integer ball counts, and the ordered product
	of the corresponding binary transitions is \(P^{\mathrm P}\) by
	\eqref{eq:MP-SBF}.
\end{proof}

The discrete kernel, its Poissonization, its QBD representation, and the
positive-measure mixed representation with
\(\dd\psi_{j,i}^{\mathrm P}(x)=x^{\alpha_i+\beta_j}\,\dd x\)
are valid throughout the enlarged region.  Outside the PBF region, the
kernel no longer has the ordered decomposition into positive binary
bidiagonal transitions, and the PBF Favard theorem does not apply.

One convenient way to select such parameters is to choose \(K\in\mathbb N\)
and integers \(A_1,\ldots,A_p,B_1,\ldots,B_q\) satisfying
\begin{align*}
-K&<A_1<\cdots<A_p<A_1+K,\\
-K&<B_1<\cdots<B_q<B_1+K,
&A_1+B_1&>-K.
\end{align*}
and set \(\alpha_i=A_i/K\) and \(\beta_j=B_j/K\).
This gives a dense arithmetic subfamily of
\(\mathcal R^{\mathrm P,+}_{p,q}\). The reduced fractions in
Proposition~\ref{pro:MP-finite-urn-model} give the smallest urn at each state
and elementary transition; simultaneous integer multiples give equivalent
urns.

\begin{example}[A rational \((p,q)=(3,2)\) Pi\~neiro chain and its five urns]
\label{ex:MP-rational-urn-cycle}
Take \((p,q)=(3,2)\),
\(\boldsymbol\alpha=
\begin{bNiceMatrix}-\frac13&0&\frac13\end{bNiceMatrix}\), and
\(\boldsymbol\beta=\begin{bNiceMatrix}0&\frac13\end{bNiceMatrix}\).
The following inequalities show that the parameters belong to the ordered
PBF region:
\[
-1<-\frac13<0<\frac13<\frac23,
\qquad
-1<0<\frac13<1,
\qquad
-\frac13+0>-1.
\]
Consequently, the ordered cyclic PBF theorem gives the positive bidiagonal
factorization
\(T^{\mathrm P}=L_1^{\mathrm P}L_2^{\mathrm P}L_3^{\mathrm P}
U_2^{\mathrm P}U_1^{\mathrm P}\).
The first three mixed forms, obtained directly from
\eqref{eq:MP-explicit-A-coefficients}--
\eqref{eq:MP-explicit-B-components}, are
\begin{align*}
\boldsymbol A_0^{\mathrm P}(x)
&=\begin{bNiceMatrix}[small]\frac23\\0\\0\end{bNiceMatrix},
&
\boldsymbol B_0^{\mathrm P}(x)
&=\begin{bNiceMatrix}[small]1&0\end{bNiceMatrix},
\\
\boldsymbol A_1^{\mathrm P}(x)
&=\begin{bNiceMatrix}[small]-8\\12\\0\end{bNiceMatrix},
&
\boldsymbol B_1^{\mathrm P}(x)
&=\begin{bNiceMatrix}[small]-\frac23&1\end{bNiceMatrix},
\\
\boldsymbol A_2^{\mathrm P}(x)
&=\begin{bNiceMatrix}[small]35\\-140\\\frac{350}{3}\end{bNiceMatrix},
&
\boldsymbol B_2^{\mathrm P}(x)
&=\begin{bNiceMatrix}[small]x+\frac25&-\frac65\end{bNiceMatrix},
\end{align*}
They satisfy the biorthogonality relation
\eqref{eq:MP-Cauchy-direct-biorthogonality}, so the example starts from
explicit mixed forms.  The Doob-normalized factors give the following
ordered product of five stochastic transitions:
\(P^{\mathrm P}=\widehat L_1^{\mathrm P}\widehat L_2^{\mathrm P}
\widehat L_3^{\mathrm P}\widehat U_2^{\mathrm P}\widehat U_1^{\mathrm P}\).
The first three downward probabilities in the lower factors are
\begin{align*}
(\delta_1(N))_{N=1}^{3}
&=\begin{bNiceMatrix}\frac14&\frac27&\frac13\end{bNiceMatrix},\\
(\delta_2(N))_{N=1}^{3}
&=\begin{bNiceMatrix}\frac15&\frac14&\frac3{10}\end{bNiceMatrix},\\
(\delta_3(N))_{N=1}^{3}
&=\begin{bNiceMatrix}\frac16&\frac29&\frac3{11}\end{bNiceMatrix},
\end{align*}
and the first three upward probabilities are
\begin{align*}
(\gamma_2(N))_{N=0}^{2}
&=\begin{bNiceMatrix}\frac25&\frac37&\frac12\end{bNiceMatrix},&
(\gamma_1(N))_{N=0}^{2}
&=\begin{bNiceMatrix}\frac13&\frac12&\frac12\end{bNiceMatrix}.
\end{align*}
Each entry is implemented by a finite two-color urn.  For example, the
path through the five factors
\[
2\xrightarrow{\widehat L_1}1
\xrightarrow{\widehat L_2}1
\xrightarrow{\widehat L_3}0
\xrightarrow{\widehat U_2}1
\xrightarrow{\widehat U_1}2
\]
uses, successively, urns with blue/red contents
\(\begin{bNiceMatrix}2&5\end{bNiceMatrix}\),
\(\begin{bNiceMatrix}1&4\end{bNiceMatrix}\),
\(\begin{bNiceMatrix}1&5\end{bNiceMatrix}\), \qquad
\(\begin{bNiceMatrix}2&3\end{bNiceMatrix}\), and
\(\begin{bNiceMatrix}1&1\end{bNiceMatrix}\); the color sequence is
blue--red--blue--blue--blue. Its probability is
\((2/7)(4/5)(1/6)(2/5)(1/2)=4/525\).

Multiplication of the five stochastic factors gives the following initial
block:
\begin{equation}
\label{eq:MP-rational-example-P}
\left.P^{\mathrm P}\right|_{\{0,\ldots,4\}\times\{0,\ldots,6\}}
=
\begin{bNiceMatrix}[cell-space-limits=2pt]
\frac25&\frac25&\frac15&0&0&0&0\\
\frac15&\frac{12}{35}&\frac7{20}&\frac3{28}&0&0&0\\
\frac1{20}&\frac{19}{105}&\frac{43}{120}&\frac{183}{616}&\frac5{44}&0&0\\
\frac1{180}&\frac{17}{315}&\frac{91}{440}&\frac{1051}{3080}&
\frac{861}{2860}&\frac{196}{2145}&0\\
0&\frac2{315}&\frac{389}{6435}&\frac{144}{715}&
\frac{5282}{15015}&\frac{201}{715}&\frac9{91}
\end{bNiceMatrix}.
\end{equation}
The rows sum to one, and the five urns generate the six net displacements
\(-3,-2,-1,0,+1,+2\).

This example also displays the distinction in
Proposition~\ref{prop:MP-boundedness-criterion}.  Its cyclic beta spacings
are \(1/3\) and \(2/3\); hence the upper-factor exponents are
\(-1/3\) and \(1/3\).  In particular,
\[
(T^{\mathrm P})_{2m,2m+1}=\Theta(m^{1/3}),
\qquad m\to\infty,
\]
so the Pi\~neiro recurrence matrix in its original step-line normalization
is unbounded.  The matrix \(P^{\mathrm P}\) displayed in
\eqref{eq:MP-rational-example-P} is nevertheless a banded row-stochastic
matrix and is bounded on every \(\ell^s\).  Thus the Doob conjugation in
this concrete example is not a bounded similarity.
\end{example}

\subsection{Bounded Poissonization and long-time behavior}
\label{subsec:MP-continuous-chain}
For any rate $\nu>0$, define
\(Q^{\mathrm P}:=\nu(P^{\mathrm P}-I)\). Then $Q^{\mathrm P}$ is a
conservative bounded $Q$-matrix with two-sided bandwidth $(p,q)$. Its
off-diagonal rates are \(q^{\mathrm P}_{N,M}=\nu p^{\mathrm P}_{N,M}\)
for \(M\neq N\), while
\(q^{\mathrm P}_{N,N}=-\nu\sum_{M\neq N}p^{\mathrm P}_{N,M}
=-\nu(1-p^{\mathrm P}_{N,N})\).
For example, the first row of the generator associated with
Example~\ref{ex:MP-rational-urn-cycle} is
\[
\begin{bNiceMatrix}[margin=5pt]q^{\mathrm P}_{0,0}&q^{\mathrm P}_{0,1}&q^{\mathrm P}_{0,2}&\Cdots&\end{bNiceMatrix}
=\nu
\begin{bNiceMatrix}[margin=5pt]-\frac35&\frac25&\frac15&0&0&\Cdots&\end{bNiceMatrix}.
\]
All remaining rates follow from the explicit factor formulas, with
\(-\nu(1-p^{\mathrm P}_{N,N})\) on the diagonal.
Equivalently, if $(N_\nu(t))_{t\ge0}$ is a Poisson process of rate $\nu$,
independent of $X^{\mathrm P}$, then
\(Y^{\mathrm P}_t:=X^{\mathrm P}_{N_\nu(t)}\) has generator
$Q^{\mathrm P}$. Its transition semigroup is
\(\mathrm e^{tQ^{\mathrm P}}
=
\mathrm e^{-\nu t}\sum_{n=0}^{\infty}
\frac{(\nu t)^n}{n!}(P^{\mathrm P})^n\).
Using \eqref{eq:MP-Pn-spectral-normalized}, this becomes
\begin{equation}
\label{eq:MP-ct-spectral}
(\mathrm e^{tQ^{\mathrm P}})_{N,M}
=
\sum_{a=1}^{p}\sum_{b=1}^{q}
\int_0^1
\widetilde B_N^{(b),\mathrm P}(x)\,
\e^{-\nu t(1-x)}\,
\widetilde A_M^{(a),\mathrm P}(x)
x^{\alpha_a+\beta_b}\,\dd x.
\end{equation}
Similarly, for \(\lambda>0\), the resolvent of the continuous-time generator is
\begin{equation}
\label{eq:MP-resolvent-spectral}
(\lambda I-Q^{\mathrm P})^{-1}_{N,M}
=
\sum_{a=1}^{p}\sum_{b=1}^{q}
\int_0^1
\frac{
\widetilde B_N^{(b),\mathrm P}(x)
\widetilde A_M^{(a),\mathrm P}(x)
}{\lambda+\nu(1-x)}
x^{\alpha_a+\beta_b}\,\dd x.
\end{equation}

The generalized hypergeometric notation fixed in
\eqref{eq:generalized-hypergeometric-definition} is used. Thus
\(\pFq{1}{1}{a}{c}{z}\) is Kummer's confluent hypergeometric function,
whereas \(\pFq{2}{1}{a,b}{c}{z}\) is the Gauss hypergeometric function.

\begin{corollary}[Evaluation of the continuous-time spectral integrals]
\label{coro:MP-continuous-spectral-explicit}
Put \(\rho_{a,b,r,s}:=\alpha_a+\beta_b+r+s+1>0\). For \(t\ge0\),
\begin{equation}
\label{eq:MP-ct-spectral-explicit}
(\mathrm e^{tQ^{\mathrm P}})_{N,M}
=
\frac{h_M^{\mathrm P}}
{h_N^{\mathrm P}\widehat\kappa_N}
\sum_{a=1}^{p}\sum_{b=1}^{q}
\sum_{r=0}^{\widetilde m_{N,b}-1}
\sum_{s=0}^{n_{M,a}-1}
\frac{c_{N,b,r}\,a_{M,a,s}}{\rho_{a,b,r,s}}
\mathrm e^{-\nu t}\,
\pFq{1}{1}
{\rho_{a,b,r,s}}
{\rho_{a,b,r,s}+1}
{\nu t}.
\end{equation}
For \(\lambda>0\),
\begingroup
\small
\begin{equation}
\label{eq:MP-resolvent-spectral-explicit}
(\lambda I-Q^{\mathrm P})^{-1}_{N,M}
=
\frac{h_M^{\mathrm P}}
{h_N^{\mathrm P}\widehat\kappa_N}
\sum_{a=1}^{p}\sum_{b=1}^{q}
\sum_{r=0}^{\widetilde m_{N,b}-1}
\sum_{s=0}^{n_{M,a}-1}
\frac{c_{N,b,r}\,a_{M,a,s}}
{(\lambda+\nu)\rho_{a,b,r,s}}
\pFq{2}{1}
{1,\rho_{a,b,r,s}}
{\rho_{a,b,r,s}+1}
{\dfrac{\nu}{\lambda+\nu}}.
\end{equation}
\endgroup
\end{corollary}

\begin{proof}
After expanding the two polynomial components in
\eqref{eq:MP-ct-spectral} and
\eqref{eq:MP-resolvent-spectral}, only the scalar integrals
\begin{align*}
\int_0^1x^{\rho-1}\e^{-\nu t(1-x)}\,\dd x
&=
\frac{\e^{-\nu t}}{\rho}\,
\pFq{1}{1}{\rho}{\rho+1}{\nu t},
\\
\int_0^1\frac{x^{\rho-1}}{\lambda+\nu(1-x)}\,\dd x
&=
\frac{1}{(\lambda+\nu)\rho}\,
\pFq{2}{1}{1,\rho}{\rho+1}{\dfrac{\nu}{\lambda+\nu}}
\end{align*}
remain. Substitution gives
\eqref{eq:MP-ct-spectral-explicit} and
\eqref{eq:MP-resolvent-spectral-explicit}.
\end{proof}

Thus the model gives a concrete continuous-time Markov chain with wide
two-sided jumps. Its transition and resolvent kernels are explicitly represented by
the finite Pi\~neiro polynomials and the elementary beta densities.
Unlike the discrete-time entries in
\eqref{eq:MP-Pn-spectral-explicit}, the continuous-time and resolvent
entries are not generally rational functions of the parameters. They are,
however, finite hypergeometric sums with coefficients rational in
\(\boldsymbol\alpha\) and \(\boldsymbol\beta\). Equivalently, the Poisson expansion of
the semigroup has the rational coefficients given explicitly by
\eqref{eq:MP-Pn-spectral-explicit}.

\begin{proposition}[Null recurrence, invariant measure, and local limits]
\label{pro:MP-null-recurrence-local-limits}
Let
\((\boldsymbol\alpha,\boldsymbol\beta)\in
\mathcal S^{\mathrm P,+}_{p,q}\), so that \(P^{\mathrm P}\) is the
stochastic kernel defined in
Subsection~\ref{subsec:MP-stochastic-normalization}, and let
\(Q^{\mathrm P}=\nu(P^{\mathrm P}-I)\), with \(\nu>0\).  Put
\[
\vartheta\coloneq\alpha_1+\beta_1+1>0,
\qquad
\pi_M^{\mathrm P}
\coloneq
h_M^{\mathrm P}A_M^{\mathrm P}(1),
\qquad M\in\Nzero.
\]
Then the discrete-time chain is irreducible, aperiodic, and null
recurrent.  Its continuous-time Poissonization is irreducible and null
recurrent.  More precisely,
\begin{align}
\label{eq:MP-origin-return-exact}
\bigl((P^{\mathrm P})^k\bigr)_{0,0}
&=
\frac{\vartheta}{k+\vartheta},
&&k\in\Nzero,
\\
\label{eq:MP-origin-return-continuous-exact}
\bigl(\e^{tQ^{\mathrm P}}\bigr)_{0,0}
&=
\e^{-\nu t}\,
\pFq{1}{1}{\vartheta}{\vartheta+1}{\nu t},
&&t\ge0.
\end{align}
The positive vector
\(\pi^{\mathrm P}=(\pi_M^{\mathrm P})_{M\ge0}\) is an invariant
measure:
\begin{equation}
\label{eq:MP-invariant-measure}
\pi^{\mathrm P}P^{\mathrm P}=\pi^{\mathrm P},
\qquad
\sum_{M=0}^{\infty}\pi_M^{\mathrm P}=+\infty.
\end{equation}
For every fixed \(N,M\in\Nzero\),
\begin{align}
\label{eq:MP-discrete-local-limit}
\lim_{k\to\infty}
k\bigl((P^{\mathrm P})^k\bigr)_{N,M}
&=
\pi_M^{\mathrm P},
\\
\label{eq:MP-continuous-local-limit}
\lim_{t\to\infty}
t\bigl(\e^{tQ^{\mathrm P}}\bigr)_{N,M}
&=
\frac{\pi_M^{\mathrm P}}{\nu}.
\end{align}
Thus the leading local asymptotic is independent of the initial state.
\end{proposition}

\begin{proof}
At \(N=0\), the step-line indices give
\(A_0^{(1),\mathrm P}(x)=\vartheta\),
\(B_0^{(1),\mathrm P}(x)=1\), and all the other initial components
vanish.  Also \(h_0^{\mathrm P}=1\).  Formula
\eqref{eq:MP-Pn-spectral} therefore gives
\[
\bigl((P^{\mathrm P})^k\bigr)_{0,0}
=
\vartheta\int_0^1x^{k+\vartheta-1}\,\dd x
=
\frac{\vartheta}{k+\vartheta},
\]
which proves \eqref{eq:MP-origin-return-exact}.  Poissonization, or
equivalently \eqref{eq:MP-ct-spectral}, gives
\eqref{eq:MP-origin-return-continuous-exact}.

\(P^{\mathrm P}\) is strictly positive throughout its structural band on
\(\mathcal S^{\mathrm P,+}_{p,q}\), so the chain is irreducible.
Moreover, \((P^{\mathrm P})_{0,0}>0\), and hence irreducibility implies
aperiodicity.  Equation \eqref{eq:MP-origin-return-exact} gives
\[
\sum_{k=0}^{\infty}
\bigl((P^{\mathrm P})^k\bigr)_{0,0}
=+\infty,
\qquad
\lim_{k\to\infty}
\bigl((P^{\mathrm P})^k\bigr)_{0,0}=0.
\]
The first relation proves recurrence.  The second excludes positive
recurrence for this irreducible aperiodic chain.  Uniformization preserves
the classification, so the continuous-time chain is also null recurrent.

Next, the fact that every Pi\~neiro spectral density has endpoint exponent
zero is used:
\(x^{\alpha_a+\beta_b}\longrightarrow1\)
as \(x\uparrow1\).
Applying the endpoint argument of
Theorem~\ref{thm:bounded-endpoint-asymptotics} directly to the
Doob-normalized formula \eqref{eq:MP-Pn-spectral} gives
\[
\lim_{k\to\infty}
k\bigl((P^{\mathrm P})^k\bigr)_{N,M}
=
\frac{h_M^{\mathrm P}}{h_N^{\mathrm P}}
\sum_{b=1}^{q}B_N^{(b),\mathrm P}(1)
\sum_{a=1}^{p}A_M^{(a),\mathrm P}(1).
\]
By the definitions of the scalar forms,
\(\sum_bB_N^{(b),\mathrm P}(1)=h_N^{\mathrm P}\) and
\(\sum_aA_M^{(a),\mathrm P}(1)=A_M^{\mathrm P}(1)\).
This proves \eqref{eq:MP-discrete-local-limit}.  The continuous-time
limit follows in the same way from \eqref{eq:MP-ct-spectral}, or from
the second limit in
Theorem~\ref{thm:bounded-endpoint-asymptotics}.

The left multiplication recurrence in
Proposition~\ref{pro:MP-Cauchy-step-line-construction}, evaluated at
\(x=1\), gives
\((A_M^{\mathrm P}(1))_{M\ge0}^{\mathsf T}T^{\mathrm P}
=(A_M^{\mathrm P}(1))_{M\ge0}^{\mathsf T}\).
Conjugating by \(H^{\mathrm P}\) shows that the vector in
\eqref{eq:MP-invariant-measure} is invariant.  Its entries are
nonnegative by \eqref{eq:MP-discrete-local-limit};
\(\pi_0^{\mathrm P}=\vartheta>0\), and irreducibility then makes
every entry strictly positive.  If its total mass were finite, it could
be normalized to a stationary distribution, contradicting null
recurrence.  This completes the proof.
\end{proof}

\begin{corollary}[Critical Green, resolvent, and return asymptotics]
\label{coro:MP-critical-return-asymptotics}
For fixed \(N,M\in\Nzero\),
\begin{align}
\label{eq:MP-critical-Green}
\sum_{k=0}^{\infty}
z^k\bigl((P^{\mathrm P})^k\bigr)_{N,M}
&\sim
\pi_M^{\mathrm P}\log\frac{1}{1-z},
&&z\uparrow1,
\\
\label{eq:MP-critical-resolvent}
\bigl((\lambda I-Q^{\mathrm P})^{-1}\bigr)_{N,M}
&\sim
\frac{\pi_M^{\mathrm P}}{\nu}
\log\frac{\nu}{\lambda},
&&\lambda\downarrow0.
\end{align}
If
\(\tau_M^+\coloneq
\inf\{k\ge1:X_k^{\mathrm P}=M\}\), then
\begin{equation}
\label{eq:MP-first-return-tail}
\Pr_M\{\tau_M^+>k\}
\sim
\frac{1}{\pi_M^{\mathrm P}\log k}.
\end{equation}
In particular,
\(\operatorname E_M[\tau_M^+]=+\infty\).
\end{corollary}

\begin{proof}
Equations \eqref{eq:MP-critical-Green} and
\eqref{eq:MP-critical-resolvent} are the Abelian and Laplace-transform
forms of
\eqref{eq:MP-discrete-local-limit} and
\eqref{eq:MP-continuous-local-limit}.  The renewal argument in
Corollary~\ref{coro:bounded-endpoint-classification}, with
\(c_{M,M}=\pi_M^{\mathrm P}\), gives
\eqref{eq:MP-first-return-tail}.
\end{proof}

\begin{corollary}[Evaluation of the continued fraction and the first-return law]
\label{coro:MP-factor-resolved-return-CF}
Let
\[
\vartheta=\alpha_1+\beta_1+1>0,
\qquad
g_0^{\mathrm P}(z)
\coloneq
\sum_{k=0}^{\infty}
\bigl((P^{\mathrm P})^k\bigr)_{0,0}z^k,
\]
and let
\[
f_0^{\mathrm P}(z)
\coloneq
\operatorname E_0\!\left[
z^{\tau_0^+}\ind_{\{\tau_0^+<\infty\}}
\right].
\]
In the ordered PBF region, let
\(\Phi_0^{\mathrm P}(\tau)\) be the boundary matrix continued fraction
constructed from the \(p+q\) stochastic Pi\~neiro factors by
\eqref{eq:factor-tail-definition}--\eqref{eq:factor-tail-recursion}.  Then
\begin{align}
\label{eq:MP-factor-CF-hypergeometric-evaluation}
e_1^{\mathsf T}\Phi_0^{\mathrm P}(\tau)e_1
&=
g_0^{\mathrm P}(\tau^{p+q})
=
\pFq{2}{1}
{1,\vartheta}
{\vartheta+1}
{\tau^{p+q}},
\\
\label{eq:MP-first-return-PGF-explicit}
f_0^{\mathrm P}(z)
&=
1-
\left[
\pFq{2}{1}
{1,\vartheta}
{\vartheta+1}
{z}
\right]^{-1}.
\end{align}
Thus the continued fraction constructed from the individual bidiagonal
factors has an explicit hypergeometric value and determines the complete
first-return distribution of the original scalar Pi\~neiro chain.

For the rational \((p,q)=(3,2)\) example,
\(\vartheta=2/3\), and
\begin{align}
\label{eq:MP-rational-origin-Green-expansion}
g_0^{\mathrm P}(z)
&=
1+\frac25z+\frac14z^2+\frac2{11}z^3+O(z^4),
\\
\label{eq:MP-rational-first-return-expansion}
f_0^{\mathrm P}(z)
&=
\frac25z+\frac9{100}z^2+\frac{63}{1375}z^3+O(z^4).
\end{align}
In particular,
\[
\Pr_0\{\tau_0^+=1\}=\frac25,
\qquad
\Pr_0\{\tau_0^+=2\}=\frac9{100},
\qquad
\Pr_0\{\tau_0^+=3\}=\frac{63}{1375}.
\]
Moreover,
\begin{equation}
\label{eq:MP-first-return-CF-boundary}
1-f_0^{\mathrm P}(z)
\sim
\frac{1}{\vartheta\log(1/(1-z))},
\qquad z\uparrow1.
\end{equation}
Consequently the return probability is one, whereas
\(\operatorname E_0[\tau_0^+]=+\infty\).
\end{corollary}

\begin{proof}
By \eqref{eq:MP-origin-return-exact},
\[
g_0^{\mathrm P}(z)
=
\sum_{k=0}^{\infty}\frac{\vartheta}{k+\vartheta}z^k
=
\pFq{2}{1}{1,\vartheta}{\vartheta+1}{z}.
\]
The boundary evaluation
\eqref{eq:factor-resolved-return-Green} identifies this Green function
with the phase-\(1\) entry of the continued fraction after the Pi\~neiro
factors are substituted.  Equation
\eqref{eq:factor-resolved-first-return-law} then gives
\(f_0^{\mathrm P}=1-1/g_0^{\mathrm P}\).  Substituting
\(\vartheta=2/3\) and expanding proves
\eqref{eq:MP-rational-origin-Green-expansion} and
\eqref{eq:MP-rational-first-return-expansion}.  Finally,
\eqref{eq:MP-critical-Green} at \(N=M=0\), together with
\(\pi_0^{\mathrm P}=\vartheta\), gives
\eqref{eq:MP-first-return-CF-boundary}.  The last assertions follow from
\eqref{eq:MP-first-return-tail}.
\end{proof}

For rational Pi\~neiro parameters, every finite continued-fraction cut is
a matrix of rational functions.  Evaluating the recursion backwards and
then applying \eqref{eq:factor-resolved-first-return-law} produces
successive exact return probabilities without forming powers of the
infinite transition matrix.  Equations
\eqref{eq:MP-factor-CF-hypergeometric-evaluation} and
\eqref{eq:MP-first-return-PGF-explicit} provide an independent closed form
against which those finite computations can be checked.  Pi\~neiro is used
here as an explicit illustration of the continued-fraction algorithm and
the return-law transformation \eqref{eq:factor-resolved-first-return-law}.

\subsection{The mixed Pi\~neiro chain as a QBD}
\label{subsec:MP-QBD}

Write each scalar state uniquely as \(N=rK+i\), where
\(r=\max\{p,q\}\),
\(K\in\Nzero\) and \(i\in\{0,\ldots,r-1\}\),
and relabel the mixed Pi\~neiro chain as
\begin{equation}
Z_n^{\mathrm P}
=\begin{bNiceMatrix}K_n^{\mathrm P}&I_n^{\mathrm P}\end{bNiceMatrix}
=
\begin{bNiceMatrix}[cell-space-limits=2pt]
\left\lfloor\frac{X_n^{\mathrm P}}{r}\right\rfloor&
X_n^{\mathrm P}-r\left\lfloor\frac{X_n^{\mathrm P}}{r}\right\rfloor
\end{bNiceMatrix}.
\end{equation}
For \(\varepsilon\in\{-1,0,1\}\), define the \(r\times r\) blocks
\begin{equation}
\left(\mathsf P_K^{\mathrm P,\varepsilon}\right)_{i,j}
\coloneq
p_{rK+i,r(K+\varepsilon)+j}^{\mathrm P},
\qquad 0\le i,j\le r-1,
\end{equation}
where entries with a negative level are zero. Proposition~\ref{pro:banded-kernel-QBD}
then gives the block-tridiagonal transition matrix
\begin{equation}
P^{\mathrm P}
=
\begin{bNiceMatrix}[cell-space-limits=4pt,margin=5pt]
\mathsf P_0^{\mathrm P,0} & \mathsf P_0^{\mathrm P,+} & 0_{r\times r} &
\Cdots[shorten-start=2pt,shorten-end=-20pt] & \phantom{0}  \\[-5pt]
\mathsf P_1^{\mathrm P,-} & \mathsf P_1^{\mathrm P,0} &
\mathsf P_1^{\mathrm P,+} &
\Ddots[shorten-start=-0pt,shorten-end=5pt] & \phantom{0} \\
0_{r\times r} & \mathsf P_2^{\mathrm P,-} &
\mathsf P_2^{\mathrm P,0} & \mathsf P_2^{\mathrm P,+} & \\
\Vdots[shorten-start=0pt,shorten-end=-15pt] &
\Ddots[shorten-start=0pt,shorten-end=-35pt] &
\Ddots[shorten-start=0pt,shorten-end=-32pt] &
\Ddots[shorten-start=0pt,shorten-end=-30pt] &
\Ddots[shorten-start=0pt,shorten-end=-0pt] &
\\
\phantom{0} & \phantom{\mathsf P_2^{\mathrm P,-}} &
\phantom{\mathsf P_2^{\mathrm P,0}} &
\phantom{\mathsf P_2^{\mathrm P,+}} & \phantom{0}
\end{bNiceMatrix}.
\end{equation}
The QBD is generally level-dependent. The level $K$ records the block of
$r$ consecutive scalar states and the phase $i$ records the exact state
inside that block; no information is discarded.

The urn construction also gives each block entry explicitly. If the process starts in
phase $i$ of level $K$, then
$\left(\mathsf P_K^{\mathrm P,\varepsilon}\right)_{i,j}$ is the sum of
the probabilities of all admissible $(p+q)$-stage urn paths whose net effect
is to end in phase $j$ of level $K+\varepsilon$. The block row is
therefore stochastic, that is,
\((\mathsf P_K^{\mathrm P,-}
+\mathsf P_K^{\mathrm P,0}
+\mathsf P_K^{\mathrm P,+})\one_r=\one_r\).

\begin{proposition}[Invariant vectors by level and limiting phase distribution]
\label{pro:MP-QBD-local-limit}
For \(L\in\Nzero\), define the invariant row vector associated with level
\(L\) as
\[
\boldsymbol\pi_L^{\mathrm P}
\coloneq
\begin{bNiceMatrix}
\pi_{rL}^{\mathrm P}&\Cdots&
\pi_{rL+r-1}^{\mathrm P}
\end{bNiceMatrix}.
\]
The discrete-time chain grouped by levels, and its continuous-time Poissonization,
are irreducible and null recurrent.  The invariant measure satisfies the
block balance equations
\begin{align}
\label{eq:MP-QBD-invariant-boundary}
\boldsymbol\pi_0^{\mathrm P}
&=
\boldsymbol\pi_0^{\mathrm P}\mathsf P_0^{\mathrm P,0}
+\boldsymbol\pi_1^{\mathrm P}\mathsf P_1^{\mathrm P,-},
\\
\label{eq:MP-QBD-invariant-interior}
\boldsymbol\pi_L^{\mathrm P}
&=
\boldsymbol\pi_{L-1}^{\mathrm P}\mathsf P_{L-1}^{\mathrm P,+}
+\boldsymbol\pi_L^{\mathrm P}\mathsf P_L^{\mathrm P,0}
+\boldsymbol\pi_{L+1}^{\mathrm P}\mathsf P_{L+1}^{\mathrm P,-},
\qquad L\ge1,
\end{align}
and \(\sum_{L\ge0}\boldsymbol\pi_L^{\mathrm P}\one_r=+\infty\).
For every pair of fixed levels \(K,L\),
\begin{align}
\label{eq:MP-QBD-discrete-local-limit}
\lim_{k\to\infty}
k\bigl((P^{\mathrm P})^k\bigr)_{K,L}
&=
\one_r\boldsymbol\pi_L^{\mathrm P},
\\
\label{eq:MP-QBD-continuous-local-limit}
\lim_{t\to\infty}
t\bigl(\e^{tQ^{\mathrm P}}\bigr)_{K,L}
&=
\frac{1}{\nu}\one_r\boldsymbol\pi_L^{\mathrm P}.
\end{align}
Both limits have rank one.  Thus the dependence on the initial phase
disappears at leading order, while the phase profile at the arrival level
is determined by the nonnormalizable invariant measure.

More explicitly, write
\(Z^{\mathrm P}(t)=[K^{\mathrm P}(t),I^{\mathrm P}(t)]\) for the
continuous-time QBD, and let either chain start from phase \(i\) of level
\(K\).
Then
\begin{align}
\label{eq:MP-QBD-level-local-limit}
\lim_{k\to\infty}
k\Pr_{K,i}\{K_k^{\mathrm P}=L\}
&=
\boldsymbol\pi_L^{\mathrm P}\one_r,
\\
\label{eq:MP-QBD-level-local-limit-continuous}
\lim_{t\to\infty}
t\Pr_{K,i}\{K^{\mathrm P}(t)=L\}
&=
\frac{1}{\nu}\boldsymbol\pi_L^{\mathrm P}\one_r.
\end{align}
For every \(j\in\{0,\ldots,r-1\}\), the conditional phase limits are
\begin{align}
\label{eq:MP-QBD-conditional-phase-discrete}
\lim_{k\to\infty}
\Pr_{K,i}\{I_k^{\mathrm P}=j\mid K_k^{\mathrm P}=L\}
&=
\frac{\pi_{rL+j}^{\mathrm P}}
{\boldsymbol\pi_L^{\mathrm P}\one_r},
\\
\label{eq:MP-QBD-conditional-phase-continuous}
\lim_{t\to\infty}
\Pr_{K,i}\{I^{\mathrm P}(t)=j\mid K^{\mathrm P}(t)=L\}
&=
\frac{\pi_{rL+j}^{\mathrm P}}
{\boldsymbol\pi_L^{\mathrm P}\one_r}.
\end{align}
Thus the limiting phase distribution at a fixed arrival level is the
normalized restriction of the invariant measure to that level, in both
discrete and continuous time.
\end{proposition}

\begin{proof}
The map between a scalar state and a level--phase pair is bijective, so it
preserves irreducibility and the recurrence classification.  Hence the QBD
is null recurrent in both time scales.  Grouping the scalar invariance
identity \(\pi^{\mathrm P}P^{\mathrm P}=\pi^{\mathrm P}\) by destination
levels gives \eqref{eq:MP-QBD-invariant-boundary} and
\eqref{eq:MP-QBD-invariant-interior}.  The divergence of the total block
mass is precisely the second assertion in
\eqref{eq:MP-invariant-measure}.

The \((i,j)\) entry of the first block in
\eqref{eq:MP-QBD-discrete-local-limit} is
\(\bigl((P^{\mathrm P})^k\bigr)_{rK+i,rL+j}\).
Equation \eqref{eq:MP-discrete-local-limit} shows that its product with
\(k\) converges to \(\pi_{rL+j}^{\mathrm P}\), independently of
\(i\).  Collecting the \(r^2\) scalar limits gives the first identity.
The second follows identically from
\eqref{eq:MP-continuous-local-limit}.  Multiplying the two block limits on
the right by \(\one_r\) gives
\eqref{eq:MP-QBD-level-local-limit} and
\eqref{eq:MP-QBD-level-local-limit-continuous}.  Finally, divide the local
limit for phase \(j\) by the corresponding level limit.  This gives
\eqref{eq:MP-QBD-conditional-phase-discrete} and
\eqref{eq:MP-QBD-conditional-phase-continuous}.
\end{proof}

For the rational example in
Example~\ref{ex:MP-rational-urn-cycle}, $r=3$, and the phases of level
$K$ are the scalar states $3K,3K+1,3K+2$. Multiplying the same five
explicit stochastic factors gives the first two block rows. At the boundary,
\begin{align}
\mathsf P_0^{\mathrm P,-}&=0,
&
\mathsf P_0^{\mathrm P,0}
&=
\begin{bNiceMatrix}[cell-space-limits=2pt]
\frac25&\frac25&\frac15\\
\frac15&\frac{12}{35}&\frac7{20}\\
\frac1{20}&\frac{19}{105}&\frac{43}{120}
\end{bNiceMatrix},
&
\mathsf P_0^{\mathrm P,+}
&=
\begin{bNiceMatrix}[cell-space-limits=2pt]
0&0&0\\
\frac3{28}&0&0\\
\frac{183}{616}&\frac5{44}&0
\end{bNiceMatrix}.
\end{align}
The next down-level, same-level, and up-level blocks are
\begin{align*}
\mathsf P_1^{\mathrm P,-}
&=
\begin{bNiceMatrix}[cell-space-limits=2pt]
\frac1{180}&\frac{17}{315}&\frac{91}{440}\\
0&\frac2{315}&\frac{389}{6435}\\
0&0&\frac{15}{2002}
\end{bNiceMatrix},
\\
\mathsf P_1^{\mathrm P,0}
&=
\begin{bNiceMatrix}[cell-space-limits=2pt]
\frac{1051}{3080}&\frac{861}{2860}&\frac{196}{2145}\\
\frac{144}{715}&\frac{5282}{15015}&\frac{201}{715}\\
\frac{125}{2002}&\frac{1719}{8008}&\frac{3329}{9724}
\end{bNiceMatrix},
\\
\mathsf P_1^{\mathrm P,+}
&=
\begin{bNiceMatrix}[cell-space-limits=2pt]
0&0&0\\
\frac9{91}&0&0\\
\frac{417}{1456}&\frac{165}{1904}&0
\end{bNiceMatrix}.
\end{align*}
In particular, the sums of the three matrices in each displayed block row
send \(\one_3\) to \(\one_3\). A step can
change the level by $-1$, $0$, or $+1$, while its phase records
which of the three scalar states is occupied.

The first two invariant level vectors are also completely explicit.
The formulas for \(A_M^{\mathrm P}(1)\) and \(h_M^{\mathrm P}\) give
\(\boldsymbol\pi_0^{\mathrm P}
=
\begin{bNiceMatrix}\frac23&\frac43&\frac73\end{bNiceMatrix}\),
\(\boldsymbol\pi_1^{\mathrm P}
=
\begin{bNiceMatrix}3&4&\frac{14}{3}\end{bNiceMatrix}\).
Direct multiplication verifies the boundary balance equation
\(\boldsymbol\pi_0^{\mathrm P}
=
\boldsymbol\pi_0^{\mathrm P}\mathsf P_0^{\mathrm P,0}
+\boldsymbol\pi_1^{\mathrm P}\mathsf P_1^{\mathrm P,-}\).
Consequently, conditioned on being in level (0) or (1) at a large
time, the limiting phase distributions are, respectively,
\(\begin{bNiceMatrix}\frac2{13}&\frac4{13}&\frac7{13}\end{bNiceMatrix}\),
\(\begin{bNiceMatrix}\frac9{35}&\frac{12}{35}&\frac{14}{35}\end{bNiceMatrix}\).
This makes the phase profile in
\eqref{eq:MP-QBD-conditional-phase-discrete}--
\eqref{eq:MP-QBD-conditional-phase-continuous} concrete for the rational
model.

\begin{proposition}[A non-block-symmetrizable rational QBD]
\label{pro:MP-QBD-not-symmetrizable}
	The rational mixed Pi\~neiro QBD of
	Example~\ref{ex:MP-rational-urn-cycle} cannot be transformed into a
	symmetric block Jacobi matrix by an invertible block diagonal
	similarity. The same conclusion holds for its Poissonized
	continuous-time generator.
\end{proposition}

\begin{proof}
	The displayed boundary blocks give
	\(\operatorname{rank}\mathsf P_0^{\mathrm P,+}=2\) and
	\(\operatorname{rank}\mathsf P_1^{\mathrm P,-}=3\).
	Indeed, the first matrix has a nonzero lower triangular \(2\times2\)
	submatrix in its last two rows and first two columns, while the second
	matrix is an invertible upper triangular \(3\times3\) matrix. If
	\(\mathcal R=\operatorname{diag}(R_0,R_1,\ldots)\) were invertible and
	\(\mathcal R P^{\mathrm P}\mathcal R^{-1}\) were symmetric,
	its \((0,1)\) and \((1,0)\) blocks would be transposes. Their ranks would
	therefore agree, whereas multiplication by \(R_0\) and \(R_1\) cannot
	change the two ranks above. This is impossible. Multiplication of both
	off-diagonal blocks by the Poisson rate \(\nu\) leaves their ranks
	unchanged, so the generator is not block symmetrizable either.
\end{proof}

Thus this example does not arise from positive-definite matrix-valued
orthogonality. Its spectral representation is mixed rather than
matrix-orthogonal: it uses distinct left and right polynomial arrays and the
entrywise positive \(2\times3\) matrix of Pi\~neiro measures.

The spectral representation is obtained by grouping the normalized
Pi\~neiro forms in triples, or in blocks of size $r$ in general. Put
\begin{equation*}
\widetilde{\boldsymbol B}_K^{(b),\mathrm P}(x)
\coloneq
\begin{bNiceMatrix}
\widetilde B_{rK}^{(b),\mathrm P}(x)\\
\Vdots\\
\widetilde B_{rK+r-1}^{(b),\mathrm P}(x)
\end{bNiceMatrix},\qquad
\widetilde{\boldsymbol A}_L^{(a),\mathrm P}(x)
\coloneq
\begin{bNiceMatrix}
\widetilde A_{rL}^{(a),\mathrm P}(x)\\
\Vdots\\
\widetilde A_{rL+r-1}^{(a),\mathrm P}(x)
\end{bNiceMatrix}.
\end{equation*}
Then, for $n\in\Nzero$,
\begin{equation}
\left((P^{\mathrm P})^n\right)_{K,L}
=
\sum_{a=1}^{p}\sum_{b=1}^{q}
\int_0^1
\widetilde{\boldsymbol B}_K^{(b),\mathrm P}(x)
\left(\widetilde{\boldsymbol A}_L^{(a),\mathrm P}(x)\right)^{\mathsf T}
x^{n+\alpha_a+\beta_b}\,\dd x.
\end{equation}
For the continuous-time QBD, the block generator
\(Q^{\mathrm P}=\nu(P^{\mathrm P}-I)\)
has blocks
\begin{equation}
\mathsf Q_K^{\mathrm P,-}=\nu\mathsf P_K^{\mathrm P,-},
\qquad
\mathsf Q_K^{\mathrm P,0}=\nu
\left(\mathsf P_K^{\mathrm P,0}-I_r\right),
\qquad
\mathsf Q_K^{\mathrm P,+}=\nu\mathsf P_K^{\mathrm P,+}.
\end{equation}
Its block transition kernel is
\begin{equation}
\left(\e^{tQ^{\mathrm P}}\right)_{K,L}
=
\sum_{a=1}^{p}\sum_{b=1}^{q}
\int_0^1
\widetilde{\boldsymbol B}_K^{(b),\mathrm P}(x)
\left(\widetilde{\boldsymbol A}_L^{(a),\mathrm P}(x)\right)^{\mathsf T}
\e^{-\nu t(1-x)}x^{\alpha_a+\beta_b}\,\dd x.
\end{equation}

The block representation also gives the QBD first-arrival
matrices.  These contain more information than scalar hitting
probabilities because they record the phase in which the arrival level is
first reached.

\begin{proposition}[Green, resolvent, and first-arrival matrices]
\label{pro:MP-QBD-first-arrival}
For \(0\le z<1\) and \(\lambda>0\), define the block Green matrix and the
block resolvent by
\begin{equation*}
\mathcal G_{K,L}^{\mathrm P}(z)
\coloneq
\sum_{n=0}^{\infty}z^n
\left((P^{\mathrm P})^n\right)_{K,L},\qquad
\mathcal R_{K,L}^{\mathrm P}(\lambda)
\coloneq
\left((\lambda I-Q^{\mathrm P})^{-1}\right)_{K,L}.
\end{equation*}
They have the explicit mixed spectral representations
\begin{align}
\label{eq:MP-QBD-Green-spectral}
\mathcal G_{K,L}^{\mathrm P}(z)
&=
\sum_{a=1}^{p}\sum_{b=1}^{q}
\int_0^1
\widetilde{\boldsymbol B}_K^{(b),\mathrm P}(x)
\left(\widetilde{\boldsymbol A}_L^{(a),\mathrm P}(x)\right)^{\mathsf T}
\frac{x^{\alpha_a+\beta_b}}{1-zx}\,\dd x,
\\
\label{eq:MP-QBD-resolvent-spectral}
\mathcal R_{K,L}^{\mathrm P}(\lambda)
&=
\sum_{a=1}^{p}\sum_{b=1}^{q}
\int_0^1
\widetilde{\boldsymbol B}_K^{(b),\mathrm P}(x)
\left(\widetilde{\boldsymbol A}_L^{(a),\mathrm P}(x)\right)^{\mathsf T}
\frac{x^{\alpha_a+\beta_b}}
{\lambda+\nu(1-x)}\,\dd x.
\end{align}
Their singular parts at the recurrent boundary are
\begin{align}
\label{eq:MP-QBD-Green-asymptotic}
\mathcal G_{K,L}^{\mathrm P}(z)
&\sim
\one_r\boldsymbol\pi_L^{\mathrm P}
\log\frac{1}{1-z},
&&z\uparrow1,
\\
\label{eq:MP-QBD-resolvent-asymptotic}
\mathcal R_{K,L}^{\mathrm P}(\lambda)
&\sim
\frac{1}{\nu}\one_r\boldsymbol\pi_L^{\mathrm P}
\log\frac{\nu}{\lambda},
&&\lambda\downarrow0.
\end{align}

Suppose that \(K\ne L\), and put
\(\tau_L\coloneq\inf\{n\ge0:K_n^{\mathrm P}=L\}\).  The discrete
first-arrival matrix, including the arrival phase, is
\begin{equation}
\label{eq:MP-QBD-first-arrival-definition}
\left(\mathcal F_{K,L}^{\mathrm P}(z)\right)_{i,j}
\coloneq
\operatorname E_{K,i}\!\left[
z^{\tau_L}\ind_{\{I_{\tau_L}^{\mathrm P}=j\}}
\right].
\end{equation}
It is determined by
\begin{equation}
\label{eq:MP-QBD-first-arrival-discrete}
\mathcal F_{K,L}^{\mathrm P}(z)
=
\mathcal G_{K,L}^{\mathrm P}(z)
\left(\mathcal G_{L,L}^{\mathrm P}(z)\right)^{-1}.
\end{equation}
For a first positive return to level \(L\), put
\(\tau_L^+\coloneq\inf\{n\ge1:K_n^{\mathrm P}=L\}\) and define
\(\mathcal F_{L,L}^{\mathrm P,+}(z)\) as in
\eqref{eq:MP-QBD-first-arrival-definition}, with \(\tau_L\) replaced by
\(\tau_L^+\).  Then
\begin{equation}
\label{eq:MP-QBD-first-return-matrix}
\mathcal F_{L,L}^{\mathrm P,+}(z)
=
I_r-\left(\mathcal G_{L,L}^{\mathrm P}(z)\right)^{-1}.
\end{equation}

For the continuous-time QBD, let
\(\sigma_L\coloneq\inf\{t\ge0:K^{\mathrm P}(t)=L\}\).  When
\(K\ne L\), its first-arrival transform is
\begin{equation}
\label{eq:MP-QBD-first-arrival-continuous-definition}
\left(\mathcal H_{K,L}^{\mathrm P}(\lambda)\right)_{i,j}
\coloneq
\operatorname E_{K,i}\!\left[
\e^{-\lambda\sigma_L}
\ind_{\{I^{\mathrm P}(\sigma_L)=j\}}
\right],
\end{equation}
and
\begin{equation}
\label{eq:MP-QBD-first-arrival-continuous}
\mathcal H_{K,L}^{\mathrm P}(\lambda)
=
\mathcal R_{K,L}^{\mathrm P}(\lambda)
\left(\mathcal R_{L,L}^{\mathrm P}(\lambda)\right)^{-1}.
\end{equation}
Because the QBD is recurrent, the undiscounted first-arrival matrices are
stochastic:
\begin{equation}
\label{eq:MP-QBD-eventual-level-arrival}
\mathcal F_{K,L}^{\mathrm P}(1)\one_r
=
\mathcal H_{K,L}^{\mathrm P}(0)\one_r
=
\one_r,
\qquad K\ne L,
\end{equation}
where the values at the boundary are understood as monotone limits.
\end{proposition}

\begin{proof}
Summing the block transition formula over \(n\), and using
\(\sum_{n\ge0}(zx)^n=(1-zx)^{-1}\), gives
\eqref{eq:MP-QBD-Green-spectral}.  Laplace transformation of the
continuous-time block kernel gives
\eqref{eq:MP-QBD-resolvent-spectral}.  Equations
\eqref{eq:MP-QBD-Green-asymptotic} and
\eqref{eq:MP-QBD-resolvent-asymptotic} follow by collecting the scalar
relations \eqref{eq:MP-critical-Green} and
\eqref{eq:MP-critical-resolvent} into \(r\times r\) blocks.

By the strong Markov property at \(\tau_L\), every visit from level \(K\)
to level \(L\) factors into the first arrival, including its phase, and a
subsequent path that starts in level \(L\).  Therefore
\begin{equation}
\mathcal G_{K,L}^{\mathrm P}(z)
=
\mathcal F_{K,L}^{\mathrm P}(z)
\mathcal G_{L,L}^{\mathrm P}(z),
\qquad K\ne L.
\end{equation}
For a return to the initial level, the zero-step visit must be separated:
\begin{equation}
\mathcal G_{L,L}^{\mathrm P}(z)
=
I_r+
\mathcal F_{L,L}^{\mathrm P,+}(z)
\mathcal G_{L,L}^{\mathrm P}(z).
\end{equation}
The second identity also shows that
\(\mathcal G_{L,L}^{\mathrm P}(z)\) is invertible and yields
\eqref{eq:MP-QBD-first-return-matrix}; the first then gives
\eqref{eq:MP-QBD-first-arrival-discrete}.  The strong Markov property at
\(\sigma_L\) similarly gives
\(\mathcal R_{K,L}^{\mathrm P}(\lambda)
=\mathcal H_{K,L}^{\mathrm P}(\lambda)
\mathcal R_{L,L}^{\mathrm P}(\lambda)\).  The diagonal block
\(\mathcal R_{L,L}^{\mathrm P}(\lambda)\) is nonsingular because
uniformization gives
\(\mathcal R_{L,L}^{\mathrm P}(\lambda)
=(\lambda+\nu)^{-1}
\mathcal G_{L,L}^{\mathrm P}(\nu/(\lambda+\nu))\).  This proves
\eqref{eq:MP-QBD-first-arrival-continuous}.  Finally, irreducible
recurrence implies that every level is reached almost surely, which gives
\eqref{eq:MP-QBD-eventual-level-arrival}.
\end{proof}

\begin{corollary}[Continued fraction for the mixed Pi\~neiro returns]
\label{coro:MP-return-CF}
For \(K\ge0\), let
\(\mathcal G_K^{\mathrm P,\uparrow}(z)\) be the Green matrix of the
mixed Pi\~neiro chain restricted to levels \(K,K+1,\ldots\).  Then
\begin{equation}
\label{eq:MP-return-CF}
\mathcal G_K^{\mathrm P,\uparrow}(z)
=
\left[
I_r-z\mathsf P_K^{\mathrm P,0}
-z^2\mathsf P_K^{\mathrm P,+}
\mathcal G_{K+1}^{\mathrm P,\uparrow}(z)
\mathsf P_{K+1}^{\mathrm P,-}
\right]^{-1}.
\end{equation}
At the boundary there is no lower level, and therefore
\begin{equation}
\label{eq:MP-boundary-Green-CF-spectral}
\begin{split}
\mathcal G_0^{\mathrm P,\uparrow}(z)
&=
\mathcal G_{0,0}^{\mathrm P}(z)
=
\sum_{a=1}^{p}\sum_{b=1}^{q}
\int_0^1
\widetilde{\boldsymbol B}_0^{(b),\mathrm P}(x)
\left(\widetilde{\boldsymbol A}_0^{(a),\mathrm P}(x)\right)^{\mathsf T}
\frac{x^{\alpha_a+\beta_b}}{1-zx}\,\dd x.
\end{split}
\end{equation}
The matrix generating function of the first positive return to level \(0\)
is consequently
\begin{equation}
\label{eq:MP-boundary-first-return-CF}
\mathcal F_{0,0}^{\mathrm P,+}(z)
=
z\mathsf P_0^{\mathrm P,0}
+z^2\mathsf P_0^{\mathrm P,+}
\mathcal G_1^{\mathrm P,\uparrow}(z)
\mathsf P_1^{\mathrm P,-}.
\end{equation}

For the bounded continuous-time process \(Q^{\mathrm P}=\nu(P^{\mathrm P}-I)\),
the corresponding tail resolvents satisfy
\begin{equation}
\label{eq:MP-resolvent-CF}
\mathcal R_K^{\mathrm P,\uparrow}(\lambda)
=
\left[
(\lambda+\nu)I_r-\nu\mathsf P_K^{\mathrm P,0}
-\nu^2\mathsf P_K^{\mathrm P,+}
\mathcal R_{K+1}^{\mathrm P,\uparrow}(\lambda)
\mathsf P_{K+1}^{\mathrm P,-}
\right]^{-1}.
\end{equation}
\end{corollary}

\begin{proof}
Apply Proposition~\ref{pro:level-return-CF} to the explicit blocks of
\(P^{\mathrm P}\).  The hypotheses hold because
Proposition~\ref{pro:banded-kernel-QBD} and the explicit grouping above
have already shown that the
Pi\~neiro kernel is block tridiagonal with precisely the blocks
\(\mathsf P_K^{\mathrm P,-}\), \(\mathsf P_K^{\mathrm P,0}\), and
\(\mathsf P_K^{\mathrm P,+}\).  Since level \(0\) is the boundary of the
whole state space, the chain restricted to the levels
\(0,1,\ldots\) is the original chain.  Hence its boundary Green matrix is
\(\mathcal G_{0,0}^{\mathrm P}(z)\), and
\eqref{eq:MP-QBD-Green-spectral} gives
\eqref{eq:MP-boundary-Green-CF-spectral}.  Formula
\eqref{eq:MP-boundary-first-return-CF} is
\eqref{eq:level-first-return-CF} at \(K=0\), and
\eqref{eq:MP-resolvent-CF} follows from
\eqref{eq:level-resolvent-CF-recursion}.
\end{proof}

All coefficients in \eqref{eq:MP-return-CF} depend explicitly on both
parameter vectors \(\boldsymbol\alpha\) and \(\boldsymbol\beta\).  Indeed,
\(\mathsf P_K^{\mathrm P,-}\), \(\mathsf P_K^{\mathrm P,0}\), and
\(\mathsf P_K^{\mathrm P,+}\) are assembled from the scalar probabilities
\(p_{N,N+s}^{\mathrm P}\), whose finite formulas were obtained from the
mixed Pi\~neiro recurrence coefficients and the harmonic vector.  In the
ordered PBF region, the factor-resolved construction of
Subsection~\ref{subsec:factor-resolved-return-CF} uses directly the elementary
probabilities \(\delta_k(N)\) and \(\gamma_j(N)\).  These are again explicit
functions of \(\boldsymbol\alpha\) and \(\boldsymbol\beta\). Thus the coefficients at
every level of that continued fraction are the individual urn
probabilities.

For the rational \((p,q)=(3,2)\) example, \(r=3\).  Therefore
\eqref{eq:MP-return-CF} is a continued fraction of \(3\times3\) matrices,
and every finite convergent is a \(3\times3\) rational matrix in \(z\).
Cutting it after level \(N\) gives the exact Green matrix of the rational
chain stopped on leaving levels \(0,\ldots,N\), and the finite fractions
increase entry by entry to the spectral matrix in
\eqref{eq:MP-boundary-Green-CF-spectral}.

All Green, resolvent, and first-arrival matrices in
Proposition~\ref{pro:MP-QBD-first-arrival} are effectively computable.
Substitution
of the finite Pi\~neiro expansions reduces every entry of
\eqref{eq:MP-QBD-Green-spectral} to finite sums of
\begin{equation}
\int_0^1\frac{x^{\rho-1}}{1-zx}\,\dd x
=
\frac{1}{\rho}\;\pFq{2}{1}{1,\rho}{\rho+1}{z},
\end{equation}
while \eqref{eq:MP-QBD-resolvent-spectral} is evaluated by
\eqref{eq:MP-resolvent-spectral-explicit}.  Thus, for the rational
three-phase example, the first-arrival transforms are obtained from finite
hypergeometric sums followed only by the inversion of an explicit
\(3\times3\) matrix.  This provides an explicit alternative to the single
matrix-quadratic rate equation of a level-homogeneous QBD, which is not
available here because the Pi\~neiro blocks depend on the level.

For the rational QBD, the first passage from the boundary level (0) to
level (1) is particularly concrete.  Before that first arrival the chain
can only move inside level (0); hence
\begin{align}
\mathcal F_{0,1}^{\mathrm P}(z)
&=
z\left(I_3-z\mathsf P_0^{\mathrm P,0}\right)^{-1}
\mathsf P_0^{\mathrm P,+},
&
\mathcal H_{0,1}^{\mathrm P}(\lambda)
&=
\nu\left((\lambda+\nu)I_3-
\nu\mathsf P_0^{\mathrm P,0}\right)^{-1}
\mathsf P_0^{\mathrm P,+}.
\end{align}
At the undiscounted boundary these two matrices coincide and the displayed
rational blocks give
\begin{equation}
\label{eq:MP-rational-QBD-first-arrival-boundary-explicit}
\mathcal F_{0,1}^{\mathrm P}(1)
=
\mathcal H_{0,1}^{\mathrm P}(0)
=
\begin{bNiceMatrix}[cell-space-limits=2pt]
\frac{69}{88}&\frac{19}{88}&0\\
\frac{141}{176}&\frac{35}{176}&0\\
\frac34&\frac14&0
\end{bNiceMatrix}.
\end{equation}
Each row sums to one.  The zero third column has a direct geometric
meaning: since the largest upward scalar jump satisfies \(q=2<r=3\), the
first entrance from level \(0\) into level \(1\) can occur only in phase
\(0\) or phase \(1\), never in phase \(2\).

\subsection{State-dependent time change and the unbounded process}

The preceding construction applies directly to the mixed Pi\~neiro kernel
already developed in this paper. Let
\((\boldsymbol\alpha,\boldsymbol\beta)\in\mathcal R^{\mathrm P,+}_{p,q}\), so that
\(P^{\mathrm P}\) has the positive stochastic bidiagonal factorization
\eqref{eq:MP-SBF} and the ordered sequence of \(p+q\) urn experiments in
Proposition~\ref{pro:MP-finite-urn-model}. Since \(P^{\mathrm P}\) is
strictly positive throughout its structural band,
\((P^{\mathrm P})_{N,N}<1\) for every \(N\). Define
\begin{equation}
\label{eq:MP-unbounded-local-rates}
v_N^{\mathrm P}
\coloneq
\frac{N+1}{1-(P^{\mathrm P})_{N,N}},
\qquad
V^{\mathrm P}\coloneq
\diag(v_0^{\mathrm P},v_1^{\mathrm P},\ldots),
\end{equation}
and
\begin{equation}
\label{eq:MP-unbounded-generator}
Q_{\mathrm{loc}}^{\mathrm P}
\coloneq
V^{\mathrm P}(P^{\mathrm P}-I).
\end{equation}
This is a state-dependent time change of the continuous-time chain with
generator \(P^{\mathrm P}-I\): it changes the holding rates but preserves
the embedded event chain \(P^{\mathrm P}\).

\begin{proposition}[The state-dependent Pi\~neiro generator]
\label{pro:MP-unbounded-generator}
The matrix \(Q_{\mathrm{loc}}^{\mathrm P}\) is an irreducible,
non-explosive, unbounded generator with exact bandwidth \((p,q)\). Its exit
rate from \(N\) is \(N+1\). At state \(N\), its probabilistic realization
runs the same \(p+q\) Pi\~neiro urns as the discrete model, but at the
state-dependent rate \(v_N^{\mathrm P}\).

All its jump rates are explicit finite expressions in
\(\boldsymbol\alpha,\boldsymbol\beta\) and \(N\). At rational parameter points they are
rational. In particular, the rational \((3,2)\) model of
Example~\ref{ex:MP-rational-urn-cycle} gives a fully explicit unbounded
five-urn process with off-diagonal jumps in
\(\{-3,-2,-1,1,2\}\), with virtual stays allowed in the local-clock
realization.
\end{proposition}

\begin{proof}
Apply Proposition~\ref{pro:canonical-unbounded-local-clock} to
\(T=P^{\mathrm P}\). Strict positivity throughout the structural band
preserves the exact bandwidth, and \eqref{eq:MP-unbounded-local-rates}
makes the exit rate
equal to \(N+1\). The factor entries, the entries of \(P^{\mathrm P}\),
and hence its diagonal are given explicitly in the Pi\~neiro section.
Therefore the local rates and all the entries of
\(Q_{\mathrm{loc}}^{\mathrm P}\) are explicit in the parameters.
\end{proof}

For the rational \((3,2)\) example, the diagonal entries displayed in
\eqref{eq:MP-rational-example-P} give, for instance,
\[
v_0^{\mathrm P}=\frac53,\qquad
v_1^{\mathrm P}=\frac{70}{23},\qquad
v_2^{\mathrm P}=\frac{360}{77},\qquad
v_3^{\mathrm P}=\frac{12320}{2029}.
\]
Multiplying the first two rows of \(P^{\mathrm P}-I\) by these local
rates gives
\begin{equation}
\label{eq:MP-unbounded-generator-first-rows}
\left.Q_{\mathrm{loc}}^{\mathrm P}\right|_{\{0,1\}\times\{0,1,2,3\}}
=
\begin{bNiceMatrix}[cell-space-limits=3pt]
-1&\frac23&\frac13&0\\
\frac{14}{23}&-2&\frac{49}{46}&\frac{15}{46}
\end{bNiceMatrix}.
\end{equation}
Thus the two displayed exit rates are \(1\) and \(2\), while the
off-diagonal rates are the rational wide-jump probabilities of the urn
kernel multiplied by the corresponding local rate.

The event chain \(P^{\mathrm P}\) is null recurrent by
Proposition~\ref{pro:MP-null-recurrence-local-limits}. Hence
\(Q_{\mathrm{loc}}^{\mathrm P}\) is recurrent. Its invariant measure is
\begin{equation}
\label{eq:MP-unbounded-invariant-measure}
\pi_{N,\mathrm{loc}}^{\mathrm P}
=
\frac{\pi_N^{\mathrm P}}{v_N^{\mathrm P}}
=
\pi_N^{\mathrm P}
\frac{1-(P^{\mathrm P})_{N,N}}{N+1},
\end{equation}
where
\(\pi_N^{\mathrm P}=h_N^{\mathrm P}A_N^{\mathrm P}(1)\) is the explicit
invariant measure in \eqref{eq:MP-invariant-measure}. In the present
Pi\~neiro model, the total mass of this measure can be determined
explicitly.

\begin{proposition}[Recurrence classification of the state-dependent process]
\label{pro:MP-unbounded-null-recurrence}
For every
\((\boldsymbol\alpha,\boldsymbol\beta)\in\mathcal R^{\mathrm P,+}_{p,q}\), one has
\begin{equation}
\label{eq:MP-invariant-measure-simplified}
\pi_N^{\mathrm P}
=
\xi_N+\eta_N
=
\alpha_{i_N}+\beta_{s_N}+e_N+d_N+1.
\end{equation}
Moreover,
\begin{equation}
\label{eq:MP-top-transition-asymptotic}
\lim_{N\to\infty}
(P^{\mathrm P})_{N,N+q}
=
\lim_{N\to\infty}
\frac{h_{N+q}^{\mathrm P}}{h_N^{\mathrm P}}
=
\left(\frac{p}{p+q}\right)^{p+q}.
\end{equation}
Consequently,
\begin{equation}
\label{eq:MP-unbounded-positive-recurrence}
\sum_{N=0}^{\infty}
\pi_N^{\mathrm P}
\frac{1-(P^{\mathrm P})_{N,N}}{N+1}
=+\infty,
\end{equation}
and \(Q_{\mathrm{loc}}^{\mathrm P}\) is null recurrent.
\end{proposition}

\begin{proof}
First, the invariant measure is simplified. For pairwise distinct nodes
\(z_0,\ldots,z_N\) and a monic polynomial \(f\) of degree \(N\), Lagrange
interpolation gives
\[
\sum_{K=0}^{N}
\frac{f(z_K)}
{\prod_{\substack{0\le r\le N\\r\ne K}}(z_K-z_r)}
=1.
\]
Evaluate the first formula in
\eqref{eq:MP-Cauchy-direct-step-line-forms} at \(x=1\), substitute
\eqref{eq:MP-Cauchy-lower-inverse}, and apply this identity with
\(z_K=\xi_K\) and
\(f(z)=\prod_{r=0}^{N-1}(z+\eta_r)\). This gives
\begin{equation}
\label{eq:MP-h-Cauchy-product}
h_N^{\mathrm P}
=
\frac{\prod_{r=0}^{N-1}(\xi_N-\xi_r)}
{\prod_{r=0}^{N-1}(\xi_N+\eta_r)}.
\end{equation}
Applying the same interpolation identity to the second formula in
\eqref{eq:MP-Cauchy-direct-step-line-forms}, now using
\eqref{eq:MP-Cauchy-upper-inverse-transpose} and
\(f(z)=\prod_{r=0}^{N-1}(\xi_r+z)\), yields
\[
A_N^{\mathrm P}(1)
=
\frac{\prod_{r=0}^{N}(\xi_N+\eta_r)}
{\prod_{r=0}^{N-1}(\xi_N-\xi_r)}.
\]
Multiplication cancels every factor except
\(\xi_N+\eta_N\), proving
\eqref{eq:MP-invariant-measure-simplified}.

Next, the size of the extreme upward transition is determined. Write
\(N=qd_N+s_N-1=pe_N+i_N-1\). Grouping the factors in
\eqref{eq:MP-h-Cauchy-product} according to their beta and alpha
indices gives
\begin{equation}
\label{eq:MP-h-grouped-product}
h_N^{\mathrm P}
=
\frac{
d_N!\,
\prod_{j=1}^{s_N-1}
(\beta_{s_N}-\beta_j)_{d_N+1}
\prod_{j=s_N+1}^{q}
(\beta_{s_N}-\beta_j+1)_{d_N}
}{
\prod_{a=1}^{i_N-1}
(\alpha_a+\beta_{s_N}+d_N+1)_{e_N+1}
\prod_{a=i_N}^{p}
(\alpha_a+\beta_{s_N}+d_N+1)_{e_N}
}.
\end{equation}
Empty products are understood as \(1\). Formula
\eqref{eq:MP-h-grouped-product} also gives the large-\(N\) asymptotics
directly. Replace each Pochhammer symbol by a quotient of gamma
functions and use
\(\Gamma(x+c)/\Gamma(x+d)
=x^{c-d}(1+\mathrm O(x^{-1}))\), \(x\to\infty\).
Since \(d_N=N/q+O(1)\) and \(e_N=N/p+\mathrm O(1)\), \(N\to\infty\), the powers contributed by
the finitely many alpha and beta factors cancel, and on every residue
class modulo \(\operatorname{lcm}(p,q)\) one obtains
\[
\frac{h_{N+q}^{\mathrm P}}{h_N^{\mathrm P}}
=
\left(\frac{p}{p+q}\right)^{p+q}
\left(1+\mathrm O\left(\frac1N\right)\right),\qquad N\to\infty.
\]
The limit is the same on all residue classes. The monic normalization of
the type-II recurrence gives
\((T^{\mathrm P})_{N,N+q}=1\). Hence the Doob transform satisfies
\((P^{\mathrm P})_{N,N+q}=h_{N+q}^{\mathrm P}/h_N^{\mathrm P}\), which
proves \eqref{eq:MP-top-transition-asymptotic}.

Finally, \eqref{eq:MP-invariant-measure-simplified} gives
\[
\frac{\pi_N^{\mathrm P}}{N+1}
\longrightarrow
\frac1p+\frac1q.
\]
Since
\(1-(P^{\mathrm P})_{N,N}\ge(P^{\mathrm P})_{N,N+q}\), it follows that
\[
\liminf_{N\to\infty}
\pi_N^{\mathrm P}
\frac{1-(P^{\mathrm P})_{N,N}}{N+1}
\ge
\left(\frac1p+\frac1q\right)
\left(\frac{p}{p+q}\right)^{p+q}>0.
\]
The summands in \eqref{eq:MP-unbounded-positive-recurrence} therefore do
not tend to zero, so the invariant measure of
\(Q_{\mathrm{loc}}^{\mathrm P}\) has infinite total mass. The process is
recurrent and irreducible, and is consequently null recurrent.
\end{proof}

For \(\lambda>0\), define
\[
D_\lambda^{\mathrm P}
\coloneq
\diag\left(
\frac{v_N^{\mathrm P}}{\lambda+v_N^{\mathrm P}}
\right)_{N\ge0},
\qquad
K_\lambda^{\mathrm P}
\coloneq D_\lambda^{\mathrm P}P^{\mathrm P}.
\]
The resolvent is
\begin{equation}
\label{eq:MP-unbounded-resolvent}
\bigl(R_{Q_{\mathrm{loc}}^{\mathrm P}}(\lambda)\bigr)_{N,M}
=
\frac{1}{\lambda+v_M^{\mathrm P}}
\sum_{k=0}^{\infty}
\bigl((K_\lambda^{\mathrm P})^k\bigr)_{N,M}.
\end{equation}
The kernel \(K_\lambda^{\mathrm P}\) has the explicit PBF obtained from
\eqref{eq:MP-SBF} by
Proposition~\ref{pro:local-kernel-PBF}. Therefore
\eqref{eq:local-resolvent-spectral-Abel} supplies a positive mixed
spectral representation of every discounted potential, and
\eqref{eq:local-Laplace-hitting-time} gives all first-passage Laplace
transforms.

Finally, grouping the states of \(Q_{\mathrm{loc}}^{\mathrm P}\) by the same
level--phase bijection used in Subsection~\ref{subsec:MP-QBD} produces an
unbounded level-dependent QBD generator. Its block resolvent is obtained by
grouping the entries in \eqref{eq:MP-unbounded-resolvent}; the strong-Markov matrix
quotients in Proposition~\ref{pro:MP-QBD-first-arrival} then give the
Laplace transforms of first arrival at a prescribed level, including the
arrival phase.

\section{Conclusions and outlook}

The main conclusion is that an ordered PBF carries substantially more
information than the product matrix.  The product determines the
left--right mixed spectral data and the usual path quantities of the Markov
chain.  The ordered factors additionally determine elementary Markov
experiments, their finite-urn realizations, Darboux intertwinings, and the
continued fractions for return laws.  These descriptions remain scalar and
retain the original bandwidth, while grouping states gives a complementary
finite-phase QBD description.  Thus non-reversible banded kernels can be
treated spectrally without forcing scalar or block self-adjointness.

Two structural consequences clarify the scope of the theory.  In
continuous time, scalar PBFs of all shifted leading truncations are too
rigid for an unbounded conservative generator of bandwidth greater than
one.  State-dependent uniformization separates the two roles instead:
\(T\) fixes the spectral and hitting structure, while \(V\) changes the
holding times and the invariant measure without changing the bandwidth.
For the explicit mixed system, the sign classifications show that
positivity of the recurrence band is weaker than positivity of every
bidiagonal factor.  The Pi\~neiro system exhibits both phenomena explicitly.
In its PBF chamber, positivity and boundedness are independent: the
cyclic-spacing criterion identifies both bounded and unbounded recurrences.
The positive harmonic vector still produces a bounded stochastic band
matrix from an unbounded recurrence, but the diagonal conjugation is then
necessarily unbounded.  Thus this explicit family contributes
simultaneously to total positivity, mixed multiple orthogonality, and
non-reversible Markov theory.

Several problems remain.  The maximal Pi\~neiro region in which every entry
permitted by the bandwidth is positive is unknown.  Further questions concern
block symmetrizability when \(p=q\), the small-\(\lambda\) behavior of the
killed resolvents, and positive spectral resolutions for unbounded
non-reversible generators.  Each problem asks how much of the factor-level
structure survives after one of the positivity, symmetry, or boundedness
hypotheses is weakened.

\appendix

\section{Derivation of the mixed Pi\~neiro bidiagonal coefficients}
\label{app:MP-bidiagonal-derivation}

This appendix derives the factor entries
\(\ell_{a,N}^{\mathrm P}\) and \(u_{b,N}^{\mathrm P}\) used in
\eqref{eq:MP-PBF} directly from the componentwise Cauchy formula
\eqref{eq:MP-explicit-A-coefficients}. The argument has two parts. First,
the elementary Christoffel connections express the bidiagonal entries as
quotients of active-component coefficients. Second, the step-line indices
reduce those quotients to the Pochhammer products displayed in
\eqref{eq:MP-lower-Pochhammer-simplified} and
\eqref{eq:MP-upper-Pochhammer-simplified}.

For \(r,s\in\Nzero\), write \(A_N^{[r,s]}\) for the type-I step-line form
obtained by replacing
\((\boldsymbol\alpha,\boldsymbol\beta)\) with
\((\boldsymbol\alpha^{[r]},\boldsymbol\beta^{[s]})\), in the notation of
Definition~\ref{def:MP-affine-cyclic-shifts}. Its active component is the
\(i_N\)-th component and has degree \(e_N\). Define its leading
coefficient by
\begin{equation}
\label{eq:app-MP-shifted-leading-coefficient-definition}
\lambda_N^{[r,s]}
\coloneq
[x^{e_N}]\left(A_N^{[r,s]}\right)^{(i_N)}(x).
\end{equation}

\begin{lemma}[Shifted active-component coefficient]
\label{lem:app-MP-shifted-leading-coefficient}
For every \(N,r,s\ge0\),
\begin{equation}
\label{eq:app-MP-shifted-leading-coefficient}
\begin{aligned}
\lambda_N^{[r,s]}
&=
\frac{
\prod_{i=1}^{p}
(\alpha_{r+i}+\beta_{s+s_N}+d_N+1)_{n_{N,i}}
}{
\prod_{j=1}^{q}
(\beta_{s+j}-\beta_{s+s_N}-d_N)_{m_{N,j}}
}
\\
&\quad\times
\frac{
(-1)^{e_N}
\prod_{j=1}^{q}
(\alpha_{r+i_N}+\beta_{s+j}+e_N+1)_{m_{N,j}}
}{
e_N!
\prod_{\substack{1\le i\le p\\i\ne i_N}}
(\alpha_{r+i}-\alpha_{r+i_N}-e_N)_{n_{N,i}}
}.
\end{aligned}
\end{equation}
\end{lemma}

\begin{proof}
Apply \eqref{eq:MP-explicit-A-coefficients} to the shifted parameters and
choose the active component \(i=i_N\). By
\eqref{eq:MP-type-I-step-indices}, one has
\(n_{N,i_N}=e_N+1\), so the coefficient of the highest power \(x^{e_N}\)
is obtained by setting \(k=e_N\). The factorial
\((n_{N,i_N}-e_N-1)!\) is then \(0!=1\). Substitution of the shifted
normalization constant \(C_N^A\) gives
\eqref{eq:app-MP-shifted-leading-coefficient}.
\end{proof}

\begin{lemma}[Elementary Christoffel quotients]
\label{lem:app-MP-Christoffel-quotients}
The nontrivial entries of the elementary bidiagonal factors satisfy
\begin{equation}
\label{eq:app-MP-Christoffel-quotients}
\ell_{a,N}^{\mathrm P}
=
\frac{\lambda_N^{[a,0]}}{\lambda_{N+1}^{[a-1,0]}},
\qquad
u_{b,N}^{\mathrm P}
=
\frac{\lambda_N^{[0,b-1]}}{\lambda_N^{[0,b]}},
\end{equation}
for \(1\le a\le p\), \(1\le b\le q\), and \(N\ge0\).
Iterating these elementary connections through one complete alpha cycle
and one complete beta cycle gives
\[
T^{\mathrm P}
=L_1^{\mathrm P}\cdots L_p^{\mathrm P}
 U_q^{\mathrm P}\cdots U_1^{\mathrm P}.
\]
\end{lemma}

\begin{proof}
Let \(S\) denote the unilateral shift,
\(S_{N,M}=\delta_{M,N+1}\). Under one affine alpha shift, the ordered
column of weights changes from
\[
(x^{\alpha_1},x^{\alpha_2},\ldots,x^{\alpha_p})
\quad\hbox{to}\quad
(x^{\alpha_2},\ldots,x^{\alpha_p},x^{\alpha_1+1}),
\]
and the beta shift acts analogously on the ordered row of beta weights.
Consequently, the scalar moment matrices of the shifted systems satisfy
\[
\mathscr M_L^{[r]}=\mathscr M(S^r)^{\mathsf T},
\qquad
\mathscr M_R^{[s]}=S^s\mathscr M.
\]
The shifted exponent sequences are affine subsequences of the original
ones. The Cauchy determinant argument in
\eqref{eq:MP-Cauchy-determinant} therefore applies without change and shows
that all the following factorizations exist uniquely.
Use the normalized Gauss--Borel factorizations
\[
\mathscr M_L^{[r]}
=(\mathscr L_L^{[r]})^{-1}(\mathscr U_L^{[r]})^{-1},
\qquad
\mathscr M_R^{[s]}
=(\mathscr L_R^{[s]})^{-1}(\mathscr U_R^{[s]})^{-1},
\]
where the \(\mathscr L\)'s are lower unitriangular and the
\(\mathscr U\)'s are upper triangular.

The identity
\(\mathscr M_L^{[a]}=\mathscr M_L^{[a-1]}S^{\mathsf T}\), together with
uniqueness of Gauss--Borel factorization, gives
\begin{equation}
\label{eq:app-MP-left-Christoffel-connection}
(\mathscr U_L^{[a-1]})^{-1}S^{\mathsf T}
=L_a^{\mathrm P}(\mathscr U_L^{[a]})^{-1}.
\end{equation}
The matrix on the left is upper Hessenberg. Its normalized
Gauss--Borel factorization therefore has a lower bidiagonal factor with
unit diagonal, which is \(L_a^{\mathrm P}\). Comparing the
\((N+1,N)\) entry in
\eqref{eq:app-MP-left-Christoffel-connection} gives
\[
(L_a^{\mathrm P})_{N+1,N}
=
\frac{(\mathscr U_L^{[a]})_{N,N}}
     {(\mathscr U_L^{[a-1]})_{N+1,N+1}}.
\]
The diagonal entry \((\mathscr U_L^{[a]})_{N,N}\) is the leading
coefficient of the active component of the \(N\)-th type-I form for the
alpha-shifted system. It is therefore \(\lambda_N^{[a,0]}\), which proves
the first quotient in \eqref{eq:app-MP-Christoffel-quotients}.

Similarly,
\(\mathscr M_R^{[b]}=S\mathscr M_R^{[b-1]}\) gives
\begin{equation}
\label{eq:app-MP-right-Christoffel-connection}
S(\mathscr L_R^{[b-1]})^{-1}
=(\mathscr L_R^{[b]})^{-1}U_b^{\mathrm P},
\qquad
\mathscr U_R^{[b]}
=\mathscr U_R^{[b-1]}(U_b^{\mathrm P})^{-1}.
\end{equation}
Here \(U_b^{\mathrm P}\) is upper bidiagonal with unit superdiagonal.
Taking diagonal entries in the second identity gives
\[
(U_b^{\mathrm P})_{N,N}
=
\frac{(\mathscr U_R^{[b-1]})_{N,N}}
     {(\mathscr U_R^{[b]})_{N,N}}
=
\frac{\lambda_N^{[0,b-1]}}{\lambda_N^{[0,b]}},
\]
which proves the second quotient.

Finally, multiplication by \(x\) in the alpha basis gives
\[
T^{\mathrm P}
=(\mathscr U_L^{[0]})^{-1}(S^p)^{\mathsf T}
\mathscr U_L^{[0]}.
\]
Iterating \eqref{eq:app-MP-left-Christoffel-connection} through \(p\)
steps yields
\[
T^{\mathrm P}
=L_1^{\mathrm P}\cdots L_p^{\mathrm P}
(\mathscr U_L^{[p]})^{-1}\mathscr U_L^{[0]}.
\]
A complete alpha cycle and a complete beta cycle both multiply the matrix
of measures by \(x\). Hence
\(\mathscr M_L^{[p]}=\mathscr M_R^{[q]}\), and uniqueness gives
\(\mathscr U_L^{[p]}=\mathscr U_R^{[q]}\). Iterating the second identity
in \eqref{eq:app-MP-right-Christoffel-connection}, and using
\(\mathscr U_R^{[0]}=\mathscr U_L^{[0]}\), gives
\[
(\mathscr U_L^{[p]})^{-1}\mathscr U_L^{[0]}
=U_q^{\mathrm P}\cdots U_1^{\mathrm P}.
\]
This proves the stated bidiagonal factorization.
\end{proof}

To evaluate the quotients, it is convenient to isolate the products
containing only beta differences. For \(s\ge0\), put
\begin{equation}
\label{eq:app-MP-beta-difference-product}
E_N^{[s]}
\coloneq
\prod_{j=1}^{q}
(\beta_{s+j}-\beta_{s+s_N}-d_N)_{m_{N,j}}.
\end{equation}
Since \(m_{N,j}=d_N+1\) for \(j<s_N\) and \(m_{N,j}=d_N\) for
\(j\ge s_N\), this product can also be written as
\begin{equation}
\label{eq:app-MP-beta-difference-product-expanded}
E_N^{[s]}
=
\prod_{j=1}^{q}
(\beta_{s+j}-\beta_{s+s_N}-d_N)_{d_N}
\prod_{j=1}^{s_N-1}
(\beta_{s+j}-\beta_{s+s_N}).
\end{equation}

\begin{proposition}[Evaluation of the Christoffel quotients]
\label{pro:app-MP-Christoffel-quotient-evaluation}
Substitution of \eqref{eq:app-MP-shifted-leading-coefficient} into
\eqref{eq:app-MP-Christoffel-quotients} gives precisely
\eqref{eq:MP-lower-Pochhammer-simplified} and
\eqref{eq:MP-upper-Pochhammer-simplified}.
\end{proposition}

\begin{proof}
First, the upper quotient is evaluated. All factors in
\eqref{eq:app-MP-shifted-leading-coefficient} that contain only alpha
differences, factorials, or signs are independent of the beta shift and
cancel. The remaining factors give
\begin{equation}
\label{eq:app-MP-upper-quotient-evaluation}
\begin{aligned}
\frac{\lambda_N^{[0,b-1]}}{\lambda_N^{[0,b]}}
&=
\prod_{i=1}^{p}
\frac{
(\alpha_i+\beta_{b-1+s_N}+d_N+1)_{n_{N,i}}
}{
(\alpha_i+\beta_{b+s_N}+d_N+1)_{n_{N,i}}
}
\frac{E_N^{[b]}}{E_N^{[b-1]}}
\\
&\quad\times
\prod_{j=1}^{q}
\frac{
(\alpha_{i_N}+e_N+\beta_{b-1+j}+1)_{m_{N,j}}
}{
(\alpha_{i_N}+e_N+\beta_{b+j}+1)_{m_{N,j}}
}.
\end{aligned}
\end{equation}
Because \(n_{N,i_N}=e_N+1\), the last product is the third product in
\eqref{eq:MP-upper-Pochhammer-simplified}. Expanding the quotient of the
two \(E_N\)'s by \eqref{eq:app-MP-beta-difference-product} gives its second
product. Hence \eqref{eq:app-MP-upper-quotient-evaluation} is exactly
\eqref{eq:MP-upper-Pochhammer-simplified}.

For the lower quotient, direct substitution gives
\begin{equation}
\label{eq:app-MP-lower-quotient-before-cancellation}
\frac{\lambda_N^{[a,0]}}{\lambda_{N+1}^{[a-1,0]}}
=
\frac{E_{N+1}^{[0]}}{E_N^{[0]}}
\mathcal R_{a,N}
\frac{
\prod_{i=1}^{p}
(\alpha_{a+i}+\beta_{s_N}+d_N+1)_{n_{N,i}}
}{
\prod_{i=1}^{p}
(\alpha_{a-1+i}+\beta_{s_{N+1}}+d_{N+1}+1)_{n_{N+1,i}}
}.
\end{equation}
where \(\mathcal R_{a,N}\) is the quotient of the second fractions in
\eqref{eq:app-MP-shifted-leading-coefficient}. This last quotient is now
evaluated explicitly.

Advancing from \(N\) to \(N+1\) increases exactly the \(s_N\)-th entry of
the beta multi-index:
\[
m_{N+1,j}=m_{N,j}+\delta_{j,s_N}.
\]
At the same time, the active alpha exponent is unchanged by the compensating
cyclic shift:
\[
g_{a,N}
\coloneq\alpha_{a+i_N}+e_N
=\alpha_{a-1+i_{N+1}}+e_{N+1}.
\]
Consequently, the quotient of the products involving alpha--beta sums in
\(\mathcal R_{a,N}\) is
\begin{equation}
\label{eq:app-MP-lower-sum-cancellation}
\frac{1}{g_{a,N}+\beta_{s_N}+d_N+1}.
\end{equation}

It remains to evaluate the quotient of the signs, factorials, and
alpha-difference products. Put
\[
D_N^{[a]}
\coloneq
\prod_{\substack{1\le i\le p\\i\ne i_N}}
(\alpha_{a+i}-g_{a,N})_{n_{N,i}}.
\]
If \(i_N<p\), then \(i_{N+1}=i_N+1\) and \(e_{N+1}=e_N\). All common
Pochhammer factors cancel, and the only uncancelled ratio is
\[
\frac{D_{N+1}^{[a-1]}}{D_N^{[a]}}
=
\frac{(\alpha_a-g_{a,N})_{e_N+1}}
     {(\alpha_{a+p}-g_{a,N})_{e_N}}
=
\frac{(\alpha_a-g_{a,N})_{e_N+1}}
     {(\alpha_a+1-g_{a,N})_{e_N}}
=\alpha_a-g_{a,N}.
\]
The signs and factorials are unchanged in this case. If \(i_N=p\), then
\(i_{N+1}=1\), \(e_{N+1}=e_N+1\), and the two alpha-difference products
are identical. The sign changes once and the factorial quotient contributes
\(e_N+1=g_{a,N}-\alpha_a\). Thus both cases give the same factor
\[
-(g_{a,N}-\alpha_a).
\]
Combining it with \eqref{eq:app-MP-lower-sum-cancellation} yields
\begin{equation}
\label{eq:app-MP-lower-residual-factor}
\mathcal R_{a,N}
=
-
\frac{\alpha_{a+i_N}-\alpha_a+e_N}
{\alpha_{a+i_N}+e_N+\beta_{s_N}+d_N+1}.
\end{equation}

Finally, insert \eqref{eq:app-MP-lower-residual-factor} into
\eqref{eq:app-MP-lower-quotient-before-cancellation} and expand
\(E_{N+1}^{[0]}/E_N^{[0]}\) by
\eqref{eq:app-MP-beta-difference-product-expanded}. The result is exactly
\eqref{eq:MP-lower-Pochhammer-simplified}.
\end{proof}

The two propositions show that the factorization
\eqref{eq:MP-PBF} and all the factor entries used in the positivity and urn
arguments follow from the Cauchy construction contained in this paper.

\section*{Declarations}

\begin{itemize}
	\item \textbf{Competing interests.}
	The author declares that he has no competing interests.

	\item \textbf{Author contributions.}
	The author is solely responsible for the conception of the work, the
	development and verification of the mathematical results, and the writing
	of the manuscript.

	\item \textbf{Use of artificial intelligence.}
	OpenAI Codex was used to assist with adversarial review of arguments,
	exact symbolic verification of selected recurrence formulas and rational
	identities, checks of internal consistency and citations, and language
	editing. Every AI-assisted suggestion, mathematical step, and bibliographic
	item used in the manuscript was independently checked by the author against
	the derivations, exact computations, or cited source records. The author
	remains fully responsible for the content of the manuscript.

	\item \textbf{Data availability.}
	No datasets were generated or analysed during the current study.

	\item \textbf{Code availability.}
	Exact-arithmetic code was used to verify the finite Pochhammer identities
	and the rational mixed Pi\~neiro urn example. The verification scripts are
	available from the author.

	\item \textbf{Ethics approval.}
	Not applicable.

	\item \textbf{Consent to participate.}
	Not applicable.

	\item \textbf{Consent for publication.}
	Not applicable.
\end{itemize}
	
\section*{Acknowledgments}

The author is grateful to Manuel D. de la Iglesia for pointing out, during
a group meeting in May 2026, that a scalar finite-band stochastic matrix can
be reorganized as a finite-phase QBD, and for illustrating the construction
with bandwidths \((2,2)\) and \((2,3)\).

The author acknowledges support from the Spanish Agencia Estatal de
Investigación through research projects PID2021-122154NB-I00,
\emph{Ortogonalidad y Aproximación con Aplicaciones en Machine Learning y
Teoría de la Probabilidad}, and PID2024-155133NB-I00,
\emph{Ortogonalidad, aproximación e integrabilidad: aplicaciones en procesos
estocásticos clásicos y cuánticos}.

	\printbibliography
	
\end{document}